\documentclass[a4paper,12pt]{amsart}

 \usepackage{hyperref}
 \usepackage{microtype}
\usepackage{amssymb,amsmath,amsfonts,amsthm,bm,latexsym,amscd,verbatim,url,nicefrac,stmaryrd,enumerate,colonequals,dsfont,appendix}
\usepackage[all]{xy}

\usepackage{times}
\usepackage{graphicx}
\usepackage{fullpage}

{\bf}{\it}
\newtheorem*{thm*}{Theorem}
\newtheorem{thm}{Theorem}[section]{\bf}{\it}
\newtheorem{prop}[thm]{Proposition}
\newtheorem{lemma}[thm]{Lemma}
\newtheorem{cor}[thm]{Corollary}

\theoremstyle{definition}
\newtheorem{dfn}[thm]{Definition}
\theoremstyle{remark}
\newtheorem{rmk}[thm]{Remark}
\theoremstyle{remark}
\newtheorem{exm}[thm]{Example}
\newtheorem{assu}[thm]{Assumption}
\newtheorem{cons}[thm]{Construction}

\newcommand{\A}{\mathbb{A}}
\newcommand{\B}{\mathbb{B}}

\newcommand{\F}{\mathbb{F}}

\newcommand{\HH}{\mathbb{H}}
\newcommand{\LL}{\mathbb{L}}
\newcommand{\N}{\mathbb{N}}
\newcommand{\Q}{\mathbb{Q}}
\newcommand{\RR}{\mathbb{R}}
\newcommand{\R}{\mathbb{R}}
\newcommand{\T}{\mathbb{T}}
\newcommand{\Z}{\mathbb{Z}}

\newcommand{\cat}{\mathbf{C}}
\newcommand{\catD}{\mathbf{D}}

\newcommand{\catT}{\mathbf{T}}

\newcommand{\ra}{\rightarrow}

\newcommand{\mcF}{\mathcal{F}}

\newcommand{\mcI}{\mathcal{I}}

\newcommand{\mcO}{\mathcal{O}}
\newcommand{\mcU}{\mathcal{U}}
\newcommand{\mcV}{\mathcal{V}}

\newcommand{\mfF}{\mathfrak{F}}
\newcommand{\mfG}{\mathfrak{G}}
\newcommand{\mfM}{\mathfrak{M}}

\newcommand{\mfX}{\mathfrak{X}}

\newcommand{\del}{\partial}
\newcommand{\adj}[4]{#1\negmedspace: #2\rightleftarrows #3:\negmedspace #4}

\DeclareMathOperator{\car}{char}

\DeclareMathOperator{\ct}{ct}

\DeclareMathOperator{\eff}{eff}
\DeclareMathOperator{\et}{\acute{e}t}

\DeclareMathOperator{\Ev}{Ev}

\DeclareMathOperator{\FormDA}{\bf{FormDA}}

\DeclareMathOperator{\Frac}{Frac}
\DeclareMathOperator{\Frob}{Frob}
\DeclareMathOperator{\Frobet}{Frob\acute{e}t}

\DeclareMathOperator{\Gal}{Gal}
\DeclareMathOperator{\gc}{gc}

\DeclareMathOperator{\hocolim}{hocolim}
\DeclareMathOperator{\Hom}{Hom}

\DeclareMathOperator{\id}{id}

\DeclareMathOperator{\Nis}{Nis}

\DeclareMathOperator{\Perf}{Perf}
\DeclareMathOperator{\PerfSm}{PerfSm}

\DeclareMathOperator{\quiet}{quiet}

\DeclareMathOperator{\RigCor}{RigCor}

\DeclareMathOperator{\RigSm}{RigSm}
\DeclareMathOperator{\RHom}{Hom_\bullet}

\DeclareMathOperator{\Sing}{Sing}
\DeclareMathOperator{\Sm}{Sm}

\DeclareMathOperator{\Spa}{Spa}
\DeclareMathOperator{\Spec}{Spec}
\DeclareMathOperator{\Spect}{Spt}
\DeclareMathOperator{\Spf}{Spf}
\DeclareMathOperator{\st}{st}
\DeclareMathOperator{\Sus}{Sus}

\DeclareMathOperator{\tr}{tr}
\DeclareMathOperator{\uhom}{\underline{Hom}}

\DeclareMathOperator{\wRig}{\widehat{Rig}}

\DeclareMathOperator{\Ch}{\bf{Ch}}

\DeclareMathOperator{\DA}{\bf{DA}}
\DeclareMathOperator{\DM}{\bf{DM}}

\DeclareMathOperator{\FSH}{\bf{FSH}}

\DeclareMathOperator{\Mod}{\bf{-Mod}}

\DeclareMathOperator{\PerfDA}{\bf{PerfDA}}

\DeclareMathOperator{\RigDA}{\bf{RigDA}}
\DeclareMathOperator{\RigDM}{\bf{RigDM}}

\DeclareMathOperator{\Psh}{\bf{Psh}}
\DeclareMathOperator{\PST}{\bf{PST}}

\DeclareMathOperator{\Sh}{\bf{Sh}}
\DeclareMathOperator{\SSpect}{\bf{Spt}}
\DeclareMathOperator{\sPsh}{\bf{sPsh}}

\DeclareMathOperator{\wRRig}{\widehat{\bf{Rig}}}
\newcommand{\wRigDA}{\wRRig\!\DA}
\newcommand{\wRigSm}{\wRig\!\Sm}

\begin{document}
	\title{A motivic version of the theorem of Fontaine and Wintenberger} 
	\author{Alberto Vezzani}
	\address{LAGA - Universit\'e Paris 13, Sorbonne Paris Cit\'e, 
		99 av. Jean-Baptiste Cl\'ement, 93430 Villetaneuse, France}
	
	\email{vezzani@math.univ-paris13.fr}

	\begin{abstract}
		We establish a tilting equivalence for rational, homotopy-invariant cohomology theories defined over non-archimedean analytic varieties. More precisely, we prove an equivalence between the categories of motives of rigid analytic varieties over a perfectoid field $K$ of mixed characteristic and over the associated (tilted) perfectoid field $K^{\flat}$ of equal characteristic. This can be considered as a motivic generalization of a theorem of Fontaine and Wintenberger, claiming that the Galois groups of $K$ and $K^\flat$ are isomorphic. %
	\end{abstract}

\maketitle

\tableofcontents


\section*{Introduction}
A theorem of Fontaine and Wintenberger \cite{fw}, later expanded by Scholze \cite{scholze}, states that there is an isomorphism between the Galois groups of a perfectoid  field $K$ and the associated (tilted) perfect field $K^\flat$ of positive characteristic. The standard example of such a pair is formed by the completions of the fields  $\Q_p(p^{1/p^{\infty}})$ and $\F_p(\!(t)\!)(t^{1/p^\infty})$. This theorem is a cornerstone of $p$-adic Hodge theory, providing a striking equivalence between  objects in mixed characteristic (finite \'etale algebras over $K$) and objects in equal characteristic $p$ (finite \'etale algebras over $K^\flat$).

In the same spirit, the ``tilting equivalence" of Scholze \cite{scholze} promotes the equivalence above to a certain kind of spaces of higher dimension, namely  the \emph{perfectoid spaces} over the two fields.  These objects are non-noetherian  analytic spaces, which can typically be thought as certain infinite covers of classical rigid analytic varieties obtained by adding all $p$-th roots of local coordinates, and their introduction has had many different applications in arithmetic geometry (see \cite{scholze-icm}).

The aim of this paper is to ``descend" Scholze's tilting equivalence to the level of classical rigid analytic varieties defined over $K$ and $K^\flat$, at least from a (co-)homological point of view. What we actually prove, is an equivalence between  the two associated categories of  (mixed, effective, with rational coefficients) rigid analytic motives $\RigDM$  defined over the two fields, introduced by Ayoub \cite{ayoub-rig} adapting the classic construction of $\DM$ by Voevodsky \cite{voe-h}:
\begin{thm*}[\ref{main}]
	Let $K$ be a perfectoid field with tilt $K^\flat$ and let $\Lambda$ be a $\Q$-algebra. There is a monoidal triangulated equivalence of categories
	\[{\mfF}\colon{\RigDM_{\et}^{\eff}(K^\flat,\Lambda)}\stackrel{\sim}{\ra}{\RigDM_{\et}^{\eff}(K,\Lambda)}.\]
\end{thm*}
We remark that, from a motivic point of view, the theorem of Fontaine and Wintenberger can  be rephrased by saying that the categories of \emph{Artin motives} over the two fields are equivalent. Conversely, our theorem shows that not only the absolute Galois groups of a perfectoid field and its tilt are isomorphic, but also their \emph{absolute local motivic Galois groups} (in the sense of the introduction of \cite{ayoub-rig}).
By \cite{vezz-DADM}, it is also possible to state the result above with respect to motives \emph{without transfers} $\RigDA$ defined over the two fields (the rigid analytic analogue of the category $\DA$ see \cite{ayoub-rig}).

The statement above involves only  rigid analytic varieties and its proof uses Scholze's theory of perfectoid spaces only in an auxiliary way. Nonetheless, we can restate  our main result  highlighting the role of perfectoid spaces as follows:
\begin{thm*}[\ref{mainperf}]
	Let $K$ be a perfectoid field 
	and let $\Lambda$ be a $\Q$-algebra. There is a monoidal triangulated equivalence of categories
	\[{\RigDM_{\et}^{\eff}(K,\Lambda)}\stackrel{\sim}{\ra}{\PerfDA_{\et}^{\eff}(K,\Lambda)}\]
\end{thm*}
The category $\PerfDA^{\eff}_{\et}(K,\Lambda)$ is built in the same way as $\RigDA^{\eff}_{\et}(K,\Lambda)$ using as a starting point the big  \'etale site of smooth, small perfectoid spaces, i.e. those which are locally \'etale over some perfectoid ball $\widehat{\B}^n$.

This last theorem provides a way to ``perfectoidify'' (and ``de-perfectoidify'') canonically a rigid analytic variety into a perfectoid motive (and vice-versa). The first thorem gives  a way to ``tilt'' (or ``untilt'') canonically also a rigid analytic varieties over $K$ into a rigid analytic motive defined on the  field $K^\flat$ (and vice-versa). 

\vspace{0.3cm}

We now make a rough sketch on how the motivic ``untilting'' and ``de-perfectoidification'' procedures work. Start from a smooth rigid variety $X$ over $K^\flat$ and  associate to it a perfectoid space $\widehat{X}$ obtained by taking the perfection of $X$. This operation can be performed canonically since $K^\flat$ has positive characteristic. 
 We then use Scholze's theorem to tilt $\widehat{X}$ obtaining a perfectoid space $\widehat{Y}$ in mixed characteristic. Suppose now that $\widehat{Y}$ is the limit of a tower of rigid analytic varieties 
 \[
 \ldots \ra Y_{h+1}\ra Y_{h}\ra\ldots\ra Y_1\ra Y_0
 \]
 such that $Y_0$ is \'etale over the Tate ball $\B^n=\Spa K\langle \upsilon_1,\ldots,\upsilon_n\rangle$ and each $Y_h$ is obtained as the pullback of $Y_0$ by the map $\B^n\ra\B^n$ defined by taking the $p^h$-powers of the coordinates $\upsilon_i\mapsto\upsilon_i^{p^h}$. 
 Under such hypotheses (we will actually need slightly stronger conditions on the tower  above) we 
 then ``de-perfectoidify'' $\widehat{Y}$ by associating to it  $Y_{\bar{h}}$ for a sufficiently big index $\bar{h}$. 
 
 The main technical problem of this construction is to show that it is independent of the choice of the tower, and on the index $\bar{h}$. It is also by definition a local procedure, which is not canonically extendable to arbitrary varieties by gluing. 
 In order to overcome these obstacles, we use in a crucial way some techniques of approximating maps between spaces up to homotopy, which are obtained by a generalization of the implicit function theorem in the non-archimedean setting. 
 
We point out that many deep motivic theorems are widely used throughout the article, and they are crucial to prove that the recipe sketched above gives rise to an equivalence (see the proof of \ref{mainperf}) and not only for the construction of explicit homotopies. Among them, the Cancellation Theorem \cite{voe-canc} and the dualizability of compact motives \cite{riou-dual}. These results admit a rigid analytic version, proved by Ayoub in \cite{ayoub-rig}. We also use the parallel between $\DA$ and $\DM$ and in order to do so, we introduce a finer topology in characteristic $p$  (the $\Frobet$ topology, also called $\quiet$ in \cite{fj}) giving rise to transfers by means of a localization procedure (see \cite{vezz-DADM}).%

The following diagram of categories of motives summarizes the situation. The equivalence in the bottom line follows easily from the ``tilting equivalence'' of Scholze. 
 All notations introduced in the theorems and in the diagram will be described in later sections. 
 $$\xymatrix{
 &\RigDM^{\eff}_{\et}(K,\Lambda)&&\RigDM^{\eff}_{\et}(K^{\flat },\Lambda)\ar[ll]_{\mfF}^\sim\\
 &\RigDA^{\eff}_{\et}(K,\Lambda)\ar[u]_{\sim}\ar@<0.5ex>[dl]^{\LL \iota^*}&&\RigDA^{\eff}_{\Frobet}(K^\flat,\Lambda)\ar[u]^{\sim}\ar[dd]_{\sim}^{\LL\Perf^*}\\
 {\wRigDA^{\eff}_{\et,\widehat{\B}^1}}(K,\Lambda)\ar@<0.5ex>[ur]^{\LL \iota_!}\\
 &\PerfDA^{\eff}_{\et}(K,\Lambda)\ar[uu]_\sim\ar[ul]_{\LL j^*}\ar@{<->}[rr]^{\sim}&&\PerfDA^{\eff}_{\et} (K^\flat,\Lambda).
 }$$

The main result also extends to the categories of non-effective motives, as follows:
\begin{thm*}[\ref{mainstable}]
 Let $K$ be a perfectoid field with tilt $K^\flat$ and let $\Lambda$ be a $\Q$-algebra. There is a monoidal triangulated equivalence of categories
 \[{\mfF}^{\st}\colon{\RigDM_{\et}(K^\flat,\Lambda)}\stackrel{\sim}{\ra}{\RigDM_{\et}(K,\Lambda)}.\]
\end{thm*}
Due to the intricate construction of the functor $\mfF$, the theorem above is not immediately obtained from the effective case and requires a little extra inspection of the categories of spectra.

\vspace{0.3cm}

As for concrete applications, our theorem  allows to ``(un-)tilt'' and ``\mbox{(de-)}perfectoidify''  any motivic  cohomology theory on rigid or perfectoid spaces satisfying \'etale descent and homotopy invariance, having coefficients over $\Q$ (it is expected to extend this over $\Z[1/p]$ as well, see Remark \ref{finalrmk}). Such an example is  the overconvergent de Rham cohomology for rigid analytic varieties over $K$. It  gives rise to new cohomology theories ``\`a la de Rham'' for (small) perfectoid spaces as well as for rigid varieties over local fields of equi-characteristic $p$, satisfying descent, homotopy invariance, finite-dimensionality  and further formal properties (see \cite{vezz-rigidreal}) which are compatible with rigid cohomology \cite{vezz-tilt4rigid}. The relation with $p$-adic periods is an object of future research. Also, we underline that such realization functors might finally give rise (thanks to the derived tannakian formalism developed in \cite{ayoub-h1}) to the {absolute local motivic Galois  group }of $K$ for any valued field $K$, the case of equi-characteristic zero being described in the introduction of \cite{ayoub-rig}. %

It is possible to use our main theorem  to answer positively to a conjecture of Ayoub \cite[Conjecture 2.5.72]{ayoub-rig} giving an explicit description of rigid analytic motives of good reduction over a local field $K$ as well as morphisms between compact rigid analytic motives over $K$ in terms of motives and morphisms defined over the residue field, extending these results from the equi-characteristic case to the mixed characteristic case (see \cite{vezz-tilt4rigid}).

We also point out that the motivic tilting equivalence can be used to ``replace'' the base field $K$ with another one $K'$ having the same tilt. A classic example is formed by the completions $K$ and $K'$ of the fields $\Q_p(p^{1/p^\infty})$ and $\Q_p(\mu_{p^\infty})$, respectively. Similar techniques are exploited by Kedlaya-Liu in \cite{ked-liu2}.

\vspace{0.3cm}

The article is organized as follows. 
In Section \ref{pre} we recall the basic definitions and the language of adic spaces while in Section \ref{wrigsm} we define the environment in which we will perform our construction, namely the category of semi-perfectoid spaces $\wRigSm$ and we define the \'etale topology on it. 
In Section \ref{motives}  we  define the  categories of motives for $\RigSm$, $\wRigSm$ and $\PerfSm$  adapting the constructions of Voevodsky and Ayoub. Thanks to the general model-categorical tools introduced in this section, we give in Section \ref{motapprox} a motivic interpretation of some approximation results of maps valid for non-archimedean Banach algebras. 
In Sections \ref{deperf0} and \ref{deperfp} we prove the existence of the de-perfectoidification functor $\LL\iota_!$ from perfectoid motives to rigid motives in zero and positive characteristics, respectively. Finally, we give in Section \ref{mainthm} the proof of our main result, both in its effective and stable form.

In the appendix, we collect some technical theorems that are used in our proof. Specifically, we first present a generalization of the implicit function theorem in the rigid setting, and conclude a  result about the approximation of maps modulo homotopy as well as its geometric counterpart. 
 We also prove the existence of compatible approximations of a collection of maps $\{f_1,\ldots,f_N\}$ from a variety in $\wRigSm$ of the form $X\times\B^n$ to a rigid variety $Y$ such that the compatibility conditions among the maps $f_i$ on the faces $X\times\B^{n-1}$ are preserved.
  This fact has 
  important consequences for computing maps to $\B^1$-local complexes of presheaves in the motivic setting. 

\section*{Acknowledgments}
This paper is part of my PhD thesis, 
carried out in a  co-tutelle program between the University of Milan and the University of Zurich. I am incredibly indebted to my advisor  Joseph Ayoub for his suggestion of pursuing the present project, his constant and generous guidance, the outstanding amount of insight that he kindly shared with me, and his endless patience. I wish to express my gratitude to my co-advisor  Luca Barbieri Viale, for his invaluable encouragement and the numerous mathematical discussions throughout the development of my thesis. I am grateful to  Peter Scholze for answering 
 many questions and also to 
   Yves Andr\'e, Fabrizio Andreatta and an anonymous referee for their  remarks and suggestions. 
 I had 
   the opportunity to give  a detailed series of talks of this work at the University of Zurich: I kindly thank  Andrew Kresch for this chance and the various discussions on this topic. 
 For the plentiful remarks   on preliminary versions of this work and his precious and friendly support, I also wish to thank Simon Pepin Lehalleur.

\section{Generalities on adic spaces}\label{pre}

We start by recalling the language of adic spaces, as introduced by Huber  \cite{huber1}, \cite{huber2} and generalized by Scholze-Weinstein \cite{sw}.  
We will always work with adic spaces over a non-archimedean valued field $K$ in the following sense. 

\begin{dfn}
A \emph{non-archimedean  field} is a topological field $K$ whose topology is induced by a non-trivial valuation of rank one. The associated norm is a multiplicative map that we denote by $|\cdot|\colon K\ra\R_{\geq0}$ and its valuation ring is denoted by $K^\circ$.  

\end{dfn}

 From now on, we fix a non-archimedean field $K$ and we  pick
  a non-zero element $\pi\in K$ with $|\pi|<1$.

\begin{dfn}A \emph{ Tate $K$-algebra} is a topological $K$-algebra $R$ for which there exists a subring $R_0$ such that the set $\{\pi^k R_0\}$ forms a basis of neighborhoods of $0$. A subring $R_0$ with the above property is called a \emph{ring of definition}. 
\end{dfn}

\begin{dfn}
Let $R$ be a Tate $K$-algebra. 
\begin{itemize}
\item A subset $S$ of $R$ is \emph{bounded} if it is contained in the set $\pi^{-N}R_0$ for some integer $N$. An element $x$ of $R$ is \emph{power-bounded} if the set $\{x^{n}\}_{n\in\N}$ is bounded. The set of power-bounded elements is a subring of $R$ that we denote by $R^\circ$. 
\item An element $x$ of $R$ is \emph{topologically nilpotent }if $\lim_{n\ra+\infty}x^n=0$. The set of topologically nilpotent elements is an ideal of $R^\circ$ that we denote by $R^{\circ\circ}$. 
 \end{itemize}
\end{dfn}

\begin{dfn}
 An \emph{affinoid $K$-algebra} is a pair $(R,R^+)$ where $R$ is a Tate $K$-algebra and $R^+$ is an open and integrally closed $K^\circ$-subalgebra of $R^\circ$. A morphism $(R,R^+)\ra(S,S^+)$ of  affinoid $K$-algebras is a pair of compatible $K^\circ$-linear continuous maps of rings $(f,f^+)$. 
  An affinoid $K$-algebra $(R,R^+)$ is called \emph{complete} if $R$ (and hence also $R^+$) is complete. 
 \end{dfn}

 \begin{rmk}
By \cite[Proposition 5.30]{wedhorn} if $R$ is a Tate $K$-algebra, then $(R,R^\circ)$ is an affinoid $K$-algebra.
 \end{rmk}

  \begin{dfn}For any complete Tate $K$-algebra $R$ we denote by $R\langle \upsilon _1,\ldots, \upsilon _n\rangle$ the Banach algebra of strictly convergent power series in $R[[\upsilon_1,\ldots,\upsilon_n]]$ endowed with the sup-norm (see  \cite[Section 1.4.1]{BGR}). 
   A \emph{topologically of finite type Tate algebra }(or simply \emph{tft Tate algebra}) is a Banach $K$-algebra $R$ isomorphic to a quotient  of the normed $K$-algebra $K\langle\upsilon _1,\ldots, \upsilon _n\rangle$ for some $n$. 
  \end{dfn}
If $R$ is a tft Tate algebra, the pair $(R,R^\circ)$ is an affinoid $K$-algebra, and $R^\circ$ is a ring of definition  whenever $R$ is reduced (see \cite[Theorem 6.2.4/1]{BGR}). 

  We now recall the definition of perfectoid pairs, introduced in \cite{scholze}.
  
    \begin{dfn}
     A  \emph{perfectoid field} $K$ is a complete non-archimedean field whose rank one valuation is non-discrete, whose residue characteristic is $p$ and  such that the Frobenius is surjective on $K^\circ/p$. In case $\car K=p$ this last condition amounts to saying that $K$ is perfect. 
    \end{dfn}
  
  \begin{dfn}Let $K$ be a perfectoid field.
  \begin{itemize}
  \item 
    A \emph{perfectoid algebra} is a Banach $K$-algebra $R$ such that $R^\circ$ is bounded and the Frobenius map is surjective on $R^\circ/p$.
    \item  A \emph{perfectoid affinoid $K$-algebra} is an affinoid $K$-algebra $(R,R^+)$ over a perfectoid field $K$ such that $R$ is perfectoid.
  \end{itemize}
  \end{dfn}
  
 \begin{rmk}
   If $R$ is a perfectoid algebra, then $(R,R^\circ)$ is a perfectoid affinoid $K$-algebra.
  \end{rmk}
  
    \begin{exm}
    Suppose that $K$ is a perfectoid field. 
    A basic example of a perfectoid algebra is the following: let $\underline{\upsilon}=(\upsilon_1,\ldots,\upsilon_N)$ be a $N$-tuple of coordinates  and  $K^\circ[\underline{\upsilon}^{1/p^\infty}]$ be the ring $\varinjlim_h K^\circ[\underline{\upsilon}^{1/p^h}]$ endowed with the sup-norm induced by the norm on $K$. We also denote by $K^\circ\langle \underline{\upsilon}^{1/p^\infty}\rangle$ its $\pi$-adic completion. By \cite[Proposition 5.20]{scholze}, the ring $K^\circ\langle \underline{\upsilon}^{1/p^\infty}\rangle[\pi^{-1}]$ is a perfectoid $K$-algebra which we will denote by $K\langle \underline{\upsilon}^{1/p^\infty}\rangle$. The pair $(K\langle \underline{\upsilon}^{1/p^\infty}\rangle,K^\circ\langle \underline{\upsilon}^{1/p^\infty}\rangle)$ is a perfectoid affinoid $K$-algebra.  
     We also define in the same way the perfectoid affinoid $K$-algebra $(K\langle \underline{\upsilon}^{\pm1/p^\infty}\rangle,K^\circ\langle \underline{\upsilon}^{\pm1/p^\infty}\rangle)$ (see \cite[Example 4.4]{scholze-ph}). 
       \end{exm}
   
      \begin{rmk}\label{bigosum}
       $K\langle \underline{\upsilon}^{1/p^\infty}\rangle$ is isomorphic as a $K\langle \underline{\upsilon}\rangle$-topological module to the completion  $\widehat{\bigoplus} K\langle \underline{\upsilon}\rangle$ of the free module $\bigoplus K\langle \underline{\upsilon}\rangle$ with basis indexed by the set $I=(\Z[1/p]\cap[0,1))^{N}$. By \cite[Proposition 2.1.5/7]{BGR} there is an explicit description of this ring as the  subring of $\prod_{i\in I}  K\langle \underline{\upsilon}\rangle$ whose elements are those $(x_i)_{i\in I }$ such that for any $\varepsilon>0$ the inequality $||x_i||<\varepsilon$ holds for almost all $i$ (that is, for all $i$ except for a finite number of them).
      \end{rmk}
  
   The following theorem summarizes some results of Scholze, including the \emph{tilting equivalence} of perfectoid algebras which will play a crucial role in our construction.
    
    \begin{thm}[\cite{scholze}]\label{tilteq} Let $K$ be a perfectoid field. 
    \begin{enumerate}
    \item (\cite[Lemma 3.4]{scholze}) The multiplicative monoid $\varprojlim_{x\mapsto x^p} K$ can be given a structure $K^\flat$ of perfectoid field with the norm induced by the  multiplicative map $\sharp\colon K^\flat\ra K$.  The field $K^\flat$ has characteristic $p$ and  coincides with $K$ in case $\car K=p$. 
    \item (\cite[Theorem 3.7]{scholze}) The functor $L\mapsto L^\flat$ for $L$ finite \'etale over $K$ induces an isomorphism $\Gal(K)\cong\Gal(K^\flat)$.
    \item (\cite[Lemma 6.2]{scholze}) There is an equivalence of categories, the \emph{tilting equivalence}, from perfectoid affinoid $K$-algebras  to perfectoid affinoid $K^\flat$-algebras   denoted by $(R,R^+)\mapsto(R^\flat,R^{\flat+})$ such that $R^\flat$ is multiplicatively isomorphic to $\varprojlim_{x\mapsto x^p} R$ and  $R^{\flat+}$  is multiplicatively isomorphic to $\varprojlim_{x\mapsto x^p} R^+$.
    \item (\cite[Proposition 5.20 and Corollary 6.8]{scholze}) The tilting equivalence associates the perfectoid $K$-algebra 
      $(K\langle \underline{\upsilon}^{1/p^\infty}\rangle,K^\circ\langle \underline{\upsilon}^{1/p^\infty}\rangle)$ to $(K^\flat\langle \underline{\upsilon}^{1/p^\infty}\rangle,K^{\flat\circ}\langle \underline{\upsilon}^{1/p^\infty}\rangle)$ and the perfectoid $K$-algebra $(K\langle \underline{\upsilon}^{\pm1/p^\infty}\rangle,K^\circ\langle \underline{\upsilon}^{\pm1/p^\infty}\rangle)$ to $(K^\flat\langle \underline{\upsilon}^{\pm1/p^\infty}\rangle,K^{\flat\circ}\langle \underline{\upsilon}^{\pm1/p^\infty}\rangle)$. 
    \end{enumerate}
   \end{thm}

We now recall Huber's construction of the spectrum of a valuation ring (see \cite{huber2}). 

\begin{cons}
 Let $(R,R^+)$ be an affinoid $K$-algebra. The set $\Spa(R,R^+)$ is the set of equivalence classes of continuous multiplicative valuations $|\cdot|$ (see \cite[Section 3]{huber1} and \cite[Definitions 2.2, 2.5]{scholze}). 
 It is endowed with the topology generated by the basis of \emph{rational subsets} $\{U(f_1,\ldots,f_n\mid g)\}$ by letting $f_1,\ldots,f_n,g$ vary among elements in $R$ such that $f_1,\ldots,f_n$ generate $R$ as an ideal and where the set $U(f_1,\ldots,f_n\mid g)$ is the set of those valuations $|\cdot|$ satisfying $|f_i|\leq|g|$ for all $i$.

If we define maps of valuation fields over $K$ (see \cite[Definition 2.26]{scholze}) as maps of affinoid $K$-algebras $(L,L^+)\ra(L',L'^+)$   such that $L'^+\cap L=L^+$ then $\Spa(R,R^+)$ has the following alternative description: it is the set $\varinjlim\Hom((R,R^+),(L,L^+))$ by letting $(L,L^+)$ vary in the category of valuation fields over $K$ (see \cite[Proposition 2.27]{scholze}). Its topology can be defined by declaring the sets $\{\phi\colon  0\neq|\phi(f)|\leq |\phi(g)|\}$ to be open, for all pairs of elements $f,g$ in $R$.

We can associate to a rational subset $U(f_1,\ldots,f_n\mid g)$ the affinoid $K$-algebra the affinoid $K$-algebra \[(\mcO(U),\mcO^+(U))=(R\langle f_1/g,\ldots,f_n/g\rangle,R\langle f_1/g,\ldots,f_n/g\rangle^+)\] defined in \cite[Section 1]{huber2} and \cite[Defintion 2.13]{scholze}. 
 This way, we  define a presheaf of affinoid $K$-algebras $(\mcO_X,\mcO^+_X)$ on a basis of  $X=\Spa(R,R^+)$.  
 
 By \cite[Lemma 1.5, Proposition 1.6]{huber2} for any $x\in X=\Spa(R,R^+)$ the valuation at $x$ extends to a valuation on $\mcO_{X,x}$ and the stalk $\mcO^+_{X,x}$ is local and corresponds to $\{f\in\mcO_{X,x}\colon|f(x)|\leq1\}$. 
\end{cons}

By \cite[Proposition 1.6]{huber2} there  holds $\mcO^+(U)=\{f\in\mcO(U)\colon |f(x)|\leq1\text{ for all }x\in U\}$ for any rational subset $U$ of $\Spa(R,R^+)$ so that $\mcO^+$ is a sheaf if and only if $\mcO$ is a sheaf.

Sadly enough, $\mcO$, $\mcO^+$ are not sheaves in general as shown at the end of \cite[Section 1]{huber2}. %
By Tate's acyclicity theorem \cite[Theorem 8.2.1/1]{BGR} and Scholze's acyclicity theorem \cite[Theorem 6.3]{scholze}, if $(R,R^+)$ is a tft Tate algebra or a perfectoid affinoid $K$-algebra, then $\mcO$, $\mcO^+$ are sheaves (see \cite[Section 2]{huber2}). Also, if $\mcO^+(U)$ is bounded for all rational subspaces $U$ of $X$ (in which case we say that $X$ is \emph{stably uniform}) then $\mcO$, $\mcO^+$ are sheaves (see \cite{buzz-ver}).

\begin{dfn}
Objects of $\mcV$ are triples $(X,\mcO_X ,\{|\cdot|_x\}_{x\in X})$ where $X$ is a topological space, $\mcO_X$ is a sheaf of complete topological $K$-algebras, and $|\cdot|_x$ is a continuous
valuation on $\mcO_{X,x}$. Maps are morphisms of ringed spaces which induce continuous $K$-algebra morphisms of sheaves  and are compatible with the valuations on stalks.
\end{dfn}

\begin{rmk}
By abuse of notation, whenever $R$ is a tft Tate algebra we  sometimes denote by $\Spa R$ the object $\Spa(R,R^\circ)$ of $\mcV$.
\end{rmk}

\begin{dfn}Let $X$ be an object of $\mcV$.
\begin{itemize}
\item We say that $X$ is an \emph{ affinoid adic space} if it is isomorphic to $\Spa(R,R^+)$ for some  affinoid $K$-algebra $(R,R^+)$.  
\item  We say that $X$ is an \emph{ affinoid rigid variety} if it is isomorphic to $\Spa(R,R^\circ)$ for some  tft Tate algebra $R$. 
\item We say that $X$ is a \emph{perfectoid affinoid space} if it is isomorphic to $\Spa(R,R^+)$ for some perfectoid affinoid $K$-algebra $(R,R^+)$.
\item We say that $X$ is an  \emph{adic space} if it is locally isomorphic to an affinoid adic space. 
\item  We say that $X$ is a \emph{ rigid variety} if it is 
 locally isomorphic to an affinoid rigid variety. 
 \item We say that $X$ is a \emph{perfectoid space} if it  is  locally isomorphic to a perfectoid affinoid space. 
\end{itemize}
\end{dfn}

There is an apparent clash of definitions between  rigid varieties as presented above, and as defined by Tate. In fact, the two categories are canonically isomorphic. 
We refer to \cite[Section 4]{huber2} and \cite[Section 2]{scholze} for a more detailed collection of results on the comparison between these theories.

\begin{rmk}\label{Spaadj2}
By \cite[Proposition 2.1(ii)]{huber2} if $X$ is an adic space and $Y=\Spa(R,R^+)$ is an affinoid adic space  then $\Hom(X,Y)\cong\Hom(({R},{R}^+),(\mcO_X(X),\mcO_X^+(X))$. Moreover, as shown in \cite[Section 3]{huber1} and \cite[Section 2]{huber2}, if $(\widehat{R},\widehat{R}^+)$ is the completion of $({R},{R}^+) $ then it is a affinoid $K$-algebra and  $\Spa(\widehat{R},\widehat{R}^+)\cong\Spa(R,R^+)$.
\end{rmk}

\begin{assu}\label{assu}
  From now on,  we will always assume that $K$ is a perfectoid field. We also make the extra assumption that the invertible element $\pi$ of $K$ satisfies $|p|\leq|\pi|<1$ and coincides with $(\pi^\flat)^\sharp$ for a chosen $\pi^\flat$ in $K^\flat$. In particular, $\pi$ is equipped with a compatible system of $p$-power roots $\pi^{1/p^h}$ (see \cite[Remark 3.5]{scholze}).
\end{assu}

We now consider some basic examples and fix some notation. 

\begin{exm}
Let $\underline{\upsilon}=(\upsilon_1,\ldots,\upsilon_N)$ be a $N$-tuple of coordinates. 
The Tate $N$-ball $\Spa(K\langle\underline{\upsilon}\rangle,K^\circ\langle\underline{\upsilon} \rangle) $  
will be denoted by $\B^N$ and the $N$-torus $\Spa(K\langle\underline{\upsilon}^{\pm1}\rangle,K^\circ\langle\underline{\upsilon}^{\pm1} \rangle)$ 
 by $\T^N$. It is the rational subspace $U(1\mid \upsilon_1\ldots\upsilon_N)$ of $\B^N$. The map of spaces induced by the inclusion  
 $(K\langle \underline{\upsilon}\rangle,K^\circ\langle \underline{\upsilon}\rangle)\ra(K\langle \underline{\upsilon}^{1/p^h}\rangle, K^\circ\langle \underline{\upsilon}^{1/p^h}\rangle)$ 
    will be denoted by $\B^N\langle \underline{\upsilon}^{1/p^h}\rangle\ra \B^N$.  We use the analogous notation  $\T^N\langle \underline{\upsilon}^{1/p^h}\rangle\ra\T^N$ for the torus. These maps are clearly isomorphic to the endomorphism of $\B^N$ resp. $\T^N$ induced by $\upsilon_i\mapsto \upsilon_i^{p^h}$.

The space defined by the perfectoid affinoid $K$-algebra $(K\langle \underline{\upsilon}^{1/p^\infty}\rangle,K^\circ\langle \underline{\upsilon}^{1/p^\infty}\rangle)$ will be denoted  by  $\widehat{\B}^N$ and referred to as the \emph{perfectoid $N$-ball}. The space defined by the perfectoid affinoid $K$-algebra  $(K\langle \underline{\upsilon}^{\pm1/p^\infty}\rangle,K^\circ\langle \underline{\upsilon}^{\pm1/p^\infty}\rangle)$ coincides with the rational subspace $U(1\mid \upsilon_1\ldots\upsilon_N)$ of $\widehat{\B}^{N}$ and will be denoted by $\widehat{\T}^{N}$ and will be referred to as the \emph{perfectoid $N$-torus}.
\end{exm}

We now recall the definition of \'etale maps on the category of adic spaces, taken from \cite[Section 7]{scholze}.

\begin{dfn}
A map of   affinoid adic spaces $f\colon\Spa(S,S^+)\ra\Spa(R,R^+)$ is \emph{finite \'etale} if the associated map $R\ra S$ is a finite \'etale map of rings, and if $S^+$ is the integral closure of $R^+$ in $S$. A map of   adic spaces $f\colon X\ra Y$ is \emph{\'etale} if for any point $x\in X$ there exists an open neighborhood $U$ of $x$ and an  affinoid open subset $V$ of $Y$ containing $f(U)$ such that $f|_U\colon U\ra V$ factors as an open embedding $U\ra W$ and a finite \'etale map $W\ra V$ for some  affinoid adic space $W$.
\end{dfn}

The previous definitions, when restricted to the case of tft Tate varieties, coincide with the usual ones, as proved in \cite[Proposition 8.1.2]{fvdp}.

\begin{rmk}
Suppose we are given a diagram of affinoid $K$-algebras
$$\xymatrix{
(R,R^+)\ar[r]\ar[d]&(S,S^+)\\
(T,T^+)
}$$
In general, it is not possible to define a push-out in the category of affinoid $K$-algebras. Nonetheless, this can be performed under some hypothesis. For example, if the affinoid $K$-algebras are tft Tate algebras then the push-out exists and it is the tft Tate algebra associated to the completion $S\widehat{\otimes}_RT$ of $S\otimes_RT$ endowed with the norm of the tensor product (see \cite[Section 3.1.1]{BGR}). In case the affinoid $K$-algebras are perfectoid affinoid, then the push-out exists and is also perfectoid affinoid. It coincides with the completion of $(L,L^+)$ where $L$ is the ring $S{\otimes}_RT$  endowed with the norm of the tensor product and $L^+$ is the integral closure of $S^+\otimes_{R^+}T^+$ in $L$ (see \cite[Proposition 6.18]{scholze}). The same construction holds in case the map $(R,R^+)\ra(S,S^+)$ is finite \'etale and $(T,T^+)$ is a perfectoid affinoid (see \cite[Lemma 7.3]{scholze}). 
By Remark \ref{Spaadj2}, 
 the constructions above give rise to fiber products in the category of adic spaces.
\end{rmk}

\section{Semi-perfectoid spaces}\label{wrigsm}

We can now introduce  a convenient generalization of both smooth rigid varieties and smooth perfectoid spaces. We recall that our base field $K$ is a perfectoid field (see Assumption \ref{assu}).

\begin{prop}\label{fibprod}
Let $\underline{\upsilon}=\upsilon_1,\ldots,\upsilon_N$ and $\underline{\nu}=\nu_1,\ldots,\nu_M$ be two systems of coordinates. Let $(R_0,R_0^\circ)$ be a tft Tate algebra and let $f\colon\Spa (R_0,R_0^\circ)\ra\T^N\times\T^M=\Spa K\langle \underline{\upsilon}^{\pm1},\underline{\nu}^{\pm1}\rangle$ be a map which is a composition of finite \'etale maps and rational embeddings. Let also $\Spa(R_h,R_h^\circ)$ be the affinoid rigid variety $\Spa(R_0,R_0^\circ)\times_{\T^N}\T^N\langle \underline{\upsilon}^{1/p^h}\rangle$. The $\pi$-adic completion $(T,T^+)$ of $(\varinjlim_hR_h,\varinjlim_hR_h^\circ)$ represents the fiber product $\Spa(R_0,R_0^\circ)\times_{\T^N}\widehat{\T}^{N}$ and defines a bounded affinoid adic space. Moreover, $(T,T^+)$ is isomorphic to the completion of $(L,L^+)$ where $L$ is the ring  $R_0{\otimes}_{K\langle \underline{\upsilon}\rangle}K\langle \underline{\upsilon}^{1/p^\infty}\rangle$ endowed with the  norm of the tensor product and $L^+$ is the  integral closure  of $R_0^\circ$ in $L$.
\end{prop}

\begin{proof}
Let $(T,T^+)$ be as in the last claim. We let $W'$ be the fiber product  of $\Spa(T,T^+)$ and $\widehat{\T}^{N}\times\widehat{\T}^{M}$ over ${\widehat{\T}^N\times\T^M}$. If $\car K=0$ by \cite[Lemma 4.5(i)]{scholze-ph} it exists and is affinoid perfectoid,   represented by a perfectoid pair $(T',T'^+)$. %
The same is true if $\car K=p$ as in this case it coincides with the completed perfection of $X_0$ (indeed, the pullback of an \'etale map over  Frobenius  is isomorphic to  Frobenius, see for example \cite[Lemma 0EBS]{stacks-project}). 

From the strict inclusion $(T,T^+)\ra(T',T'^{\circ})$ we deduce that $T^+$ is bounded since $T'^+$ is, being $T'$ perfectoid. By considering rational subspaces of $\Spa(T,T^+)$ we deduce that this space is stably uniform, hence sheafy (by \cite{buzz-ver}).

The proof of the alternative description of $(T,T^+)$ follows in the same way as  \cite[Lemma 4.5(i)]{scholze-ph}.
\end{proof}

\begin{rmk}
The statement of the previous proposition is an instance when the  second term $T^+$ of some affinoid $K$-algebra $(T,T^+)$ may\emph{ not }be equal to the ring $T^\circ$.
\end{rmk}

\begin{cor}
Let $X$ be a reduced rigid variety with an \'etale map \[f\colon X\ra\T^N\times\T^M=\Spa K\langle \underline{\upsilon}^{\pm1},\underline{\nu}^{\pm1}\rangle.\] Then the fiber product $X\times_{\T^N}\widehat{\T}^{N}$ exists in the category of adic spaces. 
\end{cor}

\begin{proof}
 This follows from Proposition \ref{fibprod} and the fact that every \'etale map is locally (on the source) a composition of rational embeddings and finite \'etale maps.
\end{proof}

\begin{dfn}\label{gc}
We denote by $\wRigSm^{\gc}/K$ the full subcategory of 
adic spaces whose objects are isomorphic to spaces $X=X_0\times_{\T^{N}}\widehat{\T}^{N}$ with respect to a map of affinoid rigid varieties $f\colon X_0\ra\T^N\times\T^M$ that is a composition of rational embeddings and finite \'etale maps. Such spaces will be called  \emph{smooth semi-perfectoid spaces with good coordinates}. 
Because of Proposition \ref{fibprod}, such fiber products $X=X_0\times_{\T^{N}}\widehat{\T}^{N}$ exist and are affinoid. Whenever $N=0$ these varieties are rigid analytic varieties and the full subcategory they form will be denoted by $\RigSm^{\gc}/K$ and referred to as\emph{ smooth affinoid rigid varieties with  {good coordinates}}. Whenever $M=0$ these varieties are perfectoid affinoid spaces and the full subcategory they form will be denoted by $\PerfSm^{\gc}/K$ and referred to as \emph{smooth affinoid perfectoids with  {good  coordinates}}. A perfectoid space $X$ in $\wRigSm^{\gc}/K$ will be sometimes denoted with $\widehat{X}$. 

When $X=X_0\times_{\T^{N}}\widehat{\T}^{N}$ is in $\wRigSm^{\gc}/K$ we denote by $X_h$ the fiber product $X_0\times_{\T^N}\T^N\langle \underline{\upsilon}^{1/p^h}\rangle$ and we will write $X=\varprojlim_h X_h$.  
We say that a presentation  $X=\varprojlim_hX_h$ of an object $X$ in $\wRigSm^{\gc}/K$ has \emph{ good reduction} if  the map $X_0\ra\T^n\times\T^m$ has an \'etale formal model $\mfX\ra\Spf(K^\circ\langle\underline{\upsilon}^{\pm1},\underline{\nu}^{\pm1}\rangle)$. We say that a presentation  $X=\varprojlim_hX_h$ of an object $X$ in $\wRigSm^{\gc}/K$ has \emph{potentially good reduction} if  there exists a finite separable field extension $L/K$ such that $X_L=\varprojlim_h(X_{h})_L$ has good reduction in $\wRigSm^{\gc}/L$. 
  We warn the reader that the association $X\mapsto X_0$ is not functorial and the varieties $X_h$ are not uniquely determined by $X$ in general. 

We denote by $\wRigSm/K$ the full subcategory of adic spaces  which are locally isomorphic to objects in $\wRigSm^{\gc}/K$ and its objects will be called \emph{smooth semi-perfectoid spaces}. We denote by $\RigSm/K$ the full subcategory of  adic spaces which are locally isomorphic to objects in $\RigSm^{\gc}/K$ and  by $\PerfSm/K$ the one of   adic spaces which are locally isomorphic to objects in $\PerfSm^{\gc}/K$. Its objects will be called \emph{smooth perfectoid spaces}. Whenever the context allows it, we omit $K$ from the notation.
\end{dfn}

\begin{rmk}
Any smooth rigid  variety (see for example  \cite[Definition 1.1.41]{ayoub-rig}) has locally good coordinates over $\T^N$ by \cite[Corollary 1.1.51]{ayoub-rig}. Hence  $\RigSm$ coincides with  the category of smooth rigid  varieties.
\end{rmk}

 We remark that the presentations of good reduction defined above are a special case of the objects  considered in \cite{andreatta-gen}.

The notation $X=\varprojlim_hX_h$ is justified by the following corollary, which is inspired by \cite[Proposition 2.4.5]{sw}.

\begin{cor}
 Let $Y=\Spa(S,S^+)$ be an affinoid such that $S^+$ is a ring of definition and let $X=\varprojlim_hX_h$ be in $\wRigSm^{\gc}$.  Then $\Hom(Y,X)\cong\varprojlim_h\Hom(Y,X_h)$.
\end{cor}

\begin{proof}Suppose that $X_h=\Spa(R_h,R_h^\circ)$ and $X=\Spa(R,R^+)$. By  Proposition \ref{fibprod} $R^+$ is a ring of definition and is the completion of $\varinjlim R_h^+$. 
   A map from $(R,R^+)$ to $(S,S^+)$ is uniquely determined by a $K^\circ$-linear map from $\varinjlim R_h^\circ$ to $S^+$. Similarly, a map from $(R_h,R_h^\circ)$ to $(S,S^+)$ is uniquely determined by a $K^\circ$-linear map from $R_h^\circ$ to $S^+$. From the isomorphism $\Hom_{K^+}(\varinjlim R_h^\circ,S^+)\cong\varprojlim_h\Hom_{K^+}(R_h^\circ,S^+)$ we then deduce the claim. 
\end{proof}

Let $\{X_h,f_h\}_{h\in I}$ be a cofiltered diagram of rigid varieties and let $\{X\ra X_h\}_{h\in I}$ be a collection of compatible maps of adic spaces. We recall that, according to \cite[Remark 2.4.5]{huber},  one writes $X\sim\varprojlim_hX_h$ when the following two conditions are satisfied:
\begin{enumerate}
 \item The induced map on topological spaces $|X|\ra\varprojlim_h|X_h|$ is a homeomorphism.
  \item For any $x\in X$ with images $x_h\in X_h$ the map of residue fields $\varinjlim_hk(x_h )\ra k(x)$ has dense image.
\end{enumerate}
The apparent clash of notations is solved by the following fact.
\begin{prop}
 Let $X=\varprojlim_hX_h$ be in $\wRigSm^{\gc}$. Then $X\sim\varprojlim_hX_h$.
\end{prop}

\begin{proof}
 This follows from $\widehat{\T}^{N}\sim\varprojlim_h\Spa K\langle \underline{\upsilon}^{\pm1/p^h}\rangle$ and from \cite[Proposition 7.16]{scholze}.
\end{proof}

\'Etale maps define a topology on $\wRigSm$ in the following way.

\begin{dfn}
A collection of \'etale maps of 
adic spaces $\{U_i\ra X\}_{i\in I}$ is an \emph{\'etale cover} if the induced map $\bigsqcup_{i\in I} U_i\ra X$ is surjective. These covers define a Grothendieck topology on $\wRigSm$ called the \emph{\'etale topology}.
 \end{dfn}

We pin down the following  facts on the topology of the objects $X=\varprojlim_h X_h$.%

\begin{prop}\label{liftmap}
Let $X=\varprojlim_h X_h$ be an object of $\wRigSm^{\gc}$. 
\begin{enumerate}
\item Any finite \'etale map $U\ra X$ is isomorphic to $U_{\bar{h}}\times_{X_{\bar{h}}}X$ for some integer ${\bar{h}}$ and some finite \'etale map $U_{\bar{h}}\ra X_{\bar{h}}$.
\item Any rational subspace $U\subset X$ is isomorphic to $U_{\bar{h}}\times_{X_{\bar{h}}}X$ for some integer $H$ and some rational subspace $U_{\bar{h}}\subset X_{\bar{h}}$.
\end{enumerate}
\end{prop}

\begin{proof}
The first statement follows from \cite[Lemma 7.5]{scholze}. 
For the second, we remark that according to \cite[Lemma 3.10]{huber1} any rational subspace $U=(f_1,\ldots,f_n|g)$ of $X$ can be defined by means of elements $f_i,g$ lying in $\varinjlim_h\mcO(X_h)$ since it is dense in $\mcO(X)$, hence the claim.
\end{proof}

The proposition above can also be used to prove that a finite \'etale extension of an object $\Spa(T,T^+)$ in $\wRigSm^{\gc}/K$ also lies  in $\wRigSm^{\gc}/K$ and hence arbitrary fiber products of \'etale maps in $\wRigSm/K$ exist, and are \'etale.

\begin{cor}\label{liftcov}
 Let $X=\varprojlim_h X_h$ be an object of $\wRigSm^{\gc}$ and let  $\mcU\colonequals\{f_i\colon U_i\ra X\}$ be an \'etale covering of adic spaces. There exists an integer ${\bar{h}}$ and a finite affinoid refinement $\{V_j\ra X\}$ of $\mcU$ which is obtained by pullback of an \'etale covering $\{V_{{\bar{h}}j}\ra X_{{\bar{h}}}\}$ of $X_{\bar{h}}$ and such that $V=\varprojlim_h V_{hj}$ lies in $\wRigSm^{\gc}$ by letting $V_{hj}$ be $V_{{\bar{h}}j}\times_{X_{\bar{h}}}X_h$ for all $h\geq {\bar{h}}$. 
\end{cor}

\begin{proof}
 Any \'etale map of adic spaces is locally a composition of rational embeddings and finite \'etale maps and they descend because of Proposition \ref{liftmap}. We therefore obtain an affinoid refinement $\{V_j\ra X\}$ of $\mcU$ such that each $V_i$ descends to some $X_h$. Since $X$ is quasi-compact, we can also refine this covering by a finite one, and  choose   a common index $\bar{h}$ where each $V_j$ descends to.
\end{proof}

\begin{cor} \label{tilt}
 A perfectoid space $X$ lies in $\PerfSm$ if and only if it is locally \'etale over $\widehat{\T}^{N}$.
\end{cor}

\begin{proof}
Let $X$ be locally \'etale over $\widehat{\T}^{N}$. Then it is locally open in a finite \'etale space over a rational subspace of $\widehat{\T}^{N}=\varprojlim_h\T^N\langle \underline{\upsilon}^{\pm1/p^h}\rangle$. By Proposition \ref{liftmap}, we conclude it is locally of the form $X_0\times_{\T^N}\widehat{\T}^{N}$ for some \'etale map $X_0\ra\T^N=\Spa(K\langle \underline{\upsilon}^{\pm1}\rangle,K^\circ\langle \underline{\upsilon}^{\pm1}\rangle)$ which is the composition of rational embeddings and finite \'etale maps.
\end{proof}

\begin{rmk}\label{fincpt}
 If $X$ is a smooth affinoid perfectoid space, then it has a finite number of connected components. Indeed, it is quasi-compact and locally isomorphic to a rational subspace of a perfectoid space which is finite \'etale over a rational subspace of $\widehat{\T}^N$.
\end{rmk}

For later use, we record the following simple example 
 of a space $X=\varprojlim_hX_h$ for which the varieties $X_h$ are easy to understand.

\begin{prop}\label{B1perf}
 Consider the smooth affinoid variety with good coordinates $$X_0= U\left({\upsilon-1}\mid{\pi}\right)\hookrightarrow\T^1=\Spa(K\langle\upsilon^{\pm1}\rangle).$$  One has  $X_h\cong\B^1$ for all $h$ and $\widehat{X}=\varprojlim_hX_h\cong\widehat{\B}^1$. 
\end{prop}

\begin{proof}
 By direct computation, the variety $X_h$ is isomorphic to $\Spa(K\langle\upsilon,\omega\rangle/(\omega^{p^h}-(\pi\upsilon+1)))$. Since $|p|\leq|\pi|$ we deduce that $|\binom{p^h}{i}|\leq|\pi|$ for all $0<i<p^h$. In particular, in the ring  $K\langle\upsilon,\omega\rangle/(\omega^{p^h}-(\pi\upsilon+1)) $  one has \[|(\omega-1)^{p^h}|=\left|\pi\upsilon+\sum_{i=1}^{p^h-1}\binom{p^h}{i}\omega^i\right|=|\pi|.\]
 Analogously, in the ring $K\langle\chi\rangle$  one has 
 \[
 |(\chi+\pi^{-1/p^h})^{p^h}-\pi^{-1}|=\left|\chi^{p^h}+\sum_{i=1}^{p^h-1}\binom{p^h}{i}\chi^{p^h-1}\pi^{-i/p^h}\right|=1.
 \]
The following maps are therefore well defined and clearly mutually inverse:
\[
 \begin{aligned}
X_h=\Spa(K\langle\upsilon,\omega\rangle/(\omega^{p^h}-(\pi\upsilon+1)) &\leftrightarrows \Spa(K\langle\chi\rangle)=\B^1\\
(\upsilon,\omega)&\mapsto((\chi+\pi^{-1/p^h})^{p^h}-\pi^{-1},\pi^{1/p^h}\chi+1)\\
\pi^{-1/p^h}(\omega-1)&\mapsfrom \chi.\\
 \end{aligned}
\]

Consider the multiplicative map $\sharp\colon K^\flat\langle\upsilon^{1/p^\infty}\rangle=(K\langle\upsilon^{1/p^\infty}\rangle)^\flat\ra K\langle\upsilon^{1/p^\infty}\rangle$ defined in \cite[Proposition 5.17]{scholze}. By  our assumptions on $\pi$ the element $(\upsilon-1)^\sharp-(\upsilon-1)$ is divisible by $\pi$ in $K^\circ\langle \upsilon^{1/p^\infty}\rangle$ and therefore the rational subspace $\widehat{X}\cong U\left({\upsilon-1}\mid{\pi}\right)$ of $\widehat{\T}^1$ coincides with $U\left({(\upsilon-1)^\sharp}\mid{\pi^{\flat\sharp}}\right)$. From \cite[Theorem 6.3]{scholze} we conclude $\widehat{X}^\flat\cong U\left({\upsilon-1}\mid{\pi^\flat}\right)\hookrightarrow\widehat{\T}^{\flat1}$ which is isomorphic to $\widehat{\B}^{\flat1}$ hence the claim.
\end{proof}

From the previous proposition we conclude in particular that the perfectoid space $\widehat{\B}^1$ lies in $\PerfSm^{\gc}$.

\section{Categories of adic motives}\label{motives}

From now on, we fix a commutative ring $\Lambda$ and work  with $\Lambda$-enriched categories. In particular, the term ``presheaf'' should be understood as ``presheaf of $\Lambda$-modules'' and similarly for the tem ``sheaf''. The presheaf $\Lambda(X)$ represented by an object $X$ of a category $\cat$ sends an object $Y$ of $\cat$ to the free $\Lambda$-module $\Lambda\Hom(Y,X)$. 

\begin{assu}
Unless otherwise stated, we assume from now on that $\Lambda$ is a $\Q$-algebra and we omit it from the notations.
\end{assu}

We make extensive use of the theory of model categories and localization, following the approach of Ayoub in  \cite{ayoub-rig} and \cite{ayoub-th2}. Fix a  site $(\cat,\tau)$. In our situation,  this will be the \'etale site of $\RigSm$ or $\wRigSm$. The category of complexes of presheaves $\Ch(\Psh(\cat))$ can be endowed with the \emph{projective model structure} for which weak equivalences are quasi-isomorphisms and fibrations are maps $\mcF\ra\mcF'$ such that $\mcF(X)\ra\mcF'(X)$ is a surjection for all $X$ in $\cat$ (cfr \cite[Section 2.3]{hovey} and \cite[Proposition 4.4.16]{ayoub-th2}). 

Also the category of complexes of sheaves $\Ch(\Sh_\tau(\cat))$ can be endowed with the \emph{projective model structure} defined in  \cite[Definition 4.4.40]{ayoub-th2}. In this structure, weak equivalences are quasi-isomorphisms of complexes of sheaves.

\begin{rmk}\label{projcof}
Let $\cat$ be a category. As shown in \cite{fausk} any projectively cofibrant complex $\mcF$ in $\Ch\Psh(\cat)$ is a retract of a complex that is the filtered colimit of bounded above complexes, each constituted by presheaves that are direct sums of representable ones.
\end{rmk}

Just like in \cite{jardine-s},  \cite{mvw}, \cite{mv-99} or \cite{riou}, we  consider the  left Bousfield localization of of the model category $\Ch(\Psh(\cat))$ with respect to the topology we select, and  a chosen ``contractible object''. We recall that left Bousfield localizations with respect to a class of maps $S$ (see \cite[Chapter 3]{hirschhorn}) is the universal model categories in which the maps in $S$ become weak equivalences. The existence of such structures is granted only under some technical hypotheses, as shown in \cite[Theorem 4.1.1]{hirschhorn} and \cite[Theorem 4.2.71]{ayoub-th2}.

\begin{prop}\label{locsets}
Let $(\cat,\tau)$ be a site with finite direct products  and let $\cat'$ be a full subcategory of $\cat$  such that every object of $\cat$ has a covering by objects of $\cat'$. Let also $I$ be an object of $\cat'$.
\begin{enumerate}
\item The projective model category $\Ch\Psh (\cat)$  admits a  left Bousfield localization $\Ch_{I}\Psh (\cat)$  with respect to the set $S_{I}$ of all maps $\Lambda(I\times X)[i]\ra\Lambda( X)[i]$ as $X$ varies in $\cat$  and $i$ varies in $\Z$. 
\item The projective model categories $\Ch\Psh (\cat)$ and $\Ch\Psh(\cat')$  admit  left Bousfield localizations  $\Ch_{\tau}\Psh (\cat)$ and $\Ch_{\tau}\Psh(\cat')$ with respect to the class $S_{\tau}$   of   maps $\mcF\ra\mcF'$ inducing isomorphisms on the $\et$-sheaves associated to $H_i(\mcF)$ and $H_i(\mcF')$ for all $i\in\Z$. Moreover, the two localized model categories are Quillen equivalent and the sheafification functor induces a Quillen equivalence to the projective model category  $\Ch\Sh_{\tau}(\cat)$.
\item The  model categories $\Ch_{\tau}\Psh (\cat)$ and $\Ch_{\tau}\Psh(\cat')$  admit  left Bousfield localizations $\Ch_{\tau,I}\Psh (\cat)$ and $\Ch_{\tau,I}\Psh(\cat')$ with respect to the set $S_{I}$ defined above.  Moreover, the two localized model categories are Quillen equivalent.
\end{enumerate}
\end{prop}

\begin{proof}According to  \cite[Theorem 4.1.1]{hirschhorn}, any model category which is  {left proper} and {cellular} (some technical properties which are defined in \cite[Definitions 12.1.1 and 13.1.1]{hirschhorn}) admits a left Boudsfield localization with respect to a set of maps. 
The model structure on complexes is left proper and cellular (see \cite[Page 7]{ss-1998}). It follows that the projective model structures  in the statement are also left proper and cellular (see  \cite[Propositions 12.1.5 and 13.1.14]{hirschhorn})  hence the first claim.

For the first part of second claim, it suffices to apply \cite[Proposition 4.4.31, Lemma 4.4.34]{ayoub-th2} showing that the localization over $S_\tau$ is equivalent to a localization over a set of maps. The second part is a restatement of  \cite[Corollary 4.4.42, Proposition 4.4.55]{ayoub-th2}.

Since by  \cite[Proposition 4.4.31]{ayoub-th2} the $\tau$-localization coincides with the Bousfield localization with respect to a set, we conclude by \cite[Theorem 4.2.71]{ayoub-th2} that the  model category $\Ch_{\tau}\Psh (\cat)$ is still left proper and cellular. The last statement then follows from  \cite[Theorem 4.1.1]{hirschhorn} and the second claim.
\end{proof}

In the situation above, we will denote by $S_{(\tau,I)}$ the union of the class $S_{\tau}$ and the set $S_I$.

\begin{rmk}
A geometrically relevant situation is induced when $I$ is endowed with a multiplication map $\mu\colon I\times I\ra I$ and maps $i_0$ and  $i_1$ from the terminal object to $I$ satisfying the relations of a monoidal object with $0$ as in the definition of an interval object (see \cite[Section 2.3]{mv-99}). Under these hypotheses, we say that the triple $(\cat,\tau,I)$ is a \emph{site with an interval}. 
\end{rmk}

\begin{exm}
The affinoid rigid variety with good coordinates $\B^1=\Spa K\langle \chi\rangle$ is an interval object with respect to the natural multiplication $\mu$ and maps $i_0$ and  $i_1$ induced by the substitution  $\chi\mapsto0 $ and  $\chi\mapsto1$ respectively. 
\end{exm}

We now apply the constructions above to the sites introduced in the previous sections. We recall that we consider adic spaces defined over a perfectoid field $K$.

\begin{cor}
The following pairs of model categories are Quillen equivalent.
\begin{itemize}
\item $\Ch_{\et}\Psh(\RigSm)$ and $\Ch_{\et}\Psh(\RigSm^{\gc})$.
\item $\Ch_{\et,\B^1}\Psh(\RigSm)$ and $\Ch_{\et,\B^1}\Psh(\RigSm^{\gc})$.
\item $\Ch_{\et}\Psh(\wRigSm)$ and $\Ch_{\et}\Psh(\wRigSm^{\gc})$.
\item $\Ch_{\et,\B^1}\Psh(\wRigSm)$ and $\Ch_{\et,\B^1}\Psh(\wRigSm^{\gc})$.
\end{itemize}
\end{cor}

\begin{proof}
 It suffices to apply Proposition \ref{locsets} to the sites with interval $(\RigSm,\et,\B^1)$ and $(\wRigSm,\et,\B^1)$ where $\cat'$ is in both cases the subcategory of varieties with good coordinates.
\end{proof}

\begin{dfn}
For $\eta\in\{\et, \B^1,(\et,\B^1)\}$ we say that a map   in  $\Ch\Psh (\RigSm)$ [resp. $\Ch\Psh (\wRigSm)$]  is a \emph{$\eta$-weak equivalence} if it is a weak equivalence in the model structure   $\Ch_{\eta}\Psh (\RigSm)$ [resp. $\Ch_{\eta}\Psh (\wRigSm)$]. 
The triangulated homotopy category associated to the localization $\Ch_{\et,\B^1}\Psh (\RigSm)$ [resp. to the localization  $\Ch_{\et,\B^1}\Psh (\wRigSm)$]   is denoted by $\RigDA_{\et}^{\eff}(K,\Lambda)$ [resp.  $\wRigDA_{\et,\B^1}^{\eff}(K,\Lambda)$].  
We omit $\Lambda$ from the notation whenever the context allows it. The image of a variety $X$ in one of these categories is denoted by $\Lambda(X)$. 
We say that an object $\mcF$ of the derived category $\catD=\catD(\Psh (\RigSm))$ [resp. $\catD=\catD(\Psh (\wRigSm))$] is \emph{$\eta$-local} if the functor $\Hom_{\catD}(\cdot,\mcF)$ sends maps in $S_\eta$ (see Proposition \ref{locsets}) to isomorphisms. This amounts to say that $\mcF$ is  quasi-isomorphic to a $\eta$-fibrant object. 
\end{dfn}

We need to keep track of  $\B^1$ in the notation of $\wRigDA_{\et,\B^1}^{\eff}(K,\Lambda)$ since later we will perform a localization on $\Ch\Psh(\wRigSm)$ with respect to a different interval object.

\begin{rmk}
Using the language of \cite{bv-dg}, the localizations defined above induce endofunctors $C^\eta$ of the derived categories 
$\catD(\Psh(\RigSm))$, $\catD(\Psh(\RigSm^{\gc}))$,  $\catD(\Psh(\wRigSm))$ and $\catD(\Psh(\wRigSm^{\gc}))$ such that $C^\eta\mcF$ is $\eta$-local for all $\mcF$ and there is a natural transformation $C^\eta\ra\id$ which is a pointwise $\eta$-weak equivalence. 
 The  functor $C^\eta$ restricts to a triangulated equivalence on the objects $\mcF $ that are $\eta$-local  
and one can compute the Hom set $\Hom(\mcF,\mcF')$ in the the homotopy category of the $\eta$-localization  as ${\catD}(\mcF,C^\eta\mcF')$.
\end{rmk}

\begin{rmk}\label{Cet}According to Proposition \ref{locsets},   for any $X$ in $\wRigSm$ and any integer $i$, one has $\Hom_{\mathbf{D}}(\Lambda(X)[-i],C^{\et}\mcF)\cong\Hom_{\mathbf{D}(\Sh_{\et}(\wRigSm))}(\Lambda(X)[-i],\mcF)$. The latter group is  the \'etale hypercohomology group $\HH_{\et}^i(X,\mcF)$ which can be computed with respect to the \emph{small} \'etale site  over $X$. The property $\Hom_{\mathbf{D}}(\Lambda(X)[-i],C^{\et}\mcF)\cong \HH^i(X_{\et},\mcF)$ characterizes $C^{\et}\mcF$ up to quasi-isomorphisms (and holds true for more general topologies, see \cite[Proposition 4.4.58]{ayoub-th2}).
\end{rmk}

We now show that the \'etale localization 
 can alternatively be described in terms of \'etale hypercoverings $\mcU_\bullet\ra X$ (see for example \cite{dhi}). Any such datum defines a simplicial presheaf $n\mapsto\bigoplus_i\Lambda(U_{ni})$ whenever $\mcU_n=\bigsqcup_i h_{U_{ni}}$ is the sum of the presheaves of sets $h_{U_{ni}}$ represented by $U_{ni}$. This simplicial presheaf can be associated to a normalized chain complex, that we   denote by $\Lambda(\mcU_\bullet)$. It is is endowed with a map to $\Lambda(X)$.

\begin{prop}\label{bddhypercov}
 The localization over  $S_{\et}$ on $\Ch\Psh(\RigSm^{\gc})$ [resp. $\Ch\Psh(\wRigSm^{\gc})$] coincides with the localization over the set $\Lambda(\mcU_\bullet)[i]\ra\Lambda(X)[i]$ as $\mcU_\bullet\ra X$ varies among bounded \'etale hypercoverings of the objects $X$ of $\RigSm^{\gc}$ [resp. $\wRigSm^{\gc}$] and $i$ varies in $\Z$.
\end{prop}

\begin{proof}
Any $\et$-local object $\mcF$ is also local with respect to the maps of the statement. We are left to prove that a complex $\mcF$ which is local with respect to the maps of the statement is also $\et$-local.

 Since $\Lambda$ contains $\Q$ the \'etale cohomology of an \'etale sheaf $\mcF$ coincides with the Nisnevich cohomology (the same proof of \cite[Proposition 14.23]{mvw} holds also here). By means of \cite[Corollary 1.2.21]{ayoub-rig} we conclude that any rigid variety $X$ has a finite cohomological dimension. By \cite[Theorem V.7.4.1]{SGAIV2} and \cite[Theorem 0.3]{sv-bk}, we obtain for any rigid variety $X$ and any complex of presheaves $\mcF$ an isomorphism
\[
 \HH^n_{\et}(X,\mcF)\cong \varinjlim_{\mcU_{\bullet}\in HR_\infty(X)}H_{-n}\RHom(\Lambda(\mcU_{\bullet}),\mcF)
\]
where $HR_\infty(X)$ is the category of bounded \'etale hypercoverings of $X$ (see \cite[V.7.3]{SGAIV2}) and $\RHom$ is the $\Hom$-complex computed in the unbounded derived category of presheaves.  
Suppose now $\mcF$ is local with respect to the maps of the statement. 
 Then $\RHom(\Lambda(\mcU_{\bullet}),\mcF)$ is quasi-isomorphic to $\RHom(X,\mcF)$ for every bounded hypercovering $\mcU_{\bullet}$ hence $H_{-n}\mcF(X)\cong\HH^n_{\et}(X,\mcF)$ by the formula above. We then conclude that the map $\mcF\ra C^{\et}\mcF$ is a quasi-isomorphism,  proving the proposition.
\end{proof}

As the following proposition shows, there are also alternative presentations of the  homotopy categories introduced so far, which we will later use.

\begin{prop}\label{cech}
 Let $\Lambda$ be a $\Q$-algebra. The natural inclusion induces a Quillen equivalence \[L_{S}\Ch(\Psh({\wRigSm^{\gc}}))\leftrightarrows\Ch_{\et}\Psh({\wRigSm}^{\gc})\] where  $L_S$ denotes the Bousfield localization with respect to  the set $S$ of shifts of the maps of complexes induced by \'etale Cech hypercoverings $\mcU_\bullet\ra X$ of objects $X$ in ${\wRigSm^{\gc}}$ such that for some presentation $X=\varprojlim_h X_h$ the covering $\mcU_0\ra X$  descends to a covering of $X_0$. 
\end{prop}

\begin{proof}
Using  Proposition \ref{bddhypercov}, it suffices to prove that the map $\Lambda(\mcU_\bullet)\ra\Lambda( X)$ is an isomorphism in the homotopy category $L_S\Ch(\Psh({\wRigSm^{\gc}}))$ for a fixed bounded \'etale hypercovering $\mcU_\bullet$ of an object $X$ in ${\wRigSm^{\gc}}$.

Since the inclusion functor $\Ch_{\geq0}\ra\Ch$ is a Quillen functor, it suffices to prove that $\Lambda(\mcU_\bullet)\ra\Lambda( X)$ is a weak equivalence in $L_T\Ch_{\geq0}(\Psh({\wRigSm^{\gc}}))$ where $T$ is the set of shifts of the maps of  complexes induced by \'etale Cech hypercoverings descending at finite level. Let $L_{\tilde{T}}\sPsh({\wRigSm^{\gc}})$ be the Bousfield localization of the projective model structure on simplicial presheaves of sets with respect to the set $\tilde{T}$ formed by maps  induced by \'etale Cech hypercoverings $\mcU_\bullet\ra X$ descending at finite level. We remark that the Dold-Kan correspondence (see \cite[Section 4.1]{ss-2003}) and the $\Lambda$-enrichment also define a left Quillen functor from $L_{\tilde{T}}\sPsh({\wRigSm^{\gc}})$ to the category $L_T\Ch_{\geq0}(\Psh({\wRigSm^{\gc}}))$. It therefore suffices to prove that $\mcU_\bullet\ra  X$ is a weak equivalence in $L_{\tilde{T}}\sPsh({\wRigSm^{\gc}})$  and this follows from the fact that bounded hypercovering define the same localization as Cech hypercoverings (see 
\cite[Theorem A.6]{dhi}) together with the fact that coverings descending to finite level define the same topology (Corollary \ref{liftcov}) and hence the same localization (\cite[Corollary A.8]{dhi}). We remark that  \cite[Corollary A.8]{dhi} applies in our case even if the coverings $\mcU\ra X$ descending to the finite level do not form a basis of the topology, as their pullback via an arbitrary map $Y\ra X$ may not have the same property. However, the proof of the statement  relies on \cite[Proposition A.2]{dhi}, where it is only used that the chosen family of coverings $\mcU\ra X$ generates the topology and that the fiber product $\mcU\times_X\mcU$ is defined. 
\end{proof}

\begin{rmk}
It is shown in the proof that the statements of Propositions \ref{bddhypercov} and \ref{cech}  hold true without any assumptions on $\Lambda$  under the condition that all varieties $X$ have  finite cohomological dimension with respect to the \'etale topology. 
\end{rmk}

As we pointed out in Remark \ref{Cet},  there is a  characterization of   $C^{\et}\mcF$ for any complex $\mcF$. 
This is also true for the $\B^1$-localization, described in the following part.

\begin{dfn}\label{cocu}
We denote by $\square$ the $\Sigma$-enriched cocubical object (see \cite[Appendix A]{ayoub-h1}) defined by putting $\square^n=\B^n=\Spa K\langle \tau_1,\ldots,\tau_n\rangle$ and considering the morphisms $d_{r,\epsilon}$ induced by the maps $\B^n\ra\B^{n+1}$ corresponding to the substitution $\tau_r=\epsilon$ for $\epsilon\in\{0,1\}$ and the morphisms $p_r$ induced by the projections $\B^n\ra\B^{n-1}$.
 For any  variety $X$ and any presheaf $\mcF$ with values in an abelian category, we can therefore consider the $\Sigma$-enriched cubical object  $\mcF(X\times\square)$ (see \cite[Appendix A]{ayoub-h1}). Associated to any $\Sigma$-enriched cubical object $\mcF$ there are the following complexes: the complex $C^\sharp_\bullet\mcF$ defined as $C^\sharp_n\mcF=\mcF_n$ and with differential $\sum (-1)^r (d_{r,1}^*-d_{r,0}^*)$; the \emph{simple complex} $C_\bullet\mcF$ defined as $C_n\mcF=\bigcap_{r=1}^n\ker d_{r,0}^*$ and with differential $\sum (-1)^r d_{r,1}^*$; the \emph{normalized complex} $N_\bullet\mcF$ defined as $N_n\mcF=C_n\cap\mcF\bigcap_{r=2}^n\ker d_{r,1}^*$ and with differential $-d_{1,1}^*$. By \cite[Lemma A.3, Proposition A.8, Proposition A.11]{ayoub-h2}, the inclusion $N_\bullet\mcF\hookrightarrow C_\bullet\mcF$ is a quasi-isomorphism and both inclusions $C_\bullet\mcF\hookrightarrow C^\sharp_\bullet\mcF$ and $N_\bullet\mcF\hookrightarrow C_\bullet\mcF$ split.

  For any complex of presheaves $\mcF$ we define the \emph{singular complex of $\mcF$}, denoted by $\Sing^{\B^1}\mcF$, to be the total complex of the simple complex associated to the  $\uhom(\Lambda(\square),\mcF)$. It sends the object $X$ to the total complex of the simple complex associated to $\mcF(X\times\square)$.
\end{dfn}

The following lemma is the cocubical version of \cite[Lemma 2.18]{mvw}.

\begin{lemma}\label{chainhomo}
For any presheaf $\mcF$ the two maps of cubical sets $i_0^*,i_1^*\colon\mcF(\square\times\B^1)\ra\mcF(\square)$ induce chain homotopic maps on the associated simple and normalized complexes. 
\end{lemma}

\begin{proof}
Consider     the isomorphism $s_n\colon\B^{n+1}\ra\B^n\times\B^1$ defined on points by separating the last coordinate and let $s_n^*$ be the induced map $\mcF(\square^n\times\B^1)\ra\mcF(\square^{n+1})$. We have $s_{n-1}^*\circ d_{r,\epsilon}^*=d_{r,\epsilon}^*\circ s_{n}^*$ for all $1\leq r\leq n$ and $\epsilon\in\{0,1\}$. We conclude that 
\[                                     
\begin{aligned}
s_{n-1}^*\circ\sum_{r=1}^{n}(-1)^r(d_{r,1}^*-d_{r,0}^*)+\sum_{r=1}^{n+1}(-1)^r(d_{r,1}^*-d_{r,0}^*)\circ (-s_{n}^*)\\=(-1)^n(d_{n+1,1}^*\circ s_n^*-d_{n+1,0}^*\circ s_n^*)=(-1)^n(i_{1}^*-i_0^*).  
\end{aligned}                                        
\]
Therefore, the maps $\{(-1)^ns_n^*\}$ define a chain homotopy from $i_0^*$ to $i_1^*$ as maps of complexes $C^\sharp_\bullet\mcF(\square\times\B^1)\ra C^\sharp_\bullet\mcF(\square)$. 

We automatically deduce that if an inclusion $C'_\bullet\mcF\ra C^\sharp_\bullet\mcF$ has a functorial retraction, then the maps $i_0^*,i_1^*\colon C'_\bullet\mcF(\square\times\B^1)\ra C'_\bullet\mcF(\square)$ are also chain homotopic.
\end{proof}

The following proposition is the rigid analytic analogue of \cite[Theorem 2.23]{ayoub-h1}, or the cocubical analogue of \cite[Lemma 2.5.31]{ayoub-rig}.

\begin{prop}\label{sing}
Let $\mcF$ be a complex in $\Ch\Psh(\wRigSm)$. 
Then $\Sing^{\B^1}\mcF$ is $\B^1$-local and $\B^1$-weak equivalent to $\mcF$ in $\Ch\Psh(\wRigSm)$.
\end{prop}

\begin{proof}
 In order to prove that $\Sing^{\B^1}\mcF$ is $\B^1$-local in $\Ch\Psh(\wRigSm)$ we need to check that each homology presheaf $H_n(\Sing^{\B^1}\mcF ) $ is homotopy-invariant. By means of \cite[Proposition 2.2.46]{ayoub-rig} it suffices to show that  the maps $i_0^*,i_1^*\colon N_\bullet\mcF(\square\times\B^1)\ra N_\bullet\mcF(\square)$ are chain homotopic, and this follows from Lemma \ref{chainhomo}.

We now prove that $\Sing^{\B^1}\mcF$ is $\B^1$-weak equivalent to $\mcF$. 
We first prove that the canonical map $a\colon\mcF\ra\uhom(\Lambda(\square^n),\mcF)$ has an inverse up to homotopy for a fixed $n$. 
Consider the map $b\colon\uhom(\Lambda(\square^n),\mcF)\ra \mcF$ induced by the zero section of $\square^n$. It holds that $b\circ a=\id$ and $a\circ b$ is homotopic to $\id$ via the map
\[
H\colon\Lambda(\B^1)\otimes\uhom(\Lambda(\square^n),\mcF)\ra\uhom(\Lambda(\square^n),\mcF)                                                                        
\]
which is deduced from the adjunction $(\Lambda(\B^1)\otimes\cdot,\uhom(\Lambda(\B^1),\cdot))$ and the map
\[
 \uhom(\Lambda(\square^n),\mcF)\ra\uhom(\Lambda(\B^1\times\square^n),\mcF)
\]
defined via the homothety of $\B^1$ on $\square^n$. As $\B^1$-weak equivalences are stable under filtered colimits and cones, we also conclude that the total complex associated to the simple complex of  $\uhom(\Lambda(\square),\mcF)$ is $\B^1$-equivalent to the one associated to the constant cubical object $\mcF$ (see for example the argument of \cite[Corollary 2.5.36]{ayoub-rig}) which is in turn quasi-isomorphic to $\mcF$.
\end{proof}

\begin{cor}\label{replacement}
Let $\Lambda$ be a $\Q$-algebra. For any $\mcF$ in $\Ch\Psh(\wRigSm)$  the localization $C^{\B^1}\mcF$ is quasi-isomorphic to $\Sing^{\B^1}\mcF$ and
 the localization $C^{\et,\B^1}\mcF$ is quasi-isomorphic to $\Sing^{\B^1}(C^{\et}\mcF^\bullet)$.
\end{cor}

\begin{proof}
  The first claim follows from Proposition \ref{sing}. We are left to prove that the complex $\Sing^{\B^1}(C^{\et}\mcF^\bullet)$ is $\et$-local. To this aim, we use the description given in Proposition \ref{bddhypercov} and we show that $\Sing^{\B^1}(C^{\et}\mcF^\bullet)$ is local with respect to shifts of maps $\Lambda(\mcU_\bullet)\ra\Lambda(X)$ induced by bounded hypercoverings $\mcU_\bullet\ra X$. 
  
Fix such a hypercovering $\mcU_\bullet\ra X$. From the isomorphisms \[H_p\RHom(\Lambda(\mcU_\bullet\times\square^q),C^{\et}\mcF)\cong H_p\RHom(\Lambda(X\times\square^q),C^{\et}\mcF)\] valid for all $p,q$ and a spectral sequence argument (see \cite[Theorem 0.3]{sv-bk}) we deduce that \[\catD(\Lambda(X)[n],\Sing^{\B^1}{C^{\et}\mcF})\cong\catD(\Lambda(\mcU_\bullet)[n],\Sing^{\B^1}C^{\et}\mcF)\] for all $n$ as wanted.
\end{proof}

We now investigate some of the natural Quillen functors which arise between the model categories introduced so far. We start by considering the natural inclusion of categories $\RigSm\ra\wRigSm$

\begin{prop}\label{Dff}
The inclusion $\RigSm\hookrightarrow{\wRigSm}$ induces a Quillen adjunction
\[
\adj{\iota^*}{\Ch_{{\et},\B^1}\Psh(\RigSm)}{\Ch_{{\et},\B^1}\Psh({\wRigSm})}{\iota_*}. 
\]
Moreover, the functor $\LL\iota^*\colon\RigDA_{\et}^{\eff}(K)\ra{\wRigDA_{\et,\B^1}^{\eff}}(K)$ is fully faithful.
\end{prop}

\begin{proof}
The first claim is a special instance of \cite[Proposition 4.4.45]{ayoub-th2}.

We prove the second claim by showing that $\RR\iota_*\LL\iota^*$ is isomorphic to the identity. Let $\mcF $ be a cofibrant object in $\Ch_{{\et},\B^1}\Psh(\RigSm)$. We need to prove that the map $\mcF\ra\iota_*(\Sing^{\B^1}C^{\et}(\iota^*\mcF))$ is an $(\et,\B^1)$-weak equivalence. Since $\iota_*$ commutes with $\Sing^{\B^1}$ we are left to prove that the map $\iota_*\iota^*\mcF=\mcF\ra\iota_*C^{\et}(\iota^*\mcF)$ is an $\et$-weak equivalence. This follows since $\iota_*$ preserves $\et$-weak equivalences, as it commutes with $\et$-sheafification.
\end{proof}

We are now interested in finding 
 a convenient set of compact objects which generate the  categories above, as triangulated categories with small sums. This will simplify many definitions and proofs in what follows.

\begin{prop}\label{genRigDA}
The category $\RigDA_{\et}^{\eff}(K)$ [resp.  $\wRigDA_{\et,\B^1}^{\eff}(K)$] 
is compactly generated (as a triangulated category with small sums) by motives $\Lambda(X)$ associated to rigid varieties $X$ which are in $\RigSm^{\gc}$ [resp. $\wRigSm^{\gc}$].
\end{prop}

\begin{proof}
The statements are analogous, and we only consider the case of $\wRigDA_{\et,\B^1}^{\eff}(K)$. It is clear that the set of functors $H_i\RHom(\Lambda(X),\cdot)$ detect quasi-isomorphisms between \'etale local objects, by letting $X$ vary in $\wRigSm^{\gc}$ and $i$ vary in $\Z$.  We are left to prove that the motives $\Lambda(X)$ with $X$ in $\wRigSm^{\gc}$ are compact. Since $\Lambda(X)$ is compact in $\catD(\Psh(\wRigSm^{\gc}))$ and $\Sing^{\B^1}$ commutes with direct sums, it suffices to prove that if $\{\mcF_i\}_{i\in I}$ is a family of $\et$-local complexes, then also $\bigoplus_i\mcF_i$ is $\et$-local. 
If $I$ is finite, the claim follows from the isomorphisms $H_{-n}\RHom(X,\bigoplus_i\mcF_i)\cong\bigoplus_i \HH^{n}(X,\mcF_i)\cong\HH^n(X,\bigoplus_i\mcF_i)$. A coproduct over an arbitrary family is a filtered colimit of finite coproducts, hence the claim follows from \cite[Proposition 4.5.62]{ayoub-th2}.
\end{proof}

\begin{rmk}
The above proof shows that the statement of Proposition \ref{genRigDA} holds true without any assumptions on $\Lambda$  under the condition that all varieties $X$ have  finite cohomological dimension with respect to the \'etale topology.
\end{rmk}

We now introduce the category of motives associated to smooth perfectoid spaces, using the same formalism as before. In this category, the canonical choice of the ``interval object'' for defining homotopies is the perfectoid ball $\widehat{\B}^1$.

\begin{exm}
 The perfectoid ball $\widehat{\B}^1=\Spa(K\langle \chi^{1/p^\infty}\rangle,K^\circ\langle \chi^{1/p^\infty}\rangle)$ is an interval object with respect to the natural multiplication $\mu$ and maps $i_0$ and  $i_1$ induced by the substitution  $\chi^{1/p^h}\mapsto0 $ and  $\chi^{1/p^h}\mapsto1$ respectively.
\end{exm}

The perfectoid variety $\widehat{\B}^1$ naturally lives in $\wRigSm$ and has good coordinates by Proposition \ref{B1perf}. It can therefore be used to define another homotopy category out of $\Ch\Psh(\wRigSm)$ and  $\Ch\Psh(\wRigSm^{\gc})$.

\begin{cor}
The following pairs of model categories are Quillen equivalent.
\begin{itemize}
\item $\Ch_{\et}\Psh(\PerfSm)$ and $\Ch_{\et}\Psh(\PerfSm^{\gc})$.
\item $\Ch_{\et,\widehat{\B}^1}\Psh(\PerfSm)$ and $\Ch_{\et,\widehat{\B}^1}\Psh(\PerfSm^{\gc})$.
\item $\Ch_{\et}\Psh(\wRigSm)$ and $\Ch_{\et}\Psh(\wRigSm^{\gc})$.
\item $\Ch_{\et,\widehat{\B}^1}\Psh(\wRigSm)$ and $\Ch_{\et,\widehat{\B}^1}\Psh(\wRigSm^{\gc})$.
\end{itemize}
\end{cor}

\begin{proof}
 It suffices to apply Proposition \ref{locsets} to the sites with interval $(\PerfSm,\et,\widehat{\B}^1)$ and $(\wRigSm,\et,\widehat{\B}^1)$ where $\cat'$ is in both cases the subcategory of affinoid rigid varieties with good coordinates.
\end{proof}

\begin{dfn}
For $\eta\in\{\et,\widehat{\B}^1,(\et,\widehat{\B}^1)\}$ we say that a map   in  $\Ch\Psh (\PerfSm)$ [resp. $\Ch\Psh (\wRigSm)$]  is a \emph{$\eta$-weak equivalence} if it is a weak equivalence in the model structure   $\Ch_{\eta}\Psh (\PerfSm)$ [resp. $\Ch_{\eta}\Psh (\wRigSm)$]. 
We say that an object $\mcF$ of the derived category $\catD=\catD(\Psh (\PerfSm))$ [resp. $\catD=\catD(\Psh (\wRigSm))$] is \emph{$\eta$-local} if the functor $\Hom_{\catD}(\cdot,\mcF)$ sends maps in $S_\eta$ (see Proposition \ref{locsets}) to isomorphisms. This amounts to say that $\mcF$ is  quasi-isomorphic to a $\eta$-fibrant object.  
The triangulated homotopy category associated to the localization $\Ch_{\et,\widehat{\B}^1}\Psh (\PerfSm)$ [resp. $\Ch_{\et,\widehat{\B}^1}\Psh (\wRigSm)$]   will be denoted by $\PerfDA_{\et}^{\eff}(K,\Lambda)$ [resp.  $\wRigDA^{\eff}_{\et,\widehat{\B}^1}(K,\Lambda)$].   
We will omit $\Lambda$ whenever the context allows it. The image of a variety $X$ in one of these categories will be denoted by $\Lambda(X)$. 
\end{dfn}

We recall one of the main results of Scholze \cite{scholze}, reshaped in our derived homotopical setting. It will constitute the bridge to pass from characteristic $p$ to characteristic $0$. As summarized in Theorem \ref{tilteq} there is an equivalence of categories between perfectoid affinoid $K$-algebras and perfectoid affinoid $K^\flat$-algebras, extending to an equivalence between the categories of perfectoid spaces over $K$ and over $K^\flat$ (see \cite[Proposition 6.17]{scholze}). We refer to this equivalence as the \emph{tilting equivalence}.

\begin{prop}\label{tiltingeq}
 There exists an equivalence of triangulated categories
\[
 \adj{(-)^\sharp}{\PerfDA_{\et}^{\eff}(K^\flat)}{\PerfDA_{\et}^{\eff}(K)}{(-)^\flat}
\]
induced by the tilting equivalence. 
\end{prop}

\begin{proof}
 The tilting equivalence (see Theorem \ref{tilteq}) induces an equivalence of the \'etale sites on perfectoid spaces over $K$ and over $K^\flat$ (see \cite[Theorem 7.12]{scholze}). 
  Moreover $(\widehat{\T}^n)^\flat=\widehat{\T}^n$ and $(\widehat{\B}^n)^\flat=\widehat{\B}^n$. It therefore induces an equivalence of sites with interval $(\PerfSm/K,\et,\widehat{\B}^1)\cong(\PerfSm/K^\flat,\et,\widehat{\B}^{1})$ hence the claim.
\end{proof}

We now investigate the triangulated functor between the categories of motives induced by the natural embedding $\PerfSm\ra\wRigSm$ in the same spirit of what we did previously in Proposition \ref{Dff}.

\begin{prop}\label{Dff2}
The inclusion $\PerfSm\hookrightarrow{\wRigSm}$ induces a Quillen adjunction
\[
\adj{j^*}{\Ch_{{\et},\widehat{\B}^1}\Psh(\PerfSm)}{\Ch_{{\et},\widehat{\B}^1}\Psh({\wRigSm})}{j_*}. 
\]
Moreover, the functor $\LL j^*\colon\PerfDA_{\et}^{\eff}(K)\ra{\wRigDA^{\eff}_{\et,\widehat{\B}^1}}(K)$ is fully faithful.
\end{prop}

\begin{proof}
The result follows in the same way as Proposition \ref{Dff}.
\end{proof}

Also in this framework, the $\widehat{\B}^1$-localization has a very explicit construction. Most proofs are  straightforward analogues of those relative to the $\B^1$-localizations, and will therefore be omitted.

\begin{dfn}
We denote by $\widehat{\square}$ the $\Sigma$-enriched cocubical object (see \cite[Appendix A]{ayoub-h2}) defined by putting $\widehat{\square}^n=\widehat{\B}^n=\Spa K\langle \tau^{1/p^\infty}_1,\ldots,\tau_n^{1/p^{1/\infty}}\rangle$ and considering the morphisms $d_{r,\epsilon}$ induced by the maps $\widehat{\B}^n\ra\widehat{\B}^{n+1}$ corresponding to the substitution $\tau_r^{1/p^h}=\epsilon$ for $\epsilon\in\{0,1\}$ and the morphisms $p_r$ induced by the projections $\widehat{\B}^n\ra\widehat{\B}^{n-1}$. 
For any complex of presheaves $\mcF$ we let $\Sing^{\widehat{\B}^1}\mcF$ be the total complex of the simple complex associated to $\uhom(\widehat{\square},\mcF)$. It sends the object $X$ to the total complex of the simple complex associated to $\mcF(X\times\widehat{\square})$.
\end{dfn}

\begin{prop}\label{singp}
Let $\mcF$ be a complex in $\Ch\Psh(\PerfSm)$ [resp. in $\Ch\Psh(\wRigSm)$]. 
Then $\Sing^{\widehat{\B}^1} \mcF$ is $\widehat{\B}^1$-local and $\widehat{\B}^1$-weak equivalent to $\mcF$. 
\end{prop}

\begin{proof}
 The fact that $\Sing^{\widehat{\B}^1}\mcF$ is $\widehat{\B}^1$-local in $\Ch\Psh(\wRigSm)$ can be deduced by Lemma \ref{i0=i1} and Lemma \ref{chainhomop}. 
We are left to prove that $\Sing^{\widehat{\B}^1}\mcF$ is $\widehat{\B}^1$-weak equivalent to $\mcF$ and this follows in the same way as in the proof of Proposition \ref{sing}.
\end{proof}

The following lemmas are used in the previous proof.

\begin{lemma}\label{i0=i1}
 A presheaf $\mcF$ in  $\Psh(\Sm\Perf)$ [resp. in $\Psh (\wRigSm)$] is $\widehat{\B}^1$-invariant if and only if $i_0^*=i_1^*\colon\mcF(X\times\widehat{\B}^1)\ra\mcF(X)$ for all $X$ in $\Sm\Perf$ [resp. in $\wRigSm$].
\end{lemma}

\begin{proof}
 This follows in the same way as \cite[Lemma 2.16]{mvw}.
\end{proof}

\begin{lemma}\label{chainhomop}
For any presheaf $\mcF$ the two maps of cubical sets $i_0^*,i_1^*\colon\mcF(\widehat{\square}\times\widehat{\B}^1)\ra\mcF(\widehat{\square})$ induce chain homotopic maps on the associated simple and normalized complexes. 
\end{lemma}

\begin{proof}
This follows in the same way as Lemma \ref{chainhomo}.
\end{proof}

\begin{cor}\label{wreplacement}
Let $\mcF$ be in $\Ch\Psh(\PerfSm)$ [resp. in $\Ch\Psh(\wRigSm)$] the $(\et,\widehat{\B}^1)$-localization $C^{\et,\widehat{\B}^1}\mcF$ is quasi-isomorphic to $\Sing^{\widehat{\B}^1}(C^{\et}\mcF)$.
\end{cor}

\begin{proof}
This follows in the same way as Corollary \ref{replacement}.
\end{proof}

\begin{prop}\label{genPerfDA} 
The category $\PerfDA_{\et}^{\eff}(K)$ [resp. $\wRigDA^{\eff}_{\et,\widehat{\B}^1}(K)$] is compactly generated (as a triangulated category with small sums) by motives $\Lambda(X)$ associated to rigid varieties $X$ which are in $\PerfSm^{\gc}$ [resp. $\wRigSm^{\gc}$].
\end{prop}

\begin{proof}
This follows in the same way as Proposition \ref{genRigDA}.
\end{proof}

\begin{rmk}
The above proof shows that the statement of Proposition \ref{genPerfDA} holds true without any assumptions on $\Lambda$  under the condition that all varieties $X$ have  finite cohomological dimension with respect to the \'etale topology. 
\end{rmk}

So far, we have defined two different Bousfield localizations on complexes of presheaves on $\wRigSm$ according to two different choices of intervals: $\B^1$ and  $\widehat{\B}^1$. We remark that the second constitutes a further localization of the first, in the following sense.

\begin{prop}\label{b1>b1}
 $\B^1$-weak equivalences in $\Ch\Psh(\wRigSm)$ are $\widehat{\B}^1$-weak equivalences.
\end{prop}

\begin{proof}
It suffices to prove that $X\times\B^1\ra X$ induces a $\widehat{\B}^1$-weak equivalence, for any variety $X$ in $\wRigSm$. This follows as the multiplicative homothety $\widehat{\B}^{1}\times\B^1\ra\B^1$ induces a homotopy between the zero map and the identity on $\B^1$.
\end{proof}

\begin{cor}
 The triangulated category  $\wRigDA^{\eff}_{\et,\widehat{\B}^1}(K)$ is equivalent to the full triangulated subcategory of  $\wRigDA_{\et,\B^1}^{\eff}(K)$ formed by $\widehat{\B}^{1}$-local objects.
\end{cor}

\begin{proof}
 Because of Proposition \ref{b1>b1}, the triangulated category $\wRigDA^{\eff}_{\et,\widehat{\B}^1}(K)$ coincides with the localization of $\wRigDA^{\eff}_{\et,{\B}^1}(K)$ with respect to the set generated by the maps $\Lambda(\widehat{\B}^1_X)[n]\ra\Lambda( X)[n]$ as $X$ varies in $\wRigSm$ and $n$ in $\Z$. 
\end{proof}

We end this section by recalling the definition of rigid motives with transfers. 
The notion of {finite correspondence}  plays an important role in Voevodsky's theory of motives. In the case of rigid varieties over a field $K$ correspondences give rise to the category $\RigCor(K)$ as defined in \cite[Definition 2.2.29]{ayoub-rig}. 

\begin{dfn}
Additive presheaves over $\RigCor(K)$ are called \emph{presheaves with transfers}, and the category they form is denoted by $\PST(\RigSm/K,\Lambda)$ or simply by $\PST(\RigSm)$ when the context allows it.
\end{dfn}

   By \cite[Definition 2.5.15]{ayoub-rig}, the projective model category $\Ch\PST(\RigSm)$ admits a Bousfield localization  $\Ch_{\et,\B^1}\PST (\RigSm)$ with respect to the union of the class  of   maps $\mcF\ra\mcF'$ inducing isomorphisms on the $\et$-sheaves associated to $H_i(\mcF)$ and $H_i(\mcF')$ for all $i\in\Z$ and the set of all maps $\Lambda(\B^{1}_X)[i]\ra \Lambda(X)[i]$  as $X$ varies in $\RigSm$  and $i$ varies in $\Z$.

\begin{dfn}
The triangulated homotopy category associated to   $\Ch_{\et,\B^1}\PST (\RigSm)$  will be denoted by $\RigDM_{\et}^{\eff}(K,\Lambda)$.  
We will omit $\Lambda$ from the notation whenever the context allows it. The image of a variety $X$ in   will be denoted by $\Lambda_{\tr}(X)$. 
\end{dfn}

\begin{rmk}
Since $\Lambda$ is a $\Q$-algebra, one can equivalently consider the Nisnevich topology in the definition above and obtain a homotopy category $\RigDM_{\Nis}^{\eff}(K,\Lambda)$ which is equivalent to $\RigDM_{\et}^{\eff}(K,\Lambda)$. 
\end{rmk}

\begin{rmk}\label{atrotr}
 The faithful embedding of categories $\RigSm\ra\RigCor$ induces a Quillen adjunction (see \cite[Proposition 2.5.19]{ayoub-rig}):
\[
\adj{a_{tr}}{\Ch_{\et,\B^1}\Psh(\RigSm)}{\Ch_{\et,\B^1}\PST(\RigSm)}{o_{tr}}
\]
such that $a_{tr}\Lambda(X)=\Lambda_{tr}(X)$ for any $X\in\RigSm$ and $o_{tr}$ is the functor of forgetting transfers. These functors induce an adjoint pair:
\[
\adj{\LL a_{tr}}{\RigDA^{\eff}_{\et}(K)}{\RigDM^{\eff}_{\et}(K)}{\RR o_{tr}}
\]
which is investigated in \cite{vezz-DADM}. 
\end{rmk}

\section{Motivic interpretation of approximation results}\label{motapprox}

In all this section, $K$ is a perfectoid field of arbitrary characteristic. 
We begin by presenting an approximation result whose proof is differed to  Appendix \ref{approx}.

\begin{prop}
\label{solleva2text}
Let $X=\varprojlim_hX_h$ be in $\wRigSm^{\gc}$. 
Let also $Y$ be an affinoid rigid variety endowed with an \'etale map $Y\ra\B^m$. For a given finite set of maps $\{f_1,\ldots,f_N\}$ in $\Hom(X\times\B^n,Y)$ we can find corresponding maps $\{H_1,\ldots,H_N\}$ in $\Hom(X\times\B^n\times\B^1, Y)$ and an integer $\bar{h}$ such that:
\begin{enumerate}
\item For all $1\leq k\leq N$ it holds $i_0^*H_k=f_k$ and $i_1^*H_k$ factors over the canonical map $X\ra X_{\bar{h}}$.
\item If ${f}_k\circ d_{r,\epsilon}={f}_{k'}\circ d_{r,\epsilon}$ for some $1\leq k,k'\leq N$ and some $(r,\epsilon)\in\{1,\ldots,n\}\times\{0,1\}$ then $H_k\circ d_{r,\epsilon}=H_{k'}\circ d_{r,\epsilon}$.
\item If for some $1\leq k\leq N$ and some $h\in\N$ the map $f_k\circ d_{1,1}\in \Hom(X\times\B^{n-1},Y)$  lies in $\Hom(X_h\times\B^{n-1},Y)$ then the element $H_k\circ d_{1,1}$ of $\Hom(X\times\B^{n-1}\times\B^1,Y)$ is constant on $\B^1$  equal to $f_k\circ d_{1,1}$.
\end{enumerate}
\end{prop}

The statement above has the following  interpretation in terms of complexes.

\begin{prop}\label{sollevaDA}
 Let $X=\varprojlim_hX_h$ be in $\wRigSm^{\gc}$ and let $Y$ be in $\RigSm^{\gc}$. The natural map
\[
\phi\colon\varinjlim_h (\Sing^{\B^1}\Lambda{(Y)})(X_h)\ra( \Sing^{\B^1}\Lambda(Y))(X)
\]
is a quasi-isomorphism.
\end{prop}

\begin{proof}
We need to prove that the natural map
\[
\phi\colon\varinjlim_h\, C_\bullet\Lambda\Hom(X_h\times\square,Y)\ra C_\bullet\Lambda\Hom(X\times\square,Y)
\]
defines bijections on homology groups.

We start by proving surjectivity. As $\square$ is a $\Sigma$-enriched cocubical object, the complexes above are quasi-isomorphic to the associated normalized complexes $N_\bullet$ which  we consider instead. Suppose that $\beta\in\Lambda\Hom(X\times\square^n,Y)$ defines a cycle in $N_n$ i.e. $\beta\circ d_{r,\epsilon}=0$ for $1\leq r\leq n$ and $\epsilon\in\{0,1\}$. This means that $\beta=\sum \lambda_k f_k$ with $\lambda_k\in\Lambda$, $f_k\in\Hom(X\times\square^n,Y)$ and $\sum \lambda_k f_k\circ d_{r,\epsilon}=0$. This amounts to say that for every $k,r,\epsilon$ the sum $\sum\lambda_{k'}$ over the indices $k'$ such that $f_{k'}\circ d_{r,\epsilon}=f_{k}\circ d_{r,\epsilon}$ is zero. By Proposition \ref{solleva2text}, we can find an integer $h$ and  maps $H_k\in\Hom(X\times\square^n\times\B^1,Y)$ such that $i_0^*H=f_k$, $i_1^*H=\phi(\tilde{f}_{k})$ with $\tilde{f}_k\in\Hom(X_h\times\square^n,Y)$ and $H_k\circ d_{r,\epsilon}=H_{k'}\circ d_{r,\epsilon}$ whenever ${f}_k\circ d_{r,\epsilon}={f}_{k'}\circ d_{r,\epsilon}$. If we denote by $H$ the cycle $\sum\lambda_k H_k\in\Lambda\Hom(X\times\square^n\times\B^1,Y)$ we therefore have $d_{r,\epsilon}^*H=0$ for all $r,\epsilon$. 

By Lemma \ref{chainhomo}, we conclude that $i_1^*H$ and $i_0^*H$ define the same homology class, and therefore  $\beta$ defines the same class as $i_1^*H$ which is the image of a class in $\Lambda\Hom(X_h\times\square^n,Y)$ as wanted.

We now turn to injectivity. Consider an element $\alpha\in\Lambda\Hom(X_0\times\square^n,Y)$ such that $\alpha\circ d_{r,\epsilon}=0$ for all $r,\epsilon$ and suppose there exists an element $\beta=\sum\lambda_if_i\in\Lambda\Hom(X\times\square^{n+1},Y)$ such that $\beta\circ d_{r,0}=0$ for $1\leq r\leq n+1$, $\beta\circ d_{r,1}=0$ for $2\leq r\leq n+1$ and $\beta\circ d_{1,1}=\phi(\alpha)$. Again, by Proposition \ref{solleva2text}, we can find an integer ${\bar{h}}$ and maps $H_k\in\Hom(X\times\square^{n+1}\times\B^1,Y)$ such that $H\colonequals\sum\lambda_k H_k$ satisfies $i_1^*H=\phi(\gamma)$ for some $\gamma\in\Lambda\Hom(X_{\bar{h}}\times\square^{n+1},Y)$, $H\circ d_{r,0}=0$ for $1\leq r\leq n+1$, $H\circ d_{r,1}=0$ for $2\leq r\leq n+1$ and $H\circ d_{1,1}$ is constant on $\B^1$ and coincides with $\phi(\alpha)$. We conclude that $\gamma\in N_n$ and $d\gamma=\alpha$. In particular, $\alpha=0$ in the homology group, as wanted.
\end{proof}

\begin{cor}\label{singqi}
 Let $\mcF$ be a projectively cofibrant complex in $\Ch\Psh(\RigSm^{\gc})$. 
For any $X=\varprojlim_hX_h$ in $\wRigSm^{\gc}$ the natural map
\[
\phi\colon\varinjlim_h (\Sing^{\B^1}\mcF)(X_h)\ra (\Sing^{\B^1}\iota^*\mcF)(X)
\]
is a quasi-isomorphism.
\end{cor}

\begin{proof}
As homology commutes with filtered colimits, by means of Remark \ref{projcof} we can assume that $\mcF$  is a bounded above complex formed by sums of representable presheaves. For any $X$ in $\wRigSm$ the homology of $\Sing^{\B^1}\mcF(X)$ coincides with the homology of the total complex associated to $C_\bullet(\mcF(X\times\square))$. The result then follows from Proposition \ref{sollevaDA} and the convergence of the spectral sequence associated to the double complex above, which is concentrated in one quadrant.
\end{proof}

The following technical proposition is actually a crucial point of our proof, as it allows  some explicit computations of morphisms in the category $\wRigDA^{\eff}_{\et}(K)$.

 \begin{prop}\label{islocal}
Let $\mcF$ be a cofibrant and $(\B^1,{\et})$-fibrant complex in $\Ch\Psh(\RigSm^{\gc})$. 
  Then $\Sing^{\B^1}(\iota^*\mcF)$ is $(\B^1,{{\et}})$-local in $\Ch\Psh(\wRigSm^{\gc})$.
\end{prop}

\begin{proof}
The difficulty lies in showing that the object  $\Sing^{\B^1}(\iota^*\mcF)$  is $\et$-local. By Propositions \ref{cech} and \ref{sing}, it suffices to prove that $\Sing^{\B^1}(\iota^*\mcF)$ is local with respect to the \'etale-Cech hypercoverings $\mcU_\bullet\ra X$ in $\wRigSm^{\gc}$ of $X=\varprojlim_hX_h$ descending at finite level. Let $\mcU_\bullet\ra X$ be one of them. Without loss of generality, we assume that it descends to an \'etale covering of $X_0$. 
In particular we conclude that $\mcU_n=\varprojlim_h\mcU_{nh}$ is a disjoint union of objects in $\wRigSm^{\gc}$.

We need to show that $\RHom(\Lambda(\mcU_\bullet),\Sing^{\B^1}(\iota^*\mcF))$ is quasi-isomorphic to $\Sing^{\B^1}(\iota^*\mcF)(X)$.
Using Corollary \ref{singqi}, we conclude that for each $n\in\N$ the complex $(\Sing^{\B^1}\iota^*\mcF)(\mcU_n)$ is quasi-isomorphic to $\varinjlim_h(\Sing^{\B^1}\iota^*\mcF)(\mcU_{nh})$. Passing to the homotopy limit on $n$ on both sides,  we deduce that  $\RHom(\Lambda(\mcU_\bullet),\Sing^{\B^1}\iota^*\mcF)$ is quasi-isomorphic to $\varinjlim_h \RHom(\Lambda(\mcU_{\bullet h}),\Sing^{\B^1}\iota^*\mcF)$. Using again Corollary \ref{singqi}, we also obtain that $(\Sing^{\B^1}\iota^*\mcF)(X)$ is quasi-isomorphic to $\varinjlim_h(\Sing^{\B^1}\iota^*\mcF)(X_h)$. 

From the exactness of $\varinjlim$ it suffices then to prove that the maps
\[
\RHom(\Lambda(\mcU_{\bullet h}),\Sing^{\B^1}\mcF)\ra \RHom(\Lambda(X_h),\Sing^{\B^1}\mcF)
\]
are quasi-isomorphisms. This follows once we show that the complex $\Sing^{\B^1}\mcF$   is $\et$-local.

 We point out that since $\mcF$ is $\B^1$-local, then the canonical map $\mcF\ra\Sing^{\B^1}\mcF$ is a quasi-isomorphism. As $\mcF$ is  $\et$-local we conclude that $\Sing^{\B^1}\mcF$  also is, hence the claim.
\end{proof}

We are finally ready to state the main result of this section.

\begin{prop}\label{colimok}
 Let $X=\varprojlim_hX_h$ be in $\wRigSm^{\gc}$. For any complex of presheaves $\mcF$ on $\RigSm^{\gc}$ the natural map
\[
\varinjlim_h\RigDA_{\et}^{\eff}(K)(\Lambda(X_h),\mcF)\ra\wRigDA_{\et,\B^1}^{\eff}(K)(\Lambda(X),\LL\iota^*\mcF) 
\]
is an isomorphism.
\end{prop}

\begin{proof}
Since any complex $\mcF$ has a fibrant-cofibrant  replacement
 in $\Ch_{\et,\B^1}\Psh(\RigSm^{\gc})$ we can  assume that $\mcF$ is cofibrant and $(\et,\B^1)$-fibrant.  Since it is  $\B^1$-local, it is quasi-isomorphic to $\Sing^{\B^1}\mcF$. By Corollary \ref{singqi}, for any integer $i$  one has  
\[
\varinjlim_h\Hom(\Lambda(X_h)[i],\Sing^{\B^1}\mcF)\cong\Hom(\Lambda(X)[i],\Sing^{\B^1}\iota^*\mcF).
\]
As $\Lambda(X)$ is a cofibrant object in $\Ch\Psh({\wRigSm}^{\gc})$ and $\Sing^{\B^1}\iota^*\mcF$ is a $(\B^1,{\et})$-local replacement of $\mcF$  
in $\Ch_{\et,\B^1}\Psh(\wRigSm^{\gc})$ by Proposition \ref{islocal}, we conclude that 
the previous isomorphism can be rephrased in the following way:
\[
\varinjlim_h{\RigDA_{\et}^{\eff}}(K)(\Lambda(X_h)[i],\mcF)\cong{\wRigDA_{\et,\B^1}^{\eff}}(K)(\Lambda(X)[i],\LL\iota^*\mcF)
\]
proving the claim.
\end{proof}

\section{The de-perfectoidification functor in characteristic \texorpdfstring{$0$}{0}}\label{deperf0}

The   results proved in Section \ref{motapprox}  are valid both for $\car K=0$ and $\car K=p$. On the contrary, the results of this section  require that $\car K=0$. We will present later their variant for the case $\car K=p$.

We start by considering the adjunction between motives with and without transfers (see Remark \ref{atrotr}). 
 Thanks to the following  theorem, we are allowed to add or ignore  transfers according to the situation.  

\begin{thm} [{\cite{vezz-DADM}}]
\label{DA=DM}
Suppose that $\car K=0$. The functors $(a_{tr}, o_{tr})$ induce an equivalence:
\[
\adj{\LL a_{tr}}{\RigDA_{\et}^{\eff}(K)}{\RigDM_{\et}^{\eff}(K)}{\R o_{tr}}.
\]
\end{thm}

\begin{rmk}\label{Qiscruc1}
 The proof of the statement above  uses in a crucial way the fact that the  ring of coefficients $\Lambda$ is a $\Q$-algebra. %
\end{rmk}

\begin{prop}\label{tower}
Suppose $\car K=0$. Let $X=\varprojlim_hX_h$ be in $\wRigSm^{\gc}$. If $h$ is big enough, then the map $\Lambda(X_{h+1})\ra\Lambda(X_h)$ is an isomorphism in $\RigDA_{\et}^{\eff}(K)$.
\end{prop}

\begin{proof}
 By means of Theorem \ref{DA=DM}, we can equally prove the statement in the category $\RigDM_{\et}^{\eff}(K)$. 
We claim that we can also make an arbitrary finite  field extension $L/K$. Indeed the transpose of the natural map $Y_L\ra Y$ is a correspondence from $Y$ to $Y_L$. Since $\Lambda$ is a $\Q$-algebra, we conclude that $\Lambda_{\tr}(Y)$ is a direct factor of $\Lambda_{\tr}(Y_L)=\LL e_\sharp\Lambda_{\tr}(Y_L)$ for any variety $Y$ where $\LL e_\sharp$ is the functor $ \RigDM_{\et}^{\eff}(L)\ra\RigDM_{\et}^{\eff}(K)$ induced by restriction of scalars. In particular, if $\Lambda_{\tr}((X_{h+1})_L)\ra\Lambda_{\tr}((X_h)_L)$ is an isomorphism in $\RigDM_{\et}^{\eff}(L)$ then  $
\Lambda_{\tr}((X_{h+1})_L)\ra\Lambda_{\tr}((X_h)_L)$ is an isomorphism in $\RigDM_{\et}^{\eff}(K)$ and therefore also  $\Lambda_{\tr}(X_{h+1})\ra\Lambda_{\tr}(X_h)$ is.

By Lemma \cite[1.1.52]{ayoub-rig}, we can suppose that $X_0=\Spa(R_0,R_0^\circ)$ with $R_0=S\langle \sigma,\tau\rangle/(P(\sigma,\tau))$ where $S=\mcO(\T^M)$, $\sigma=(\sigma_1,\ldots,\sigma_N)$ is a $N$-tuple of coordinates, $\tau=(\tau_1,\ldots,\tau_m)$ is a $m$-tuple of coordinates  and $P$ is a set of $m$ polynomials  in $ S[\sigma,\tau]$ with  $\det(\frac{\del P}{\del \tau})\in R_0^\times$. In particular $X_1=\Spa(R_1,R_1^\circ)$ with $R_1=S\langle \sigma,\tau\rangle/(P(\sigma^p,\tau))$ and the map $f\colon X_1\ra X_0$ is induced by $\sigma\mapsto\sigma^p$, $\tau\mapsto\tau$. Since the map $f$ is finite and surjective, we can also consider the transpose correspondence $f^T\in\RigCor(X_0,X_1)$. The composition $f\circ f^T$ is associated to the correspondence $X_0\stackrel{f}{\leftarrow} X_1\stackrel{f}{\rightarrow} X_0$ which is the cycle $\deg(f)X_0=p^N\cdot\id_{X_0}$. The composition $f^T\circ f$ is associated to the correspondence $X_1\stackrel{p_1}{\leftarrow} X_1\times_{X_0}X_1\stackrel{p_2}{\ra} X_1$. Since $\T^N\langle\sigma^{1/p}\rangle\times_{\T^N}\T^N\langle\sigma^{1/p}\rangle\cong \T^N\langle\sigma^{1/p}\rangle\times\mu_p^N$ we conclude that the above correspondence is $X_1\stackrel{p_1}{\leftarrow} X_1\times(\mu_p)^N\stackrel{\eta}{\ra} X_1$ where $\eta$ is induced by  the multiplication map $\T^N\times\mu_p^N\ra\T^N$. 
Up to a finite field extension, we can  assume that $K$ has the $p$-th roots of unity. The above correspondence is then equal to $\sum f_\zeta$ where each $f_\zeta$ is a map $X_1\ra X_1$ defined by $\sigma_i\mapsto \zeta_i\sigma_i$, $\tau\mapsto\tau$ for each $N$-tuple $\zeta=(\zeta_i)$ of $p$-th roots of unity. If we prove that each $f_\zeta$ is homotopically equivalent to $\id_{X_1}$ then we get $\frac{1}{p^N}f^{T}\circ f=\id$, $f\circ\frac{1}{p^N}f^T=\id$ in $\RigDM_{\et}^{\eff}$ as wanted.

 We are left to find a homotopy between $\id$ and  $ f_\zeta$ for a fixed $\zeta=(\zeta_1,\ldots,\zeta_n)$ up to considering higher indices $h$. For the sake of clarity, we consider them as maps $\Spa \bar{R}_1\ra\Spa {R}_1$ 
  where we put $\bar{R}_h=S\langle\bar{\sigma},\bar{\tau}\rangle/(P(\bar{\sigma}^{p^h},\bar{\tau}))$ for any integer $h$. The first map is induced by $\sigma\mapsto\bar{\sigma}$, $\tau\mapsto\bar{\tau}$ and the second induced by $\sigma\mapsto\zeta\bar{\sigma}$, $\tau\mapsto\bar{\tau}$. Let $F_h=\sum_n a_n (\sigma-\bar{\sigma})^n$ be the unique array of formal power series in $\bar{R}_h[[\sigma-\bar{\sigma}]]$ centered in $\bar{\sigma}$ associated to the polynomials 
$P(\sigma^{p^h},\tau)$  
in $\bar{R}_h[\sigma,\tau]$ via Corollary \ref{implicitC}. Let also $\phi_h$ be the map $\bar{R}_h\ra\bar{R}_{h+1}$. From the formal equalities $P(\sigma^{p^{(h+1)}},F_{h+1}(\sigma))=0$, $P(\sigma^{p^h},\phi(F_h(\sigma)))=\phi_h(P(\sigma^{p^h},F_h(\sigma)))=0$ and the uniqueness of $F_{h+1}$ we deduce $F_{h+1}(\sigma)=\phi_h(F_h(\sigma^p))$.

 We therefore have
 \[\begin{aligned}
F_{h+1}(\sigma)&=\sum_n \phi_h(a_n) (\sigma^p-\bar{\sigma}^p)^n\\
&=\sum_n \phi_h(a_n)\left((\sigma-\bar{\sigma})^{p-1}+\sum_{j=1}^{p-1}\binom{p}{j}(\sigma-\bar{\sigma})^{j-1}\bar{\sigma}^{p-j}\right)^n\!(\sigma-\bar{\sigma})^n.
\end{aligned}
\]
 The expression
 \[
Q(x)=x^{p}+\sum_{j=1}^{p-1}\binom{p}{j}x^{j}\bar{\sigma}^{p-j}
   	\]
   	is a polynomial in $x$ and it easy to show that the mapping
   	 $   	x\mapsto Q(x)$ 
   	extends to a map $\bar{R}_{h+1}\langle x\rangle\ra\bar{R}_{h+1}\langle x\rangle$. 
   	We deduce that we can read off the convergence in the circle of radius $1$ around $\bar{\sigma}$ and the values of $F_{h+1}$ on its expression given above.
   	
We remark that the norm of $Q(\sigma-\bar{\sigma})$ in the circle of radius $\rho\leq1$ around $\bar{\sigma}$ is bounded by $\max\{\rho^{p},|p|\}\leq\max\{\rho,|p|\}$.
 Suppose that $F_h$ converges in a circle of radius $\rho$ with $0<\rho\leq1$ around $\bar{\sigma}$ and in there it takes values in power-bounded elements. By the expression above, the same holds true for $F_{h+1}$ in the circle of radius $\min\{\rho|p|^{-1},1\}$ around $\bar{\sigma}$.  By induction we conclude that 
  for a sufficiently big $h$ the power series $F_h$ converges in a circle of radius $\delta>|p|^{1/(p-1)}$ around $\bar{\sigma}$ and its values in it are power bounded. 
Up to rescaling indices, we suppose that this holds for $h=1$. 

From the relation $F_{h+1}(\sigma)=\phi_h(F_h(\sigma^p))$ we also conclude $F_1(\zeta\bar{\sigma})=F_1(\bar{\sigma})=\bar{\tau}$. Therefore, since $|\zeta_i-1|=|p|^{1/(p-1)}$ for all $i$, the map
\[
\begin{aligned}
X_1=\Spa(S\langle\sigma,\tau\rangle/P(\sigma^p,\tau))&\leftarrow X_1\times\B^1=\Spa(S\langle\bar{\sigma},\bar{\tau},\chi\rangle/(P(\bar{\sigma}^p,\bar{\tau}))\\
(\sigma_i,\tau_j)&\mapsto(\bar{\sigma}_i+(\zeta_i-1)\bar{\sigma}_i\chi,  F_1(\bar{\sigma}+(\zeta-1)\bar{\sigma}\chi))
\end{aligned}
\]
is a well defined map, inducing a homotopy between $\id_{X_1}$ and $f_\zeta$ as claimed.
\end{proof}

It cannot be expected that all maps $X_{h+1}\ra X_h$ are isomorphisms in $\RigDA^{\eff}_{\et}(K)$: consider for example $X_0=\T^1\langle\nu^{1/p}\rangle\ra\T^1$. Then $X_0$  is a connected variety, while $X_1$ is not. That said, there is a particular class of objects $X=\varprojlim_hX_h$  in $\wRigSm^{\gc}$ for which this happens: this is the content of the following proposition which nevertheless will  not be used in the following. 
 
We recall that a presentation 
$X=\varprojlim_hX_h$ of an object in $\wRigSm^{\gc}$ is of  good reduction if the map $X_0\ra\T^N\times\T^M$ has a formal model which is an \'etale map over $\Spf K^\circ\langle\underline{\upsilon}^{\pm1},\underline{\nu}^{\pm1}\rangle$ and is of potentially good reduction if this happens after base change by a separable finite field extension $L/K$.

\begin{prop}\label{goodred}
 Let $\car K=0$ and let 
  $X=\varprojlim_hX_h$ be a presentation of a variety in $\wRigSm^{\gc}$ of potentially good reduction. The maps $\Lambda(X_{h+1})\ra\Lambda(X_h)$ are isomorphisms in $\RigDA_{\et}^{\eff}(K)$ for all $h$.
\end{prop}

\begin{proof}
 If the map $X_0\ra \T^N\times\T^M$ has an \'etale formal model, then also the map $X_h\ra \T^N\langle\underline{\upsilon}^{1/p^h}\rangle\times\T^M $ does. It is then sufficient to consider only the case $h=0$. 
  Since $L/K$ is finite   and $\Lambda$ is a $\Q$-algebra, by the same argument of the proof of Proposition \ref{tower} we can assume that $\varprojlim_hX_h$ has good reduction.  Also, by means of Theorem \ref{DA=DM} and the Cancellation theorem \cite[Corollary 2.5.49]{ayoub-rig}, we can equally prove the statement in the stable category $\RigDA_{\et}(K)$ defined in \cite[Definition 1.3.19]{ayoub-rig}.

Let $\mfX_0\ra\Spf K^\circ\langle\underline{\upsilon}^{\pm1},\underline{\nu}^{\pm1}\rangle$ be a formal model of the map $X_0\ra\T^n\times\T^m$. We let $\bar{X}_0$ be the special fiber over the residue field $k$ of $K$. The variety $X_1$ has also a smooth formal model $\mfX_1$ whose special fiber is $\bar{X}_1$. By definition, the natural map $\bar{X}_1\ra\bar{X}_0$ is the push-out of the (relative) Frobenius map $\A^{\dim X}_{k}\ra \A^{\dim X}_{k}$ which is isomorphic to the relative Frobenius map and hence an isomorphism of correspondences as $p$ is invertible in $\Lambda$. We conclude that $\Lambda_{\tr}(\bar{X}_1)\ra\Lambda_{\tr}(\bar{X}_0)$ is an isomorphism in $\DM_{\et}(k)$. 

Let $\FormDA_{\et}(K^\circ)$ be the stable category of motives of formal varieties $\FSH_{\mfM}(K^\circ)$ defined in \cite[Definition 1.4.15]{ayoub-rig} associated to the model category $\mfM=\Ch(\Lambda\Mod)$. 
Using \cite[Theorem B.1]{ayoub-etale} we deduce that the map $\Lambda(\bar{X}_1)\ra\Lambda(\bar{X}_0)$ is an isomorphism in $\DA_{\et}(k)$ as is its  image  via the following functor (see \cite[Remark 1.4.30]{ayoub-rig}) induced by  the special fiber functor  and the generic fiber functor:
$$\xymatrix{
\DA_{\et}(k)&\FormDA_{\et}(K^\circ)\ar[r]^-{(-)_\eta}\ar[l]^-{\sim}_-{(-)_\sigma}&\RigDA_{\et}(K).
}$$
This morphism is precisely the map $\Lambda(X_1)\ra\Lambda(X_0)$ proving the claim. 
\end{proof}

We are now ready to present the main result of this section.

\begin{thm}\label{premain}
Let $\car K=0$. The functor $\LL\iota^*\colon\RigDA^{\eff}_{\et}(K)\ra{\wRigDA^{\eff}_{\et}}(K)$ has a left adjoint  $\LL\iota_!$ 
and the counit map $\id\ra\LL\iota_!\LL\iota^*$ is invertible. 
 Whenever $X=\varprojlim_hX_h$ is an object of $\wRigSm^{\gc}$
then $\LL\iota_!\Lambda(X)\cong\Lambda(X_h)$ for a sufficiently large index $h$.  If moreover $X=\varprojlim_hX_h$ is  of potentially good reduction,
then $\LL\iota_!\Lambda(X)\cong\Lambda(X_0)$.
\end{thm}

\begin{proof}
We start by proving 
 that the canonical map
\[
\RigDA^{\eff}_{\et}(K)(\Lambda(X_{\bar{h}}),\mcF)\ra{\wRigDA^{\eff}_{\et}(K)}(\Lambda(X),\LL\iota^*\mcF)
\]
is an isomorphism, for every $X=\varprojlim_hX_h$ and for $\bar{h}$ big enough. 
 By Proposition \ref{colimok}, it suffices to prove that the natural map
\[
\RigDA^{\eff}(K)(\Lambda(X_{\bar{h}}),\LL a_{tr}\mcF)\ra\varinjlim_h\RigDA^{\eff}(K)(\Lambda(X_h),\LL a_{tr}\mcF)
\]
is an isomorphism for some $\bar{h}$. This follows from Proposition \ref{tower} since the maps $\Lambda(X_{h+1})\ra\Lambda( X_h)$  are isomorphisms if $h\geq \bar{h}$ for some big enough $\bar{h}$. In case $\varprojlim_hX_h$ is of potentially good reduction, then Proposition \ref{goodred} ensures that we can choose $\bar{h}=0$.

We conclude that the subcategory $\catT$ of $\wRigDA^{\eff}_{\et,\widehat{\B}^1}(K)$ formed by the objects $M$ such that the functor $N\mapsto{\wRigDA^{\eff}_{\et,\widehat{\B}^1}}(K)(M,\LL\iota^*N)$ is corepresentable contains all motives $\Lambda(X)$ with $X$ any object  of ${\wRigSm^{\gc}}$. Since these objects form a set of compact generators of $\wRigDA^{\eff}_{\et,\widehat{\B}^1}(K)$ by Proposition \ref{genRigDA}, we deduce  the existence of the functor $\LL\iota_!$  by Lemma \ref{extriangadj}.

The formula $\LL\iota_!\LL\iota^*\cong\id$ is a formal consequence of the fact that $\LL\iota_!$ is the left adjoint of a fully faithful functor $\LL\iota^*$ (see Proposition \ref{Dff}).
\end{proof}

\begin{lemma}\label{extriangadj}
Let $\mfG\colon\catT\ra\catT'$ be a triangulated functor of triangulated categories. The full subcategory $\cat$ of $\catT'$ of objects $M$ such that the functor $a_M\colon N\mapsto\Hom(M,\mfG N)$ is corepresentable is closed under cones and small direct sums.
\end{lemma}

\begin{proof}
For any object $M$ in $\cat$ we denote by $\mfF M$ the object corepresenting the functor $a_M$. Let now $\{M_i\}_{i\in I}$ be a set of objects in $\cat$. It is immediate to check that $\bigoplus_{i}\mfF M_i$ corepresents the functor $a_{\bigoplus_i M_i}$.

Let now  $M_1$, $M_2$ be two objects of $\cat$  and $f\colon M_1\ra M_2$ be a map between them. There are canonical maps $\eta_i\colon M_i\ra\mfG\mfF M_i$ induced by the identity $\mfF M_i\ra\mfF M_i$ and the universal property of $\mfF M_i$. By composing with $\eta_2$ we obtain a morphism $\Hom(M_1,M_2)\ra\Hom(M_1,\mfG\mfF M_2)\cong\Hom(\mfF M_1,\mfF M_2)$ sending $f$ to a map $\mfF f$.  Let $C$ be the cone of $f$ and $D$ be the cone of $\mfF f$. We claim that $D$ represents $a_C$. From the triangulated structure we obtain a  map of distinguished  triangles
$$\xymatrix{
M_1\ar[r]^f\ar[d]^{\eta_1}	&	M_2\ar[r]\ar[d]^{\eta_2}	&	C\ar[r]\ar[d]	&\\
\mfG\mfF M_1\ar[r]^{\mfG\mfF f}	&	\mfG\mfF M_2\ar[r]	&	\mfG D\ar[r]&
}$$
inducing for any object $N$ of $\catT$ the following maps of long exact sequences
$$\xymatrix{
&\Hom(M_1,\mfG N)\ar[l]	&	\Hom(M_2,\mfG N)\ar[l]	&	\Hom(C,\mfG N)\ar[l]	&\ar[l]\\
&\Hom(\mfG\mfF M_1,\mfG N)\ar[l]\ar[u]&	\Hom(\mfG\mfF M_2,\mfG N)\ar[l]	\ar[u]&	\Hom(\mfG D,\mfG N)\ar[l]\ar[u]&\ar[l]\\
&\Hom(\mfF M_1,N)\ar[l]\ar[u]&	\Hom(\mfF M_2,N)\ar[l]	\ar[u]&	\Hom(D,N)\ar[l]\ar[u]&\ar[l]
}$$
Since the vertical compositions are isomorphisms for $M_1$ and $M_2$ we deduce that they all are, proving that $D$ corepresents $a_C$ as wanted. 
\end{proof}

We remark that we used the fact that $\Lambda$ is a $\Q$-algebra at least twice in the proof of Theorem \ref{premain}: to allow for field extensions and  correspondences using Theorem \ref{DA=DM} as well as to invert the map defined by multiplication by $p$. %

 The following fact is  a straightforward corollary of Theorem \ref{premain}.
 
 \begin{prop}\label{imageofB1}
 Let $\car K=0$. The motive $\LL\iota_!\Lambda(\widehat{\B}^1)$ is isomorphic to $\Lambda$.
 \end{prop}
 
 \begin{proof} 
In order to prove the claim, it suffices to prove that $\LL\iota_!\Lambda(\widehat{\B}^1)\cong\Lambda(\B^1)$. This follows from Proposition \ref{B1perf} and the description of $\LL\iota_!$ given in Theorem \ref{premain}. 
 \end{proof}

We recall that all the homotopy categories we consider are monoidal (see \cite[Propositions 4.2.76 and 4.4.62]{ayoub-th2}), and the tensor product $\Lambda(X)\otimes\Lambda(X')$ of two motives associated to varieties $X$ and $Y$ coincides with $\Lambda(X\times X')$. The unit object is obviously the motive $\Lambda$. Due to the explicit description of the functor $\LL\iota_!$ we constructed above, it is easy to prove that it respects the monoidal structures.

\begin{prop}\label{istensorial}
Let $\car K=0$. The functor $\LL\iota_!\colon\wRigDA^{\eff}_{\et,\B^1}(K)\ra\RigDA^{\eff}_{\et}(K)$ is a monoidal functor.
\end{prop}

\begin{proof}
 Since $\LL\iota_!$ is the left adjoint of a monoidal functor $\LL \iota^*$ there is a canonical natural transformation of bifunctors $\LL\iota_!(M\otimes M')\ra \LL\iota_! M\otimes \LL\iota_! M'$. In order to prove it is an isomorphism, it suffices to check it on a set of generators of $\wRigDA^{\eff}_{\et,\B^1}$ such as motives of semi-perfectoid varieties $X=\varprojlim_hX_h$, $X'=\varprojlim_hX'_h$. Up to rescaling, we can suppose that $\LL\iota_!\Lambda(X)=\Lambda(X_0)$ and $\LL\iota_!\Lambda(X')=\Lambda(X'_0)$ by  Theorem \ref{premain}. 
 In this case, by  definition of the tensor product, we obtain the following isomorphisms
 \[\LL\iota_!(\Lambda(X)\otimes\Lambda(X'))\cong \LL\iota_!\Lambda(X\times X')\cong \Lambda(X_0\times X_0')\cong\Lambda(X_0)\otimes\Lambda(X'_0)\cong \LL\iota_!\Lambda(X)\otimes \LL\iota_!\Lambda(X')\] 
proving our claim.
\end{proof}

The following proposition can be considered to be a refinement of Theorem \ref{premain}.

\begin{prop}\label{islocal2}
Let  $\car K=0$.  The functor $\LL\iota_!$ factors through $\wRigDA^{\eff}_{\et,\B^1}\ra\wRigDA^{\eff}_{\et,\widehat{\B}^1}$ and  the image of the functor $\LL \iota^*\colon\RigDA^{\eff}_{\et}(K)\ra\wRigDA^{\eff}_{\et,\B^1}(K)$ lies in the subcategory of $\widehat{\B}^1$-local objects. In particular, the triangulated adjunction 
\[\adj{\LL\iota_!}{\wRigDA^{\eff}_{\et,{\B}^1}(K)}{\RigDA_{\et}^{\eff}(K)}{\LL\iota^*}\] 
restricts to a  triangulated adjunction 
\[\adj{\LL\iota_!}{\wRigDA^{\eff}_{\et,\widehat{\B}^1}(K)}{\RigDA^{\eff}_{\et}(K)}{\LL\iota^*}.\] 
\end{prop}

\begin{proof}
By Propositions \ref{imageofB1} and  \ref{istensorial}, $\LL\iota_!$ is a monoidal functor sending $\Lambda(\widehat{\B}^1)$ to $\Lambda$. This proves the first claim. 

From the adjunction $(\LL\iota_!,\LL\iota^*)$ we then obtain the following isomorphisms, for any $X$ in $\wRigSm^{\gc}$ and any $M$ in $\RigDA^{\eff}_{\et}(K)$:
\[
\begin{aligned}
\wRigDA^{\eff}_{\et,{\B}^1}(K)(\Lambda(X\times\widehat{\B}^1),\LL\iota^*M)\cong\RigDA^{\eff}_{\et}(K)(\LL\iota_!\Lambda(X)\otimes\Lambda,M)\\
\cong \RigDA^{\eff}_{\et}(K)(\LL\iota_!\Lambda(X),M)\cong\wRigDA^{\eff}_{\et,{\B}^1}(K)(\Lambda(X),\LL\iota^*M)
\end{aligned}
\]
proving the second claim.
\end{proof}

\begin{rmk}
In the statement of the proposition above, we make a slight abuse of notation when denoting with $(\LL\iota_!,\LL\iota^*)$ both adjoint pairs. It  will be clear from the context which one we consider at each instance.
\end{rmk}

\section{The de-perfectoidification functor in characteristic \texorpdfstring{$p$}{p}}\label{deperfp}

We now consider the case of a perfectoid field $K^\flat$ of characteristic $p$ and try to generalize the results of Section \ref{deperf0}. We will need to perform an extra localization on the  model structure, and in return we will prove a stronger result. In  this section, we always assume that the base perfectoid field has characteristic $p$. In order to emphasize this hypothesis, we   denote it with $K^\flat$. 

In positive characteristic, we are not able to prove Theorem \ref{DA=DM} as it is stated, and it is therefore not clear that the maps $X_{h+1}\ra X_h$ associated to an object $X=\varprojlim_hX_h$ of $\wRigSm$ are isomorphisms in $\RigDA^{\eff}_{\et}(K^\flat)$ for a sufficiently big $h$. In order to overcome this obstacle, we  localize our model category further.

For any variety $X$ over $K^\flat$ we denote by $X^{(1)}$ the pullback of $X$ over the Frobenius map $\Phi\colon K^\flat\ra K^\flat$, $x\mapsto x^p$. The absolute Frobenius morphism induces a $K^\flat$-linear map $X\ra X^{(1)}$. Since $K^\flat$ is perfect, we can also denote by $X^{(-1)}$ the  pullback of $X$ over the inverse of the Frobenius map $\Phi^{-1}\colon K^\flat\ra K^\flat$ and $X\cong(X^{(-1)})^{(1)}$. There is in particular a canonical map $X^{(-1)}\ra X$ which is isomorphic to the map $X'\ra X$ induced by the absolute Frobenius, where we denote by $X'$ the same variety $X$ endowed with the structure map $X\ra \Spa K\stackrel{\Phi}{\ra}{\Spa K}$.

\begin{prop}
The  model category $\Ch_{\et,\B^1}\Psh (\RigSm/K^\flat)$ admits a left Bousfield localization denoted by $\Ch_{\Frobet,\B^1}\Psh(\RigSm/K^\flat)$ with respect to  the set   $S_{\Frob}$ of relative Frobenius maps $\Phi\colon \Lambda(X^{(-1)})[i]\ra \Lambda(X)[i]$ as $X$ varies in $\RigSm$ and $i$ varies in $\Z$. 
\end{prop}

\begin{proof}
Since by  \cite[Proposition 4.4.31]{ayoub-th2} the $\tau$-localization coincides with the Bousfield localization with respect to a set, we conclude by \cite[Theorem 4.2.71]{ayoub-th2} that the  model category $\Ch_{\et,\B^1}\Psh (\RigSm/K^\flat)$ is still left proper and cellular.   
 We can then apply \cite[Theorem 4.1.1]{hirschhorn}.
\end{proof}

\begin{dfn}\label{DAfrobetK}
 We  denote by  $\RigDA^{\eff}_{\Frobet}(K^\flat,\Lambda)$ the  homotopy category associated to  $\Ch_{\Frobet,\B^1}\Psh(\RigSm/K^\flat)$. We  omit $\Lambda$ whenever the context allows it.
  The image of a rigid variety $X$ in this category is denoted by $\Lambda(X)$. 
\end{dfn}

The triangulated category $\RigDA^{\eff}_{\Frobet}(K)$ is canonically isomorphic to the full triangulated subcategory of $\RigDA^{\eff}_{\et}(K)$ formed by $\Frob$-local objects, i.e. objects that are local with respect to the maps in  $S_{\Frob}$. Modulo this identification,  there is an obvious functor ${\RigDA^{\eff}_{\et}}(K^\flat)\ra\RigDA^{\eff}_{\Frobet}(K^\flat)$ associating to $\mcF$ a $\Frob$-local object $C^{\Frob}\mcF$.

Inverting Frobenius morphisms is enough to obtain an analogue of Theorem \ref{DA=DM} in characteristic $p$. 

 \begin{thm}[{\cite{vezz-DADM}}]\label{DA=DMp}
 Let $\car K^\flat=p$. 
 The functors $(a_{tr}, o_{tr})$ induce an equivalence of triangulated categories:
\[
{\LL a_{tr}}\colon{\RigDA_{\Frobet}^{\eff}}(K^\flat)\cong{\RigDM^{\eff}_{\et}}(K^\flat).
\]
\end{thm}

\begin{rmk}\label{Qiscruc2}
 The proof of the statement above  uses in a crucial way the fact that the  ring of coefficients $\Lambda$ is a $\Q$-algebra. %
\end{rmk}

We now investigate the relations between the category $\RigDA_{\Frobet}^{\eff}(K^\flat)$ we have just defined, and the other categories of motives introduced so far.

\begin{prop}\label{2-6}
 Let $X_0$ be in $\RigSm/K^\flat$ endowed with an \'etale map $X_0\ra\T^N\times\T^M=\Spa(K^\flat\langle\underline{\upsilon}^{\pm1},\underline{\nu}^{\pm1}\rangle)$. The map $X_1=X_0\times_{\T^N}\T^N\langle\underline{\upsilon}^{\pm1/p}\rangle\ra X_0$ 
 is invertible in $\RigDA^{\eff}_{\Frobet}(K^\flat)$.
\end{prop}

\begin{proof}
The map of the claim 
is a factor of $X_0\times_{(\B^N\times\B^M)}(\B^N\langle\underline{\upsilon}^{1/p}\rangle\times\B^M\langle\underline{\nu}^{1/p}\rangle)\ra X_0$ which is isomorphic to the relative Frobenius map $X_0^{(-1)}\ra X_0$ (see for example \cite[Theorem 3.5.13]{gabberramero}). If we consider the diagram 
\[
X_{1}^{(-1)}\stackrel{a}{\ra} X_0^{(-1)}\stackrel{b}{\ra} X_{1}\stackrel{c}{\ra}X_0
\]
we conclude that the two compositions $ba$ and $cb$ are isomorphisms hence also $c$ is an isomorphism, as claimed. 
\end{proof}

 \begin{prop}\label{areb1local}
  The image via $\LL\iota^*$ of a $\Frob$-local object of $\RigDA^{\eff}_{\et}(K^\flat)$ is $\widehat{\B}^1$-local. In particular, the functor $\LL\iota^*$ restricts to a functor $\LL\iota^*\colon \RigDA^{\eff}_{\Frobet}(K^\flat)\ra\wRigDA^{\eff}_{\et,\widehat{\B}^1}(K^\flat)$. 
 \end{prop}
 
 \begin{proof}
  Let $X'=\varprojlim_hX'_h$ be in $\wRigSm^{\gc}$. We consider the object $X'\times\widehat{\B}^1=\varprojlim_h(X'_h\times X_h)$ where we use the description $\widehat{\B}^1=\varprojlim_hX_h$ of Proposition \ref{B1perf}. Let $M$ be a $\Frob$-local object of $\RigDA^{\eff}_{\et}(K^\flat)$. 
 From  Propositions \ref{colimok} and \ref{2-6} we then deduce the following isomorphisms
 \[
 \begin{aligned}
  &\wRigDA^{\eff}_{\et,\B^1}(K^\flat)(X'\times\widehat{\B}^1,\LL\iota^*M)\cong\varinjlim_h\RigDA^{\eff}_{\et}(K^\flat)(X'_h\times X_h,M)\\ &\cong\RigDA^{\eff}_{\et}(K^\flat)(X'_0\times\B^1,M)\cong\RigDA^{\eff}_{\et}(K^\flat)(X'_0,M)\\ &\cong\varinjlim_h\RigDA^{\eff}_{\et}(K^\flat)(X'_h,M)\cong  \wRigDA^{\eff}_{\et,\B^1}(K^\flat)(X',\LL\iota^*M)
 \end{aligned}
 \]
 proving the claim.
 \end{proof}

We remark that in positive characteristic 
  the perfection  $\Perf\colon X\mapsto \varprojlim X^{(-i)}$ is functorial. This makes the description of various functors a lot easier. 
  We recall that we denote by 
  \[
  \adj{\LL j^*}{\PerfDA^{\eff}_{\et}(K^\flat)}{\wRigDA^{\eff}_{\et,\widehat{\B}^1}(K)}{\RR j_*}
  \]
  the adjoint pair induced by the inclusion of categories $j\colon\PerfSm\ra\wRigSm$. 

\begin{prop}\label{perfisji}
 The perfection functor $\Perf\colon\wRigSm\ra\PerfSm$ induces an adjunction
\[
 \adj{\LL\Perf^*}{\wRigDA^{\eff}_{\et,\B^1}(K^\flat)}{\PerfDA^{\eff}_{\et}(K^\flat)}{\RR\Perf_*}
\]
and $\LL\Perf^*$ factors through $\wRigDA^{\eff}_{\et,{\B}^1}(K^\flat)\ra\wRigDA^{\eff}_{\et,\widehat{\B}^1}(K^\flat) $. 
Moreover, the functor $\LL\Perf^*$  coincides with $\RR j_*$ on $\wRigDA^{\eff}_{\et,\widehat{\B}^1}(K^\flat)$.

\end{prop}

\begin{proof}
 The perfection functor is continuous with respect to the \'etale topology 
and maps $\B^1$ and $\widehat{\B}^{1}$ to $\widehat{\B}^1$ hence the first claim.

We now consider the functors $j\colon\PerfSm\ra\wRigSm$ and $\Perf\colon\wRigSm\ra\PerfSm$. They induce two Quillen pairs $(j^*,j_*)$ and $(\Perf^*,\Perf_*)$ on the associated $(\et,\widehat{\B}^1)$-localized model categories of complexes. Since $\Perf$ is a right adjoint of $j$ we deduce that $\Perf^*$ is a right adjoint of $j^*$ and hence we obtain an isomorphism $j_*\cong\Perf^*$ which shows the second claim. 
\end{proof}

\begin{prop}\label{llperf}
 Let  $\Lambda$ be a $\Q$-algebra. The functor 
 \[
 \LL\Perf^*\LL\iota^*\colon\RigDA_{\et}^{\eff}(K^\flat)\ra\PerfDA_{\et}^{\eff}(K^\flat)
 \] 
 factors over $\RigDA^{\eff}_{\Frobet}(K^\flat)$ and 
is isomorphic to $\RR j_*\LL\iota^*C^{\Frob}$.
\end{prop}

\begin{proof}
 The first claim follows as the perfection of $X^{(-1)}$ is canonically isomorphic to the perfection of $X$ for any object $X$ in $\RigSm$.

The second part of the statement follows from the first claim and the commutativity of the following diagram, which   is ensured by Propositions \ref{areb1local} and \ref{perfisji}.
 $$\xymatrix{
  \RigDA_{\Frobet}^{\eff}(K^\flat)\ar [dd]\ar[r]^-{\LL\iota^*}&\wRigDA_{\et,\widehat{\B}^1}^{\eff}(K^\flat)\ar[dd]\ar[dr]^-{\RR j_*}
 \\
 &&\PerfDA^{\eff}_{\et}(K^\flat)\\
 \RigDA_{\et}^{\eff}(K^\flat)\ar[r]^-{\LL\iota^*}	& \wRigDA_{\et,{\B}^1}^{\eff}(K^\flat)	\ar[ur]^-{\LL\Perf^*} 
 }$$
\end{proof}

\begin{thm}\label{eqcharp}
  Let  $\Lambda$ be a $\Q$-algebra. The functor $${\LL\Perf^*}\colon\RigDA^{\eff}_{\Frobet}(K^\flat)\ra\PerfDA^{\eff}_{\et}(K^\flat)$$ defines a monoidal, triangulated equivalence of  categories.
\end{thm}

\begin{proof}
Let $X_0$ and $Y$ be objects of $\RigSm^{\gc}$. Suppose $X_0$ is endowed with an \'etale map over $\T^N$ which is a composition of finite \'etale maps and inclusions, and let $\widehat{X}$ be $\varprojlim_hX_h$. We can identify $\widehat{X}$ with $\Perf X_0$.  Since $C^{\Frob}\Lambda(Y)$ is $\Frob$-local, by Proposition \ref{2-6} the maps  
\[\RigDA^{\eff}_{\et}(K^\flat)(\Lambda(X_h),C^{\Frob}\Lambda(Y))\ra\RigDA^{\eff}_{\et}(K^\flat)(\Lambda(X_{h+1}),C^{\Frob}\Lambda(Y))\]
 are isomorphisms for all $h$. 
Using  Propositions \ref{colimok} and \ref{areb1local}, we obtain the following sequence of isomorphisms for any $n\in \Z$:
\[\begin{aligned}
&\RigDA^{\eff}_{\Frobet}(K^\flat)(\Lambda(X_0),\Lambda(Y)[n])\cong \RigDA^{\eff}_{\et}(K^\flat)(\Lambda(X_0),C^{\Frob}\Lambda(Y)[n]) \\ &\cong\varinjlim_h\RigDA^{\eff}_{\et}(K^\flat)(\Lambda(X_h),C^{\Frob}\Lambda(Y)[n])
\\&\cong \wRigDA^{\eff}_{\et,\B^1}(K^\flat)(\Lambda(\widehat{X}),\LL\iota^*C^{\Frob}\Lambda(Y)[n])\\ &\cong\wRigDA^{\eff}_{\et,\widehat{\B}^1}(K^\flat)(\Lambda(\widehat{X}),\LL\iota^*C^{\Frob}\Lambda(Y)[n])\\
&\cong\PerfDA^{\eff}_{\et}(K^\flat)(\Lambda(\widehat{X}),\RR j_*\LL\iota^*C^{\Frob}\Lambda(Y)[n])\\
&\cong\PerfDA^{\eff}_{\et}(K^\flat)(\LL\Perf^*(X_0),\LL\Perf^*(Y)[n]).
\end{aligned}
\]
where the last isomorphism follows from the identification $\widehat{X}\cong\Perf X_0$ and Proposition  \ref{llperf}. In particular, we deduce that the triangulated functor $\LL\Perf^*$ maps a set of compact generators to a set of compact generators (see Propositions \ref{genRigDA} and \ref{genPerfDA}) and on these objects it is  fully faithful. 
By means of \cite[Lemma 1.3.32]{ayoub-rig}, we then conclude it is a triangulated equivalence of categories, as claimed.
\end{proof}

\begin{rmk}
From the proof of the previous claim, we also deduce that the  inverse $\RR\Perf_{*}$ of $\LL\Perf^*$ sends the motive associated to an object $X=\varprojlim_hX_h$ to the motive of $X_0$. This functor is then analogous to the de-perfectoidification functor $\LL j^*\circ\LL\iota_!$ of  Theorem \ref{premain}.
\end{rmk}

\section{The main theorem}\label{mainthm}

Thanks to the results of the previous sections, we can reformulate Theorem \ref{premain} in terms of motives of rigid varieties. We will always assume that $\car K=0$ since the results of this section are tautological when $	\car K=p$. 

\begin{cor}\label{adjKKflat}
 There exists a triangulated adjunction of categories
\[
\adj{\mfF}{\RigDM^{\eff}_{\et}(K^\flat)}{\RigDM^{\eff}_{\et}(K)}{\mfG}
\]
such that $\mfF$ is a monoidal functor.
\end{cor}

\begin{proof}
From Theorem \ref{premain} and Proposition \ref{istensorial}, we can define an adjunction \[\adj{\mfF'}{\RigDA^{\eff}_{\Frobet}(K^\flat)}{\RigDA^{\eff}_{\et}(K)}{\mfG'}\]
 by putting $\mfF'\colonequals \LL\iota_!\circ\LL j^*\circ(-)^\sharp\circ\LL\Perf^*$. We remark that by Proposition \ref{istensorial}, $\mfF'$ is also monoidal. The claim then follows from the equivalence of motives with and without transfers (Theorems \ref{DA=DM} and \ref{DA=DMp}).
\end{proof}

Our goal is to prove that the adjunction of Corollary \ref{adjKKflat}  is an equivalence of categories. To this aim, we recall the construction of the stable versions of the rigid motivic categories given in \cite[Definition 2.5.27]{ayoub-rig}.

\begin{dfn}\label{RigDM}
Let $T$ be the cokernel in $\PST(\RigSm/K)$ of the unit map $\Lambda_{\tr}(K)\ra\Lambda_{\tr}(\T^1)$. We denote by $\RigDM_{\et}(K,\Lambda)$ or simply by $\RigDM_{\et}(K)$ the homotopy category of the stable $(\et,\B^1)$-local model structure on symmetric spectra $\SSpect_{T}^\Sigma(\Ch_{\et,\B^1}\PST(\RigSm/K))$. 
\end{dfn}

As explained in \cite[Section 2.5]{ayoub-rig}, $T$ is 
 cofibrant and the cyclic permutation induces the identity on $T^{\otimes3}$ in $\RigDM^{\eff}_{\et}$. 
Moreover, by \cite[Theorem 9.3]{hovey-sp}, $T\otimes-$ is a Quillen equivalence in this category, which is actually the universal model category where this holds (in some weak sense made precise by \cite[Theorem 5.1, Proposition 5.3 and Corollary 9.4]{hovey-sp}). We recall that the canonical functor $\RigDM^{\eff}_{\et}(K)\ra\RigDM_{\et}(K)$ is fully faithful, as proved in \cite[Corollary 2.5.49]{ayoub-rig} as a corollary of the Cancellation Theorem \cite[Theorem 2.5.38]{ayoub-rig}.

\begin{dfn}
We denote by $\Lambda(1)$ the motive $T[-1]$ in $\RigDM^{\eff}_{\et}(K)$. For any positive integer $d$ we let $\Lambda(d)$ be $\Lambda(1)^{\otimes d}$. The functor $(\cdot)(d)\colonequals(\cdot)\otimes\Lambda(d)$ is an auto-equivalence of $\RigDM_{\et}(K)$ and its inverse will be denoted with $(\cdot)(-d)$.
\end{dfn}

\begin{dfn}
We denote by $\RigDM_{\et}^{\ct}(K,\Lambda)$ or simply by $\RigDM_{\et}^{\ct}(K)$ the full triangulated subcategory of $\RigDM_{\et}(K,\Lambda)$ whose objects are the compact ones. They are of the form $M(d)$ for some compact object $M$ in $\RigDM^{\eff}_{\et}(K)$ and some $d$ in $\Z$. This category is called the category of \emph{constructible motives}.
\end{dfn}

We now present an important result that is a crucial step toward the proof of our main theorem. The motivic property it induces will be given right afterwards. 

\begin{prop}\label{approxtiltsimp}
Let $\widehat{X}$ be a smooth affinoid perfectoid. The natural map of complexes \[\Sing^{\widehat{\B}^{1}}(\Lambda(\widehat{\T}^{d}))(\widehat{X})\ra\Sing^{\widehat{\B}^{1}}(\Lambda(\T^d))(\widehat{X})\]
is a quasi-isomorphism.
 \end{prop}

\begin{proof}
   We let $\widehat{X}$ be $\Spa(R,R^+)$. A map $f$ in $\Hom(\widehat{X}\times\widehat{\B}^n,\T^d)$ [resp. in $\Hom(\widehat{X}\times\widehat{\B}^n,\widehat{\T}^d)$] corresponds to $d$ elements $f_1,\ldots,f_d$ in the group $(R^+\langle \tau_1^{1/p^\infty},\ldots,\tau_n^{1/p^\infty}\rangle)^{{\times}}$ [resp. in the group $(R^{\flat+}\langle \tau_1^{1/p^\infty},\ldots,\tau_n^{1/p^\infty}\rangle)^{{\times}}$] and the map between the two objects is induced by the multiplicative tilt map $R^{\flat+}\langle \tau_1^{1/p^\infty},\ldots,\tau_n^{1/p^\infty}\rangle\ra R^+\langle \tau_1^{1/p^\infty},\ldots,\tau_n^{1/p^\infty}\rangle$. 

We now present some facts about homotopy theory for cubical objects, which mirror classical results for simplicial objects (see for example \cite[Chapter IV]{may-sim}). 
We remark that the map of the statement is induced by a map of enriched cubical $\Lambda$-vector spaces (see \cite[Definition A.6]{ayoub-h1}), which is obtained by adding $\Lambda$-coefficients to a map of enriched cubical sets 
\[
 \Hom(\widehat{X}\times\widehat{\square},\widehat{\T}^d)\ra\Hom(\widehat{X}\times\widehat{\square},\T^d).
\]
 Any enriched cubical object has connections  in the sense of 
\cite[Section 1.2]{brownhiggins}, induced by the maps $m_i$ in \cite[Definition A.6]{ayoub-h1}. We recall that the category of cubical sets with connections can be endowed with a model structure by which all objects are cofibrant and weak equivalences are defined through the geometric realization (see \cite{jardine-cub}). Moreover, its homotopy category is canonically equivalent to the one of simplicial sets, as cubical sets with connections form a strict test category by \cite{maltsiniotis-cub}. 

The two cubical sets appearing above are abelian groups on each level and the maps defining their cubical structure are group homomorphisms. They therefore are cubical groups. By \cite{tonks}, they are fibrant objects and  their homotopy groups $\pi_i$ coincide with the homology $H_iN$ of the associated normalized complexes  of abelian groups (see Definition \ref{cocu}).  The $\Lambda$-enrichment functor is tensorial with respect to the monoidal structure of cubical sets introduced in \cite[Section 11.2]{BHS} and   the cubical Dold-Kan functor, associating to a cubical $\Lambda$-module with connection its normalized complex (see \cite[Section 14.8]{BHS}) is a left Quillen functor. We deduce that in order to prove the statement of the proposition it suffices to show that the  two normalized complexes of abelian groups are quasi-isomorphic. We also remark that it suffices to consider the case $d=1$.

We   prove the following  claim: the $n$-th homology of the complex $N((R\widehat{\otimes}\mcO(\widehat{\square}))^{+{\times}})$ is $0$ for $n>0$. Let $f$ be invertible in $R^+\langle{\tau}_1^{1/p^\infty},\ldots,{\tau}_n^{1/p^\infty}\rangle$ with $d_{r,\epsilon}f=1$ for all $(r,\epsilon)$. We claim that $f-1$ is topologically nilpotent. Up to adding a topological nilpotent element, we can assume that $f\in R^+[\underline{\tau}]$. 
Since $f$ is invertible, its image in $(R^+/R^{\circ\circ})[\underline{\tau}^{1/p^\infty}]$ is invertible as well. Invertible elements in this ring are just the invertible constants. We deduce that all  coefficients of  $f-f(0)=f-1$ are topologically nilpotent and hence $f-1$ is topologically nilpotent. 
In particular, the element  $H=f+\tau_{n+1}(1-f)$  in $R^+\langle\underline{\tau}^{1/p^\infty},\tau_{n+1}^{1/p^\infty}\rangle$ is invertible, satisfies $d_{r,\epsilon}H=1$ for all $\epsilon $ and all $1\leq r\leq n$ and determines a homotopy between $f$ and $1$. This proves the claim.

We can also prove that the $0$-th homology of  the complex $N((R\widehat{\otimes}\mcO(\widehat{\square}))^{+{\times}})$  coincides with $R^{+{\times}}/(1+R^{\circ\circ})$. This amounts to showing that the image of the ring map 
\[
\begin{aligned}
\{f\in R^+\langle\tau^{1/p^\infty}\rangle^{\times}\colon f(0)=1\}&\ra R^{+{\times}}
\\
f&\mapsto f(1)
\end{aligned}
\]
coincides with $1+R^{\circ\circ}$. 
Let $f$ be invertible in $R^+\langle\tau^{1/p^\infty}\rangle$ with $f(0)=1$. As proved above,  $f-1$ is topologically nilpotent so that also $f(1)-1$ is. Vice-versa  if $a\in R$ is topologically nilpotent then the element $1+a\tau\in R^+\langle\tau^{1/p^\infty}\rangle$ is invertible,  satisfies $f(0)=1$ and $f(1)=1+a$ proving the claim.

We are left to prove that the multiplicative map $\sharp$ induces an isomorphism $(R^{\flat+})^{\times}/(1+R^{\flat\circ\circ})\ra (R^{+})^{\times}/(1+R^{\circ\circ})$. We start by proving it is injective. Let $a\in R^{\flat+}$ such that $(a^\sharp-1)$ is topologically nilpotent. 
Since $(a^\sharp-1)=(a-1)^\sharp$ in $R^+/\pi$ we deduce that the element $ (a-1)^\sharp-(a^\sharp-1)$ is also topologically nilpotent. We conclude that $(a-1)^\sharp$ as well as $(a-1)$ are  topologically nilpotent, as wanted.

We now prove surjectivity. Let $a$ be invertible in $R^+$. In particular both $a$ and $a^{-1}$ are power-bounded. From the isomorphism $R^{\flat+}/\pi^\flat\cong R^+/\pi$ we deduce that there exists an element $b\in R^{\flat+}$ such that $b^\sharp=a+\pi\alpha=a(1+\pi\alpha a^{-1})$ for some (power bounded) element $\alpha\in R^+$. We deduce that $(1+\pi\alpha a^{-1})$ lies in $1+R^{\circ\circ}$ and that $b^\sharp$ is invertible. Since the multiplicative structure of $R^\flat$ is isomorphic to $\varprojlim_{x\mapsto x^p}R$ and $\sharp$ is given by the projection to the last component, we deduce that as $b^\sharp$ is invertible, then also $b$ is. In particular, the image of $b\in (R^{\flat+})^{\times}$ in $(R^+)^{\times}/(1+R^{\circ\circ})$ is equal to $a$ as wanted.
\end{proof}

We recall that by Corollary \ref{adjKKflat} there is an adjunction 
\[
\adj{\mfF}{\RigDM^{\eff}_{\et}(K^\flat)}{\RigDM^{\eff}_{\et}(K)}{\mfG}
\]
and our goal is to prove it is an equivalence.

\begin{prop}\label{G'}
 The motive $\mfG\Lambda(d)$ is isomorphic to $\Lambda(d)$ for any positive integer $d$.
\end{prop}

\begin{proof}
The natural map $\Lambda(d)\ra\mfG\Lambda(d)$ is induced by the isomorphism  $\mfF\Lambda(d)\cong\Lambda(d)$. We need to prove it is an isomorphism. The motive $\Lambda(d)$ is a direct factor of the motive $\Lambda(\T^d)[-d]$ and the map above is  induced by $\Lambda(\T^d)\ra\mfG\Lambda(\T^d)$. It suffices then to prove that the map $\Lambda(\T^d)\ra\mfG\Lambda(\T^d)$ is an isomorphism.

 By the definition of the adjoint pair $(\mfF,\mfG)$ given in Corollary \ref{adjKKflat}, we can equivalently consider the adjunction
\[
\adj{\LL\iota_!\LL j^*}{\PerfDA^{\eff}_{\et}(K)}{\RigDA^{\eff}_{\et}(K)}{\RR j_*\LL\iota^*}
\]
and prove that $\Lambda(\widehat{\T}^d)\ra(\RR j_*\circ\LL\iota^*)\Lambda(\T^d)$ is an isomorphism in $\PerfDA^{\eff}_{\et}(K)$.

From Proposition \ref{approxtiltsimp} we deduce that the complexes  $\Sing^{\widehat{\B}^{1}}\Lambda(\widehat{\T}^{d})$ and $j_*\Sing^{\widehat{\B}^{1}}\Lambda(\T^d)$ 
 are quasi-isomorphic in $\Ch\Psh(\PerfSm)$. Since $j_*$ commutes with $\Sing^{\widehat{\B}^1}$ and with $\et$-sheafification,  the quasi-isomorphism above can be restated as 
 \[
 \Sing^{\widehat{\B}^{1}}\Lambda(\widehat{\T}^{d})\cong\RR j_*\Sing^{\widehat{\B}^{1}}\Lambda(\T^d).
 \]
Due to Proposition \ref{singp} and the isomorphism  $\LL\iota^*\Lambda(\T^d)\cong\Lambda(\T^d)$ this implies
$
\Lambda(\widehat{\T}^{d})\cong\RR j_*\LL\iota^*\Lambda(\T^d)
$ 
as wanted.
\end{proof}

\begin{rmk}\label{Tvgc}
Along the proof of the previous proposition, we showed in particular that $\RR j_*\LL\iota^*\Lambda(\T^d)\cong\Lambda(\widehat{\T}^d)$.
\end{rmk}

\begin{rmk}\label{Rjsums}
 Since $j_*$ commutes with $\et$-sheafification, it preserves $\et$-weak equivalences. It also commutes with $\Sing^{\widehat{\B}^1}$ and therefore preserves $\B^1$-weak equivalences. We conclude that $\RR j_*=j_*$ and in particular $\RR j_*$ commutes with small direct sums.
\end{rmk}

\begin{dfn}\label{vgc}
A rigid analytic varieties with good coordinates $X_0\ra\T^N$ such that the induced maps $\Lambda(X_{h+1})\ra\Lambda(X_h)$ are invertible in $\RigDM^{\eff}_{\et}(K,\Q)$ is called a \emph{variety with very good coordinates}. 
\end{dfn}

We are finally ready to present the proof of our main result.

\begin{thm}\label{main}Let $K$ be a perfectoid field and $\Lambda$ be a $\Q$-algebra. 
 The adjunction
 \[\adj{\mfF}{\RigDM_{\et}^{\eff}(K^\flat)}{\RigDM_{\et}^{\eff}(K)}{\mfG}\]
 is a monoidal, triangulated equivalence of categories.
\end{thm}

Before the proof, we remark that by putting theorems \ref{tiltingeq} and \ref{eqcharp}  together, the theorem above has the following restatement:

\begin{thm}\label{mainperf}
	Let $K$ be a perfectoid field and $\Lambda$ be a $\Q$-algebra. The adjunction
	\[\adj{\LL\iota_!\LL j^*}{\PerfDA_{\et}^{\eff}(K)}{\RigDM_{\et}^{\eff}(K)}{\RR j_*\LL\iota^*}\]
	is a monoidal, triangulated equivalence of categories.
\end{thm}

\begin{proof}
By Theorem \ref{premain} the functor $\LL\iota_!\LL j^*\colon\PerfDA^{\eff}_{\et}(K)\ra\RigDA^{\eff}_{\et}(K)$ sends the motive $\Lambda(\widehat{X})$ associated to a perfectoid  $\widehat{X}=\varprojlim_hX_h$ to the motive $\Lambda(X_0)$ associated to $X_0$ up to rescaling indices. It is triangulated,  commutes with sums, and its essential image contains motives $\Lambda(X_0)$ of varieties $X_0$ having very good coordinates $X_0\ra\T^N$ (see Definition \ref{vgc}).  
By Proposition \ref{tower}, for every rigid variety with good coordinates $X_0\ra\T^N$ there exists an index $h$ such that $X_h=X_0\times_{\T^N}\T^N\langle\underline{\upsilon}^{\pm1/p^h}\rangle$ has very good coordinates. Since $\car K=0$  the map $\T^N\langle\underline{\upsilon}^{\pm1/p^h}\rangle\ra\T^N$ is finite \'etale, and therefore also the map $X_h\ra X_0$ is. We conclude that any rigid variety with good coordinates \ has a finite \'etale covering  with very good coordinates, and hence the motives associated to varieties with very good coordinates generate the \'etale topos. In particular, the motives associated to them generate $\RigDA^{\eff}_{\et}(K)$ and hence 
the functor $\LL\iota_!\circ\LL j^*$ maps a set of compact generators to a set of compact generators.

Since $\mfF$ is monoidal and $\mfF(\Lambda(1))=\Lambda(1)$ it extends formally to a monoidal functor from the category $\RigDA^{\ct}_{\et}(K^\flat)$ to $\RigDA_{\et}(K) $  by putting $\mfF(M(-d))=\mfF(M)(-d)$. 
Let now $M$, $N$ in $\RigDM_{\et}(K^\flat)$ be twists of the motives associated to the analytification of smooth projective varieties $X$ resp. $X'$. They are strongly dualizable objects of $\RigDM_{\et}(K^\flat)$ since $\Lambda_{\tr}(X)$ and $\Lambda_{\tr}(X')$ are strongly dualizable in $\DM_{\et}(K^\flat)$. Fix an integer $d$ such that $N^\vee(d)$ lies in $\RigDM^{\eff}_{\et}(K^\flat)$. The objects $M$, $N$, $M^\vee$ and $N^\vee$ lie in $\RigDM^{\ct}_{\et}(K^\flat)$  and moreover $\mfF(N^\vee)=\mfF(N)^\vee$. From Lemma \ref{Fff} we also deduce that the functor $\mfF$ induces a bijection
\[
 \RigDM^{\ct}_{\et}(K^\flat)(M\otimes N^\vee,\Lambda)\cong \RigDM_{\et}(K)(\mfF(M)\otimes \mfF(N)^\vee,\Lambda).
\]
By means of the Cancellation theorem \cite[Corollary 2.5.49]{ayoub-rig} the first set is isomorphic to  $\RigDM_{\et}^{\eff}(K^\flat)(M,N)$ and the second is isomorphic to $\RigDM_{\et}^{\eff}(K)(\mfF(M), \mfF(N))$. We then deduce that 
 all motives $M$ associated to the analytification of smooth projective varieties lie in the left orthogonal of the cone of the map $N\ra\mfG\mfF N$ which is closed under direct sums and cones. Since $\Lambda$ is a $\Q$-algebra, such motives  generate $\RigDM_{\et}^{\eff}(K^\flat)$ by means of \cite[Theorem 2.5.35]{ayoub-rig}.  We conclude that   $N\cong\mfG\mfF N$. Therefore the category $\catT$ of objects $N$ such that  $N\cong\mfG\mfF N$ contains  all motives  associated to the analytification of smooth projective varieties. It is clear that $\catT$ is closed under cones. The functors $\mfF$ and $\LL\iota^*$ commute with direct sums  as they are left adjoint functors.  As pointed out in Remark \ref{Rjsums} also the functor $\RR j_*$ does. Since $\mfG$ is a composite of $\RR j_*\LL\iota^*$ with equivalences of categories, it commutes with small sums as well. We conclude that $\catT$ is closed under direct sums. Using again \cite[Theorem 2.5.35]{ayoub-rig} we deduce $\catT=\RigDM_{\et}^{\eff}(K^\flat)$ proving that $\mfF$ is fully faithful. This is enough to prove that it is an equivalence of categories, by applying \cite[Lemma 1.3.32]{ayoub-rig}.
\end{proof}

\begin{lemma}\label{Fff}
 Let $M$ be an object of $ \RigDA^{\ct}_{\et}(K^\flat)$. The functor $\mfF$ induces an isomorphism
\[
  \RigDM^{\ct}_{\et}(K^\flat)(M,\Lambda)\cong  \RigDM_{\et}(K)(\mfF(M),\Lambda).
\]
\end{lemma}

\begin{proof}
Suppose that $d$ is an integer such that $M(d)$ lies in $\RigDA^{\eff}_{\et}(K^\flat)$.   One has   $\mfF\Lambda(d)\cong\Lambda(d)$ and by  Proposition \ref{G'} 
the unit map $\eta\colon\Lambda(d)\ra\mfG\mfF\Lambda(d)$ is an isomorphism. 
In particular from the adjunction $(\mfF,\mfG)$ we obtain a commutative square
$$\xymatrix{
\RigDM^{\eff}_{\et}(K^\flat)(M(d),\Lambda(d)) \ar[r]^-{\mfF}\ar^-{=}[d] &	\RigDM_{\et}^{\eff}(K)(\mfF M(d),\mfF\Lambda(d))\ar[d]^-{\sim}\\
\RigDM^{\eff}_{\et}(K^\flat)(M(d),\Lambda(d)) \ar[r]^-{\eta}_-{\sim}	&	\RigDM_{\et}^{\eff}(K)(M(d),(\mfG\mfF)\Lambda(d))
}$$
in which the top arrow is then an isomorphism. By the Cancellation theorem \cite[Corollary 2.5.49]{ayoub-rig} we also obtain the following commutative square
$$\xymatrix{
\RigDM^{\ct}_{\et}(K^\flat)(M(d),\Lambda(d)) \ar[r]^-{\mfF}\ar^-{\sim}[d] &	\RigDM_{\et}(K)(\mfF M(d),\Lambda(d))\ar[d]^-{\sim}\\
\RigDM^{\eff}_{\et}(K^\flat)(M(d),\Lambda(d)) \ar[r]^-{\mfF}_-{\sim}	&	\RigDM_{\et}^{\eff}(K)(\mfF M(d),\Lambda(d))
}$$
and hence also the top arrow is an isomorphism. We conclude the claim from the following commutative square, whose vertical arrows are isomorphisms since the functor 
 $(\cdot)(d)$ is invertible in $\RigDM^{}_{\et}(K)$:
 $$\xymatrix{
 \RigDM^{\ct}_{\et}(K^\flat)(M,\Lambda) \ar[r]^-{\mfF}\ar[d]^-{\sim}_-{(\cdot)(d)} &	\RigDM_{\et}(K)(\mfF M,\Lambda)\ar[d]^-{\sim}_-{(\cdot)(d)}\\
 \RigDM^{\ct}_{\et}(K^\flat)(M(d),\Lambda(d)) \ar[r]^-{\mfF}_-{\sim}	&	\RigDM_{\et}(K)(\mfF M(d),\Lambda(d)).
}$$
\end{proof}

\begin{rmk}\label{Qiscruc3}
 In the proof of Theorem \ref{main} we again used the hypothesis that $\Lambda$ is a $\Q$-algebra in order to  apply  \cite[Theorem 2.5.35]{ayoub-rig} which states that the motives associated to the analytification of smooth projective varieties generate $\RigDM_{\et}^{\eff}(K^\flat)$.
\end{rmk}

\begin{rmk}\label{gcok}
In the proof, we also showed that the motives $\Lambda(X)[i]$ where $X$ has very good coordinates and $i\in\Z$ generate $\RigDM^{\eff}_{\et}(K)$ as a triangulated category with small sums.
\end{rmk}

\begin{cor}\label{cofreplvgc}
Let $X_0\ra\T^N$ be a variety with very good coordinates. Then $\RR j_*\LL\iota^*\Lambda(X_0)\cong\Lambda(\varprojlim_hX_h)$. 
\end{cor}

\begin{proof}
By  the description given in Theorem \ref{premain}, we conclude that $\RR\iota_!\LL j^*(\Lambda(\varprojlim_hX_h))\cong\Lambda(X_{0})$. The claim then follows from Theorem \ref{main}, which shows that $\RR\iota_!\circ\LL j^* $ is a quasi-inverse of $\RR j_*\circ\LL\iota^*$.
\end{proof}

We remark that the  proof of Theorem \ref{main} also induces the  following result.

\begin{cor}\label{mainct}
 The functor
 \[
 \begin{aligned}
{\mfF}\colon{\RigDM_{\et}^{\ct}(K^\flat)}&\ra{\RigDM_{\et}^{\ct}(K)}\\
M(d)&\mapsto(\mfF M)(d)
 \end{aligned}
 \]
 is a monoidal equivalence of categories.
\end{cor}

We can refine the previous corollary by stating the stable version of our main result (Theorem \ref{mainstable}) which is based on the following intermediate results. For the definitions of spectra of model categories, we refer to \cite[Section 4.3]{ayoub-th2} and \cite{hovey-sp}.

\begin{dfn}
Let $T$ be the cokernel of the unit map $\Lambda\ra\Lambda(\T^1)$ in $\Psh(\RigSm/K)$ and let $\widehat{T}$ be the cokernel of the unit map $\Lambda\ra\Lambda(\widehat{\T}^1)$ in $\Psh(\PerfSm/K)$. They are direct factors of cofibrant object hence cofibrant.
\begin{enumerate}
\item We denote by $\RigDA_{\et}(K,\Lambda)$ the homotopy category of the stable $(\et,\B^1)$-local model structure on spectra $\SSpect_{T}(\Ch_{\et,\B^1}\Psh(\RigSm/K))$. We omit $\Lambda$ whenever the context allows it.
\item  We denote by $\wRigDA_{\et,\widehat{\B^1}}(K,\Lambda)$ the homotopy category of the stable $(\et,\widehat{\B}^1)$-local model structure on spectra $\SSpect_{\iota^*T}(\Ch_{\et,\widehat{\B}^1}\Psh(\wRigSm/K))$. We omit $\Lambda$ whenever the context allows it.
\item We denote by $\PerfDA_{\et}(K,\Lambda)$ the homotopy category of the stable $(\et,\widehat{\B}^1)$-local model structure on spectra $\SSpect_{\widehat{T}}(\Ch_{\et,\B^1}\Psh(\PerfSm/K))$. We omit $\Lambda$ whenever the context allows it.
\item For $\cat$ equal to $\RigSm$, $\wRigSm$ or $\PerfSm$, we denote by $\Sus_k$ the $k$-th \emph{suspension functor} from the model category of complexes of presheaves on $\cat$ to the associated spectra, which is the left  adjoint of the functor
\[\Ev_k\colon(M_n)_{n\in\N}\mapsto M_k.\]
\end{enumerate}
\end{dfn}

\begin{rmk}
The functor $-\otimes T$ [resp. $-\otimes\iota^*T$ resp. $-\otimes\widehat{T}$] has a prolongation to a left Quillen endofunctor on the associated spectra, which is furthermore a Quillen equivalence. The category of spectra is the universal model category with this property, in a weak sense made precise by \cite[Theorem 5.1 and Proposition 5.3]{hovey-sp}. %
\end{rmk}

\begin{rmk}\label{symmnon}
In contrast with Definition \ref{RigDM}, we use above the categories of \emph{non-symmetric} spectra. On the other hand, we remark that by  \cite[Proposition 4.3.47]{ayoub-th2} the model categories of symmetric and non-symmetric spectra are Quillen equivalent. We then conclude by  Proposition \ref{tiltingeq}, Theorem \ref{DA=DMp}, Proposition \ref{eqcharp} and  \cite[Theorem 5.5]{hovey-sp} that the canonical adjunctions 
$${\RigDA_{\et}(K,\Lambda)}\leftrightarrows{\RigDM_{\et}(K,\Lambda)}{}$$
$${\RigDM_{\et}(K^\flat,\Lambda)}\leftrightarrows{\PerfDA_{\et}(K,\Lambda)}{} $$
are equivalences of categories, and the prolongation of $-\otimes T$ [resp. $-\otimes\widehat{T}$] corresponds to the functor $M\mapsto M(1)$.
\end{rmk}

\begin{rmk}\label{Susts}
We point out that the functor $\Sus_0$ has an explicit description. For example, in the case of $\Ch_{\et,\B^1}\Psh(\RigSm/K)$ it sends a complex $\mcF$ to the spectrum $(\mcF\otimes T^{\otimes n})_{n\in\N}$ with the obvious transition maps. The functor $\Sus_k$ is the composition $t^k\Sus_0$ where $t$ is the \emph{shift functor} such that $\Ev_0 tM=0$ and  $\Ev_{k+1}tM=\Ev_kM$ for any spectrum $M$. 
The pair $(\Sus_k,\Ev_k)$ is a Quillen adjunction (see \cite[Lemma 4.3.24]{ayoub-th2}) as well as the pair $(t,s)$ where $s$ is such that $\Ev_{k}sM=\Ev_{k+1}M$ for any spectrum $M$ (see \cite[Definition 3.7]{hovey-sp}). 
\end{rmk}

\begin{dfn}
We denote by  
\[\adj{\Spect\iota^*}{\SSpect_{T}(\Ch_{\et,{\B}^1}\Psh(\RigSm))}{\SSpect_{\iota^*T}(\Ch_{\et,\widehat{\B}^1}\Psh(\wRigSm))}{\Spect\iota_*}\]
the Quillen adjunction induced by the pair $(\iota^*,\iota_*)$ by means of \cite[Proposition 5.3]{hovey-sp}.
\end{dfn}

The natural map $\widehat{\T}^1\ra\iota^*\T^1$ induces a map $\widehat{T}\ra j_*\iota^*T $ which in turn defines a natural transformation $\tau\colon j^*(-\otimes\widehat{T})\ra(j^*-)\otimes \iota^*T$. Nevertheless, it is not clear that for a smooth affinoid perfectoid  $\widehat{X}$ the map $\tau\colon j^*(\Lambda(\widehat{X})\otimes\widehat{T})\ra(j^*\Lambda(\widehat{X}))\otimes \iota^*T$  is a weak equivalence, so that
we can not apply the criterion of \cite[Proposition 5.3]{hovey-sp} to define a Quillen adjunction between $\SSpect_{\widehat{T}}(\Ch_{\et,\widehat{\B}^1}\Psh(\PerfSm/K))$ and  $\SSpect_{\iota^*T}(\Ch_{\et,\widehat{\B}^1}\Psh(\wRigSm/K))$. Therefore the following proposition, albeit easy to prove, has  a non-trivial content.

\begin{prop}\label{Spectj}
The functor
\[
\begin{aligned}
\Spect j_*\colon \SSpect_{\iota^*T}(\Ch_{\et,\widehat{\B}^1}\Psh(\wRigSm/K))&\ra \SSpect_{\widehat{T}}(\Ch_{\et,\widehat{\B}^1}\Psh(\PerfSm/K))\\
(M_n)_{n\in\N}&\mapsto(j_*M_n)_{n\in\N}
\end{aligned}
\] 
is well defined, preserves stable weak equivalences and induces a triangulated functor
\[
\RR\Spect j_*\colon\wRigDA_{\et,\widehat{\B}^1}(K,\Lambda){\ra} \PerfDA_{\et}(K,\Lambda).
\]
\end{prop}

\begin{proof}
If $(M_n)$ is a spectrum we can define the transition maps $j_*M_{n}\otimes \widehat{T}\ra j_*M_{n+1}$ with the following composition 
\[
j_*M_{n}\otimes \widehat{T}\stackrel{\tau}{\ra} j_*M_{n}\otimes j_*\iota^*{T}\cong j_*(M_n\otimes \iota^*T)\ra j_* M_{n+1}
\]
deduced by the transition maps  $M_n\otimes \iota^*T\ra  M_{n+1}$. Moreover,  as shown in the first part of the proof of \cite[Proposition 5.3]{hovey-sp}  $\Spect j_*$ is a right Quillen functor with respect to the projective model structures on spectra, whose weak equivalences are level-wise weak equivalences. It also preserves such weak equivalences,  as proved in Remark \ref{Rjsums}.

The stable model structure on spectra is obtained as a left Bousfield localization with respect to the maps $\Sus_{n+1}(\mcF\otimes \iota^*T)\ra\Sus_n\mcF$ [resp. $\Sus_{n+1}(\mcF\otimes\widehat{T})\ra\Sus_n\mcF$] as $\mcF$ runs among cofibrant objects (see for example \cite[Definition 4.3.29]{ayoub-th2}). We now prove that $\Spect j_*$ also preserves stable weak equivalences.

Consider  the natural map \[\Ev_k(\Spect j_*\Sus_{n+1} (\mcF\otimes\iota^*T))\ra \Ev_k(\Spect j_*\Sus_n\mcF).\] 
For $k>n$ it is an equality, hence the map 
\[\Spect j_*\Sus_{n+1}(\mcF\otimes\iota^*T)\ra \Spect j_*\Sus_n\mcF.\]
 is a stable weak equivalence by \cite[Lemma 4.3.59]{ayoub-th2}. We then deduce that $\Spect j_*$  preserves stable weak equivalences as claimed.

The fact that $\RR\Spect j_*$ is triangulated follows from the fact that it coincides with the restriction to stably local spectra of the functor induced by $\Spect j_*$ 
on the homotopy categories of the  projective model structures on spectra, which is triangulated being a right derived  Quillen functor.
\end{proof}

\begin{rmk}\label{RSpectjsums}
The functor $\RR\Spect j_*$ commutes with small direct sums by its explicit description and Remark \ref{Rjsums}. 
\end{rmk}

\begin{prop}\label{Sus}
There is an invertible natural transformation of  functors from the category $\RigDA_{\et}^{\eff}(K,\Lambda)$ to $\PerfDA_{\et}(K,\Lambda)$:
\[\eta\colon\LL\Sus_0\circ \RR j_*\circ\LL\iota^*\xrightarrow{\sim} \RR\Spect j_*\circ\LL\Spect\iota^*\circ\LL\Sus_0.\]
\end{prop}

\begin{proof} For any cofibrant object $\mcF$ in $\Ch\Psh(\RigSm)$ one has
\[
\Ev_k(\LL\Sus_0\circ \RR j_*\circ\LL\iota^*)(\mcF)\cong Q(j_*\iota^*\mcF)\otimes\widehat{T}^{\otimes k}
\]
\[
\Ev_k( \RR\Spect j_*\circ\LL\Spect\iota^*\circ\LL\Sus_0)(\mcF)\cong j_*\iota^*\mcF\otimes j_*(\iota^*T)^{\otimes k}
\]
where $Q$ is a cofibrant-replacement functor on $\Ch\Psh(\wRigSm)$. 
The natural transformation $\eta$ of the statement is then induced by the canonical maps $Q(j_*\iota^*\mcF)\ra j_*\iota^*\mcF$ and $\widehat{T}\ra j_*\iota^*T$. 

The two functors of the statement are triangulated and commute with small direct sums by Remarks \ref{Rjsums}, \ref{RSpectjsums} and Proposition \ref{Spectj}. We deduce that the subcategory of  $\RigDA_{\et}^{\eff}(K,\Lambda)$ on which $\eta$ is invertible is triangulated and closed under direct sums. 

By Proposition \ref{genRigDA} and Remark \ref{gcok} we conclude that it suffices to prove that $\eta$ is invertible on a fixed motive $\Lambda(X)[i]$ with $X$ with very good coordinates and $i\in\Z$. For such an object, using Corollary \ref{cofreplvgc}, we have an explicit description of a cofibrant replacement of $j_*\iota^*\Lambda(X)[i]$ namely $\Lambda(\widehat{X})[i]$ where $\widehat{X}=\varprojlim_hX_h$ is built with respect to some very good coordinates $X\ra\T^N$ of $X$. We are left to prove that the maps $\Lambda(\widehat{X})\otimes\widehat{T}^{\otimes k}[i]\ra j_*\iota^*\Lambda(X)\otimes j_*(\iota^*T)^{\otimes k}[i]$
induce a weak equivalence of spectra. We claim that each one of them is a weak equivalence.

The object $T$ [resp. $\widehat{T}$] is a direct summand of $\Lambda(\T^1)$ [resp. $\Lambda(\widehat{\T}^1)$] and the decomposition is compatible with the map $\widehat{T}\ra j_*\iota^*T$. Therefore, in order to prove the claim we can show that the maps
\[\Lambda(\widehat{X}\times\widehat{\T}^k)\cong\Lambda(\widehat{X})\otimes\Lambda(\widehat{\T}^{ k})\ra j_*\iota^*\Lambda(X)\otimes j_*\Lambda(\T^k)\cong j_*\iota^*\Lambda(X\times\T^k)\]
are weak equivalences, and this follows from Corollary \ref{cofreplvgc} and Remark \ref{Tvgc}.
\end{proof}

We define the following composite functor (see Remark \ref{symmnon}):
\[\mfG^{\st}\colon \RigDM_{\et}(K)\cong\RigDA_{\et}(K){\xrightarrow{\RR\Spect j_*\circ\LL\Spect\iota^*}}\PerfDA_{\et}(K)\cong\RigDM_{\et}(K^\flat)\]
which can be inserted in the diagram
$$\xymatrix{
\RigDM^{\eff}_{\et}(K,\Lambda)\ar[d]\ar[r]^{\mfG}_{\sim}	&	\RigDM^{\eff}_{\et}(K^\flat,\Lambda)\ar[d] \\
\RigDM_{\et}(K,\Lambda)\ar[r]^{\mfG^{\st}}	&	\RigDM_{\et}(K^\flat,\Lambda)
}$$
whose vertical functors are fully faithful by the Cancellation Theorem  \cite[Corollary 2.5.49]{ayoub-rig}. It is commutative, up to a natural transformation, by means of Proposition \ref{Sus}.

\begin{prop}\label{ctgen}
The category $\RigDM_{\et}(K,\Lambda)$ [resp. $\RigDM_{\et}(K,\Lambda)$] is generated, as a triangulated category with small sums, by the objects of $\RigDM^{\ct}_{\et}(K,\Lambda)$ [resp. $\RigDM^{\ct}_{\et}(K^\flat,\Lambda)$]
\end{prop}

\begin{proof}
We prove the statement only for $K$. 
From the Cancellation theorem \cite[Corollary 2.5.49]{ayoub-rig} we deduce that any object $M$ of $\RigDM_{\et}(K,\Lambda)$ is isomorphic to 
$\hocolim_{n}M_n(-n)$ with $M_n$ in $\RigDM_{\et}^{\eff}(K,\Lambda)$ (see the proof of  \cite[Proposition 7.4.1]{kahn-ff}). In particular, it sits in a distinguished triangle
$$\xymatrix{
\bigoplus M_n(-n)\ar[r] &\bigoplus M_n(-n)\ar[r]& M\ar[r]&
}$$
Since $\RigDM^{\eff}_{\eff}(K,\Lambda)$ is generated by compact objects by Proposition \ref{genRigDA}, we conclude that each $M_n(-n)$ lies in the triangulated category with small sums generated by the objects of $\RigDM^{\ct}_{\et}(K,\Lambda)$ so that also $M$ does. 
\end{proof}

We can finally prove the stable version of the main result of this article.

\begin{thm}\label{mainstable}
The functor
\[\mfG^{\st}\colon \RigDM_{\et}(K,\Lambda)\ra \RigDM_{\et}(K^\flat,\Lambda)\]
is a monoidal triangulated equivalence of categories and, up to a natural transformation, it restricts to the equivalence
\[\mfG\colon \RigDM_{\et}^{\eff}(K,\Lambda)\ra \RigDM_{\et}^{\eff}(K^\flat,\Lambda) \] 
of Theorem \ref{main} on the  subcategories of effective motives.
\end{thm}

\begin{proof}
By its definition, the functor   $\RR\Spect j_*$ commutes with the derived shift functor 
$\LL t$ (see Remark \ref{Susts}). We claim that also $\LL\Spect\iota^*$ does. We can equivalently show that $\RR\Spect\iota_*$ commutes with $\RR s$ and this again follows from their definitions (see Remark \ref{Susts} and \cite[Theorem 5.3]{hovey-sp}). 

By \cite[Theorem 3.8]{hovey-sp} the functor $\LL t$ on $\RigDA_{\et}(K,\Lambda)$ [resp. on $\PerfDA_{\et}(K,\Lambda) $] is a quasi-inverse of the prolongation of $\LL(-\otimes T)$ [resp. of $\LL(-\otimes\widehat{T})$]. Therefore,  we conclude that for any $M$ in $\RigDA_{\et}(K,\Lambda)$:
\[
(\RR\Spect j_*\circ\LL\Spect\iota^*)(\LL(-\otimes{T}))^{-1}M\cong  (\LL(-\otimes\widehat{T}))^{-1}(\RR\Spect j_*\circ\LL\Spect\iota^*)M
\]
which implies for any $M$ in $\RigDM_{\et}(K,\Lambda)$ the following canonical isomorphism:
\[\mfG^{\st}(M(-1))\cong(\mfG^{\st} M)(-1).\]

Since we already showed in Proposition \ref{Sus} that $\mfG^{\st}$ restricts to $\mfG$ on effective motives,  we conclude that it restricts  on $\RigDM^{\ct}_{\et}(K,\Lambda)$ to a quasi-inverse of the functor $\mfF$ of Corollary \ref{mainct}. In particular, this restriction is  fully faithful and its essential image contains $\RigDM^{\ct}_{\et}(K^\flat,\Lambda)$. 
By Remark \ref{RSpectjsums} we also deduce that $\mfG^{\st}$ commutes with small direct sums. The claim of the theorem then follows from Proposition \ref{ctgen} and \cite[Lemma 1.3.32]{ayoub-rig}. 
\end{proof}

\begin{rmk}
It is worth noticing that along the proof of our main theorem, we have not used any result on almost algebra (which nonetheless has a critical role in the theory of perfectoid spaces).
\end{rmk}

\begin{rmk}\label{finalrmk}
The reader may wonder if the equivalence of categories    $\RigDM_{\et}(K,\Lambda)\cong \RigDM_{\et}(K^\flat,\Lambda)$ still holds true for more general rings of coefficients $\Lambda$. The hypothesis $\Q\subset\Lambda$ has been used several times along the proof of the main statement, and most crucially in the following two instances. %
 First, it was used to invoke the results of \cite{vezz-DADM} on the equivalence of motives with and without transfers (see Remarks \ref{Qiscruc1} and \ref{Qiscruc2}). %
 Secondly, the hypothesis $\Q\subset\Lambda$ was used in order to apply the result of Ayoub \cite[Theorem 2.5.35]{ayoub-rig} about a generating set for the triangulated category $\RigDM(K^\flat,\Lambda)$ (see Remark \ref{Qiscruc3}). %

 On the other hand, we conjecture that the tilting equivalence  $\RigDM_{\et}(K,\Lambda)\cong \RigDM_{\et}(K^\flat,\Lambda)$  can be generalized to  ring of coefficients $\Lambda$ over $\Z[\frac{1}{p}]$. As a matter of fact, for any prime $\ell\neq p$ we expect the category $\RigDM_{\et}(K,\Z/\ell)$ to be  equivalent to the derived category of Galois representations over $\Z/\ell$, similarly to the case of $\DM_{\et}^{\eff}(K,\Z/\ell)$ (see \cite[Theorem 9.35]{mvw}). The tilting equivalence in the case $\Lambda=\Z/\ell$  would then follow  from the classic theorem of Fontaine and Wintenberger. 
\end{rmk}

\begin{appendix}

\section{An implicit function theorem and approximation results}\label{approx}

The aim of this appendix is to prove Proposition \ref{solleva2text} which will be obtained as a corollary of several intermediate  approximation results   
for maps defined from  objects of $\wRigSm^{\gc}$ to rigid analytic varieties. 

Along this section, we assume that $K$ is a complete non-archimedean field. 
 We begin our analysis with the analogue of the inverse mapping theorem, which is a variant of \cite[Theorem 2.1.1]{igusa}. 

\begin{prop}
\label{implicit}
Let $R$ be a   $K$-algebra, let $ {\sigma}=(\sigma_1,\ldots,\sigma_n)$ and $ {\tau}=(\tau_1,\ldots,\tau_m)$ be two systems of coordinates and let $ {P}=(P_1,\ldots,P_m)$ be a collection of polynomials in $R[ {\sigma}, {\tau}]$ such that $ {P}( {\sigma}=0, {\tau}=0)=0$ and $\det(\frac{\del P_i}{\del \tau_j})( {\sigma}=0, {\tau}=0)\in R^\times$. There exists a unique collection $ {F}=(F_1,\ldots,F_m)$ of $m$ formal power series in $R[[ {\sigma}]]$ such that $ {F}( {\sigma}=0)=0$ and $ {P}( {\sigma}, {F}( {\sigma}))=0$ in $R[[ {\sigma}]]$. 
 
Moreover, if $R$ is a  Banach $K$-algebra, then  $F_1,\ldots,F_n$ have a positive radius of convergence.
\end{prop}

\begin{proof}
 Let $f$ be the polynomial $\det(\frac{\del P_i}{\del \tau_j})$ in $R[\sigma,\tau]$ and let $S$ be the ring $ R[ {\sigma}, {\tau}]_f/( {P})$. The induced map $R[ {\sigma}]\ra S$ is \'etale, and from the hypothesis $f(0,0)\in R^\times$ we conclude that the map $R[ {\sigma}, {\tau}]/( {P})\ra R$, $( {\sigma}, {\tau})\mapsto0$ factors through $S$. 

Suppose given a factorization as $R[\sigma]$-algebras $S\ra R[\sigma]/(\sigma)^n\ra R$ of the map $S\ra R$. By the  \'etale lifting property (see \cite[Definition IV.17.1.1 and Corollary IV.17.6.2]{EGAIV4}) applied to the square
$$\xymatrix{
R[ {\sigma}] \ar[d]\ar[r] & R[ {\sigma}]/( {\sigma})^{n+1}\ar[d]\\
S\ar[r]\ar@{.>}[ur]^{\exists!}& R[ {\sigma}]/( {\sigma})^n
}$$
we obtain a uniquely defined $R[ {\sigma}]$-linear map $S\ra R[ {\sigma}]/( {\sigma})^{n+1}$ factoring $S\ra R$ and hence by induction a uniquely defined $R[ {\sigma}]$-linear map $ R[ {\sigma}, {\tau}]/( {P})\ra R[[\sigma]]$ factoring $R[ {\sigma}, {\tau}]/( {P})\ra R$ as wanted. The power series $F_i$ is the image of $\tau_i$ via this map.

Assume now that $R$ is a Banach $K$-algebra. We want to prove that the array $F=(F_1,\ldots,F_m)$ of formal power series in $R[[ {\sigma}]]$ constructed above is convergent around $0$. As $R$ is complete, this amounts to proving estimates on the valuation of the coefficients of $ {F}$. To this aim, we now try to give an explicit description of them, depending on the coefficients of $ {P}$.  Whenever $I$ is a $n$-multi-index $I=(i_1,\ldots,i_n)$ we denote by $ {\sigma}^I$ the product $\sigma_1^{i_1}\cdot\ldots\cdot \sigma_n^{i_n}$ and we adopt the analogous notation for $\tau$.

We remark that the claim is not affected by any invertible $R$-linear transformation of the polynomials $P_i$. Therefore, by multiplying the column vector $P$ by the matrix $(\frac{\partial P_i}{\partial \tau_j})(0,0)^{-1}$ we reduce to the case in which $(\frac{\partial P_i}{\partial \tau_j})(0,0)=\delta_{ij}$. 
We can then write the polynomials $P_i$ in the following form:
\[
 P_i(\sigma,\tau)=\tau_i-\sum_{|J|+|H|>0}c_{iJH}\sigma^J\tau^H
\]
where $J$ is an $n$-multi-index, $H$ is an $m$-multi-index and the coefficients $c_{iJH}$ equal $0$ whenever $|J|=0$ and $|H|=1$.

We will determine the functions $F_i(\sigma)$ explicitly. We start by writing them as 
\[
 F_i(\sigma)=\sum_{|I|>0}d_{iI}\sigma^I
\]
with unknown coefficients $d_{iI}$ for any $n$-multi-index $I$. We denote their $q$-homogeneous parts by 
\[
 F_{iq}(\sigma)\colonequals \sum_{|I|=q}d_{iI}\sigma^I.
\]
We need to solve the equation $P(\sigma,F(\sigma))=0$ which can be rewritten as
\[
 F_i(\sigma) = \sum_{J,H} c_{iJH}\sigma^J\left(\prod_{r=1}^m F_r(\sigma)^{h_r}\right)
\]
where we denote by $h_r$ the components of the $m$-multi-index $H$.

By comparing the $q$-homogeneous parts we get
\[
 F_{iq}(\sigma)=\sum_{(J,H,\Phi)\in\Sigma_{iq}} c_{iJH}\sigma^J\prod_{r=1}^m \prod_{s=1}^{h_{r}} F_{r, \Phi(r,s)}(\sigma)
\]
where the set $\Sigma_{iq}$ consists of triples $(J,H,\Phi)$ in which $J$ is a $n$-multi-index, $H$ is a $m$-multi-index and $\Phi$ is a function that associates to any element $(r,s)$ of the set 
\[\{(r,s):r=1,\ldots,m; s=1,\ldots, h_r\}\] 
a positive (non-zero!) integer $\Phi(r,s)$ such that $\sum \Phi(r,s)=q-|J|$.

If $\Phi(r,s)\geq q$ for some $r$  we see by  definition that $|J|=0$, $|H|=1$ and we know that in this case $c_{i0H}=0$. In particular, we conclude that the right hand side of the formula above involves only $F_{rq'}$'s with $q'<q$. Hence, we can determine the coefficients $d_{iI}$ by induction on $|I|$. 
 Moreover, by construction, each coefficient $d_{iI}$ can be expressed as 
 \begin{equation}
 \label{QiI}
 d_{iI}=Q_{iI}(c_{iJH})
 \end{equation}
 where each $Q_{iI}$ is a polynomial in $c_{iJH}$ for $|J|+|H|\leq |I|$ with coefficients in $\N$.

 We can fix a non-zero topological nilpotent element $\pi$ such  that $||c_{iJK}||\leq|\pi|^{-1}$ for all $i,J,H$. From the argument above, we deduce inductively that each coefficient $d_{iI}$ is a finite sum of products of the form $\prod c_{kJH}$ with $\sum|J|\leq|I|$. In particular, each product has at most $|I|$ factors and hence $||d_{iI}||\leq |\pi|^{-|I|}$. We conclude $||d_{iI}\pi^{2|I|}||\leq|\pi|^{|I|}$ which tends to $0$ as $|I|\ra\infty$.
\end{proof}

The previous statement has an immediate generalization. 

\begin{cor}\label{implicitC}
 Let $R$ be a non-archimedean  Banach  $K$-algebra, let $ {\sigma}=(\sigma_1,\ldots,\sigma_n)$ and $ {\tau}=(\tau_1,\ldots,\tau_m)$ be two systems of coordinates, let $\bar{\sigma}=(\bar{\sigma}_1,\ldots,\bar{\sigma}_n)$ and $ \bar{\tau}=(\bar{\tau}_1,\ldots,\bar{\tau}_m)$ two sequences of elements of $R$ and let $ {P}=(P_1,\ldots,P_m)$ be a collection of polynomials in $R[ {\sigma}, {\tau}]$ such that $ {P}( {\sigma}=\bar{\sigma}, {\tau}=\bar{\tau})=0$ and $\det(\frac{\del P_i}{\del \tau_j})( {\sigma}=\bar{\sigma}, {\tau}=\bar{\tau})\in R^\times$. There exists a unique collection $ {F}=(F_1,\ldots,F_m)$ of $m$ formal power series in $R[[ {\sigma-\bar{\sigma}}]]$ such that $ {F}( {\sigma}=\bar{\sigma})=\bar{\tau}$ and $ {P}( {\sigma}, {F}( {\sigma}))=0$ in $R[[ {\sigma-\bar{\sigma}}]]$ and they have a positive radius of convergence around $\bar{\sigma}$. 
\end{cor}

\begin{proof}
 If we apply Proposition \ref{implicit} to the polynomials $P'_i\colonequals P(\bar{\sigma}+\eta,\bar{\tau}+\theta)$ we obtain an array of formal power series $F'=(F_1,\ldots,F_m')$ in $R[[\eta]]$ with positive radius of convergence  such that $P'(\eta,F'(\eta))=0$. If we now put $\sigma\colonequals \bar{\sigma}+\eta$ and $F\colonequals \bar{\tau}+F'$ we get $P(\sigma,F(\sigma-\bar{\sigma}))=0$ in $R[[\sigma-\bar{\sigma}]]$ as wanted.
\end{proof}

We now assume that $K$ is perfectoid and we come back to the category $\wRigSm^{\gc}$ that we introduced above (see Definition \ref{gc}). We recall that an object $X=\varprojlim_hX_h$ of this category is the pullback over $\widehat{\T}^N\ra\T^N$ of a map $X_0\ra\T^N\times\T^M$ that is a composition of rational embeddings and finite \'etale maps from an affinoid tft adic space $X_0$ to a torus $\T^N\times\T^M=\Spa K\langle\underline{\upsilon}^{\pm1},\underline{\nu}^{\pm1}\rangle$ and $X_h$ denotes the pullback of $X_0$  by $\T^N\langle\underline{\upsilon}^{1/p^h}\rangle\ra\T^N$.

\begin{prop}
\label{roots3}
Let $X=\varprojlim_hX_h$ be an object of $\wRigSm^{\gc}$. If an element $\xi$ of $\mcO^+(X)$ is algebraic and separable over each generic point of $\Spec \mcO(X_0)$ then it lies in $\mcO^+(X_{\bar{h}})$ for some ${\bar{h}}$.
\end{prop}

\begin{proof}
Let $X_0$ be $\Spa (R_0,R_0^\circ)$ let $X_h$ be $\Spa(R_h,R_h^\circ)$ and $X$ be $\Spa(R,R^+)$. For any $h\in\N$  one has  $R_h=R_0\widehat{\otimes}_{K\langle \underline{\upsilon}^{\pm1}\rangle} K\langle \underline{\upsilon}^{\pm1/p^h}\rangle$ and $R^+$ coincides with the  $\pi$-adic completion of $\varinjlim_hR^\circ_h$ by Proposition \ref{fibprod}. 
The proof is divided in several steps.

{\it Step 1}: We can suppose that $R$ is perfectoid. Indeed, we can consider the refined tower   $X_h'=X_0\times_{\T^N\times\T^M}(\T^N\langle\underline{\upsilon}^{1/p^h}\rangle\times\T^M\langle\underline{\nu}^{1/p^h}\rangle)$ whose limit $\widehat{X}$ is perfectoid. If the claim is true for this tower, we conclude that $\xi$ lies in the intersection of $\mcO(X_h')$ and $\mcO(X)$ inside $\mcO(\widehat{X})$ for some $h$. By Remark  $\ref{bigosum}$ this is the intersection \[\left(\widehat{\bigoplus}_{I\in(\Z[1/p]\cap[0,1))^N}R_0\underline{\upsilon}^I\right)\cap\left(\widehat{\bigoplus}_{\substack{I\in\{a/p^h\colon 0\leq a<p^h\}^N\\J\in\{a/p^h\colon 0\leq a<p^h\}^M}}R_0\underline{\upsilon}^I\underline{\nu}^J\right)\]
which coincides with \[\widehat{\bigoplus}_{\substack{I\in\{a/p^h\colon 0\leq a<p^h\}^N}}R_0\underline{\upsilon}^I=R_h.\]

{\it Step 2}: We can always assume that each $R_h$ is an integral domain. Indeed, the number of connected components of $\Spa R_h$ may rise, but it is bounded by the number of connected components of the affinoid perfectoid $X$ which is finite by Remark \ref{fincpt}. 

We deduce that the number of connected components of $\Spa R_h$ stabilizes for $h$ large enough. Up to shifting indices, we can then suppose that $\Spa R_0$ is the finite disjoint union of irreducible rigid varieties $\Spa R_{i0}$ for $i=1,\ldots,k$ such that $R_{ih}=R_{i0}\widehat{\otimes}_{K\langle \underline{\upsilon}^{\pm1}\rangle} K\langle \underline{\upsilon}^{\pm1/p^h}\rangle$ is a domain for all $h$. We denote by $R_i$ the ring $R_{i0}\widehat{\otimes}_{K\langle \underline{\upsilon}^{\pm1}\rangle} K\langle \underline{\upsilon}^{\pm1/p^\infty}\rangle$. Let now $\xi=(\xi_i)$ be an element in $ R^+=\prod R_i^+$ that is separable over $\prod\Frac R_i$ i.e. each $\xi_i$ is separable over $\Frac R_i$. If the proposition holds for $R_i$ we then conclude that $\xi_i$ lies in $R^\circ_{ih}$ for some large enough $h$ so that $\xi\in R_h^\circ$ as claimed.

{\it Step 3}: We prove that we can consider a non-empty rational subspace  $U_0=\Spa R_0\langle f_i/g\rangle$ of $X_0$ instead. Indeed, using Remark $\ref{bigosum}$ if the result holds for $U_0$  assuming $\bar{h}=0$ we deduce that $\xi$ lies in the intersection of $R\cong\widehat{\bigoplus} R_0$ and of $R_0\langle f_i/g\rangle$ inside $R\langle f_i/g\rangle\cong\widehat{\bigoplus}R_0\langle f_i/g\rangle$ which coincides with $R_0$.

{\it Step 4}: We prove that we can assume $\xi$ to be integral over $R_0$. Indeed, let $P_\xi$ be its minimal polynomial over $\Frac(R_0)$. We can suppose there is a common denominator $d$ such that $P_\xi$ has coefficients  in  $R_0[1/d][x]$. By \cite[Proposition 6.2.1/4(ii)]{BGR} we can also assume that $|d|=1$. In particular, by \cite[Proposition 7.2.6/3]{BGR}, 
 the rational subspace associated to $R_0\langle1/d\rangle$ is not empty. By Step 3, we can then restrict to it and assume $\xi$ integral over $R_0$ and $R_0[\xi]\cong R_0[x]/P_\xi(x)$. 

{\it Step 5}: We can  suppose that $P_\xi(x)$ is the minimal polynomial of $\xi$ with respect to all non-empty rational subspaces of $X_h$ for all $h$. If it is not the case, from the previous steps we can rescale indices and restrict to a rational subspace with respect to which the degree of $P_\xi(x)$ is lower. Since the degree is bounded from below, we conclude the claim. 

{\it Step 6}: We now conclude the claim. By Step 3 and Step 4 %
we may suppose that the map $R_0\ra R_0[\xi]$ is finite \'etale. Since this map splits after tensoring with $R$, we can deduce by \cite[Lemma 7.5]{scholze} that it splits already at some level ${h}$. By Step 5, this implies that $P_\xi(x)$ is of degree $1$ and hence $\xi\in R_{0}$. %
\end{proof}

We introduce now the geometric application of Propositions \ref{implicit} and \ref{roots3}. It states that a map  from $\varprojlim_hX_h\in\wRigSm$ to a rigid variety factors, up to $\B^1$-homotopy, over one of the intermediate varieties $X_h$. Analogous statements are widely used in  \cite{ayoub-rig}  (see for example \cite[Theorem 2.2.58]{ayoub-rig}): there, these are obtained as corollaries of  Popescu's theorem (\cite{popescu-85} and \cite{popescu-86}), 
which is not available in our non-noetherian setting.

\begin{prop}
\label{solleva}
Let $X=\varprojlim_hX_h$ be in $\wRigSm^{\gc}$.  
Let  $Y$ be an affinoid rigid variety endowed with an \'etale map $Y\ra\B^n$ and let $f\colon X\ra Y$ be a map of adic spaces. 
\begin{enumerate}
 \item There exist $m$ polynomials $Q_1,\ldots,Q_m$ in $K[ \sigma_1,\ldots,\sigma_n,\tau_1,\ldots,\tau_m]$ such that $Y\cong\Spa A$ with $A\cong K\langle \sigma,\tau\rangle/(Q)$ and $\det(\frac{\del Q_i}{\del \tau_j})\in A^\times$.
\item 
There exists a map $H\colon X\times\B^1\ra Y$ such that $H\circ i_0=f$ and $H\circ i_1$ factors over the canonical map $X\ra X_h$ for some integer $h$.
\end{enumerate}
Moreover, if $f$ is induced by the map $K\langle \sigma,\tau\rangle\ra\mcO(X)$, $\sigma\mapsto s, \tau\mapsto t$ the map $H$ can be defined via
\[
(\sigma,\tau)\mapsto(s+(\tilde{s}-s)\chi,F(s+(\tilde{s}-s)\chi))
\]
where $F$ is the unique array of formal power series in $\mcO(X)[[\sigma-s]]$ associated to the polynomials $P(\sigma,\tau)$ by Corollary \ref{implicitC}, and $\tilde{s}$ is any element in $\varinjlim_h\mcO^+(X_h)$ such that the radius of convergence of $F$ is larger than $||\tilde{s}-s||$ and $F(\tilde{s})$ lies in $\mcO^+(X)$.
\end{prop}

\begin{proof}
The first claim follows  from  \cite[Lemma 1.1.52]{ayoub-rig}. We turn to the second claim. 
Let $X_0$ be $\Spa (R_0,R_0^\circ)$ and $X$ be $\Spa(R,R^+)$. For any $h\in\N$ we denote $R_0\widehat{\otimes}_{K\langle \underline{\upsilon}\rangle} K\langle \underline{\upsilon}^{\pm1/p^h}\rangle$ with $R_h$ so that $R^+$ coincides with the  $\pi$-adic completion of $\varinjlim_hR^\circ_h$ by Proposition \ref{fibprod}.

The map $f$ is determined by the choice of $n$ elements ${s}=(s_{1},\ldots,s_{n})$ and $m$ elements $t=(t_{1},\ldots,t_{m})$ of $R^+$ such that $P(s,t)=0$. We prove that the formula for $H$ provided in the statement defines a map $H$ with the required properties.

By Corollary \ref{implicitC}  there exists a collection $F=(F_{1},\ldots,F_{m})$ of $m$ formal power series in $R[[\sigma-s]]$ with a positive radius of convergence such that $F(s)=t$ and $P(\sigma,F(\sigma))=0$. As $\varinjlim_hR^\circ_h$ is dense in $R^+$ we can find an integer $\bar{h}$ and elements $\tilde{s}_i\in R^\circ_{\bar{h}}$ such that $||\tilde{s}-s||$ is smaller than the convergence radius of $F$. By  renaming the indices, we can assume that $\bar{h}=0$. As $F$ is continuous and $R^+$ is open, we can also assume that the elements $F_{j}(\tilde{s})$ 
 lie in $R^+$.  
 We are left to prove that they actually lie in $\varinjlim_hR_h^\circ$. 
Since the determinant of $(\frac{\del P_{i}}{\del \tau_j})(\tilde{s},F(\tilde{s}))$ is invertible, the field $L\colonequals\Frac(R_0)(F_{1}(\tilde{s}),\ldots,F_m(\tilde{s}))$ is algebraic and separable over $\Frac(R_0)$. 
We can then apply Proposition \ref{roots3} to conclude that each element $F_{j}(\tilde{s})$ lies in $R^\circ_h$ for a sufficiently big integer $h$.
\end{proof}

The goal of the rest of this section is to prove Proposition \ref{solleva2text}. To this aim, we present a generalization of the results above for collections of maps. As before, we start with an algebraic statement and then translate it into a geometrical fact for our specific purposes.

\begin{prop}\label{2dH=02}
Let $R$ be a Banach $K$-algebra and let $\{R_h\}_{h\in\N}$ be a collection of nested complete subrings of $R$ such that $\varinjlim R_h$ is dense in $R$. Let $s_1,\ldots,s_N$ be elements of $R\langle\theta_1,\ldots,\theta_n\rangle$. 
 For any $\varepsilon>0$ there exists an integer $h$ and elements $\tilde{s}_1,\ldots,\tilde{s}_N$ of $R_h\langle\theta_1,\ldots,\theta_n\rangle$
 satisfying the following conditions.
\begin{enumerate}
\item $|s_\alpha-\tilde{s}_\alpha|<\varepsilon$ for each $\alpha$.
\item For any $\alpha,\beta\in\{1,\ldots,N\}$ and any $k\in\{1,\ldots,n\}$ such that $s_\alpha|_{\theta_k=0}=s_\beta|_{\theta_k=0}$ we also have $\tilde{s}_\alpha|_{\theta_k=0}=\tilde{s}_\beta|_{\theta_k=0}$.
\item For any $\alpha,\beta\in\{1,\ldots,N\}$ and any $k\in\{1,\ldots,n\}$ such that $s_\alpha|_{\theta_k=1}=s_\beta|_{\theta_k=1}$ we also have $\tilde{s}_\alpha|_{\theta_k=1}=\tilde{s}_\beta|_{\theta_k=1}$.
\item\label{cond112} For any $\alpha\in\{1,\ldots,N\}$ if $s_\alpha|_{\theta_1=1}\in R_{h'}\langle \underline{\theta}\rangle$ for some $h'$ then $\tilde{s}_\alpha|_{\theta_1=1}={s}_\alpha|_{\theta_1=1}$.
\end{enumerate}
\end{prop}

\begin{proof}
We will actually prove a stronger statement, namely that we can reinforce the previous conditions with the following:
\begin{enumerate}\setcounter{enumi}{4}
 \item\label{condx}  For any $\alpha,\beta\in\{1,\ldots,N\}$ any subset $T$ of $\{1,\ldots,n\}$ and any map $\sigma\colon T\ra\{0,1\}$ such that $s_\alpha|_{\sigma}=s_\beta|_{\sigma}$ then $\tilde{s}_\alpha|_{\sigma}=\tilde{s}_\beta|_{\sigma}$.
 \item For any $\alpha\in\{1,\ldots,N\}$ any subset $T$ of $\{1,\ldots,n\}$ containing $1$ and any map $\sigma\colon T\ra\{0,1\}$ such that $s_\alpha|_{\sigma}\in R_h\langle\underline{\theta}\rangle$ for some $h$ then $\tilde{s}_\alpha|_{\sigma}={s}_\alpha|_{\sigma}$.
\end{enumerate}
Above we denote by $s|_\sigma$ the image of $s$ via the substitution  $(\theta_{t}=\sigma(t))_{t\in T}$. 
We proceed by induction on $N$, the case $N=0$ being trivial. 

Consider the conditions we want to preserve that involve the index $N$. They are of the form
\[
 {s}_i|_\sigma={s}_N|_{\sigma}
\]
and are indexed by some pairs $(\sigma, i)$ where $i$ is an index and  $\sigma$ varies in a set of maps $\Sigma$. Our procedure consists in determining by induction the elements $\tilde{s}_1,\ldots,\tilde{s}_{N-1}$ first, and then deduce the existence of $\tilde{s}_N$ by means of Lemma \ref{chinesecor2b} by lifting the elements $\{\tilde{s}_i|_{\sigma}\}_{(\sigma,i)}$. 
Therefore, we first define $\varepsilon'\colonequals\frac{1}{C}\varepsilon$ where $C=C(\Sigma)$ is the constant introduced in Lemma \ref{chinesecor2b} and then apply the induction hypothesis to the first $N-1$ elements with respect to $\varepsilon'$.

By the induction hypothesis, the elements $\tilde{s}_i|_\sigma$ satisfy the compatibility condition of Lemma \ref{chinesecor2b} and lie in $R_h\langle\underline{\theta}\rangle$ for some integer $h$. Without loss of generality, we assume $h=0$. 
By Lemma \ref{chinesecor2b} we can find an element $\tilde{s}_N$ of $R_h\langle\underline{\theta}\rangle$ lifting them 
 such that $|\tilde{s}_N-s_N|<C\varepsilon'=\varepsilon$ as wanted.
\end{proof}

The following lemmas are used in the proof of the previous proposition.

\begin{lemma}\label{grobner2}
For any normed ring $R$ and any map $\sigma\colon T_\sigma\ra\{0,1\}$ defined on a subset $T_\sigma$ of $\{1,\ldots,n\}$ we denote by $I_\sigma$ the ideal of $R\langle\underline{\theta}\rangle$ generated by $\theta_i-\sigma(i)$ as $i$ varies in $T_\sigma$.
For any finite set $\Sigma$ of such maps and any such map ${\eta}$  one has  $\left(\bigcap_{\sigma\in\Sigma} I_\sigma\right) + I_{\eta}=\bigcap_{\sigma\in\Sigma}(I_{\sigma}+I_{\eta})$.
\end{lemma}

\begin{proof}
We only need to prove the inclusion $\bigcap(I_{\sigma}+I_{\eta})\subseteq\left(\bigcap I_\sigma \right)+ I_{\eta}$. We can make induction on the cardinality of $T_{\eta}$ and restrict to the case in which $T_{\eta}$ is a singleton. By changing variables, we can suppose $T_{{\eta}}=\{1\}$ and ${\eta}(1)=0$ so that $I_{\eta}=(\theta_1)$. 

We first suppose that $1\notin T_\sigma$ for all $\sigma\in\Sigma$. Let $s$ be an element of $\bigcap(I_\sigma+(\theta_1))$. This means we can find elements $s_\sigma\in I_\sigma$ and polynomials $p_\sigma\in R\langle \underline{\theta}\rangle$ such that $s=s_\sigma+p_\sigma\theta_1$. Since $I_\sigma$ is generated by polynomials of the form $\theta_i-\epsilon$ with $i\neq1$ we can suppose that $s_\sigma$ contains no $\theta_1$ by eventually changing $p_\sigma$. Let now $\sigma$, $\sigma'$ be in $\Sigma$. From the equality
\[
 s_\sigma=(s_\sigma+p_\sigma\theta_1)|_{\theta_1=0}=(s_{\sigma'}+p_{\sigma'}\theta_1)|_{\theta_1=0}=s_{\sigma'}
\]
we conclude that $s_\sigma\in\bigcap I_\sigma$. Therefore $s\in\bigcap I_\sigma+(\theta_1)$ as claimed.

We now move to the general case. Suppose $\bar{\sigma}(1)=1$ for some $\bar{\sigma}\in\Sigma$. Then $I_{\bar{\sigma}}+I_{\eta}=R\langle \underline{\theta}\rangle$ and if $f\in\bigcap_{\sigma\neq\bar{\sigma}} I_\sigma$ then $f=-f(\theta_1-1)+f\theta_1\in\bigcap_{\sigma}I_\sigma+(\theta_1)$. Therefore, the contribution of $I_{\bar{\sigma}}$ is trivial on both sides and we can erase it from $\Sigma$. We can therefore suppose that $\sigma(1)=0$ whenever $1\in T_\sigma$.

For any $\sigma\in\Sigma$ let $\sigma'$ be its restriction to $T_\sigma\setminus\{1\}$. We have $I_{\sigma'}\subseteq I_\sigma$ and $I_{\sigma'}+(\theta_1)=I_{\sigma}+(\theta_1)$ for all $\sigma\in\Sigma$. By what we already proved, the statement holds for the set $\Sigma'\colonequals\{\sigma':\sigma\in\Sigma\}$. Therefore:
\[
\bigcap_{\sigma\in\Sigma}(I_\sigma+(\theta_1))=\bigcap_{\sigma'\in\Sigma'}(I_{\sigma'}+(\theta_1))=\bigcap_{\sigma'\in\Sigma'} I_{\sigma'}+(\theta_1)\subseteq\bigcap_{\sigma\in\Sigma} I_\sigma +(\theta_1)                                                            
\]
proving the claim.
\end{proof}

We recall (see \cite[Definition 1.1.9/1]{BGR}) that a morphism of normed groups $\phi\colon G\ra H$ is \emph{strict} if the homomorphism $G/\ker\phi\ra\phi(G)$ is a homeomorphism, where the former group is endowed with the quotient topology and the latter with the topology inherited from $H$. In particular, we say that a sequence of normed $K$-vector spaces
\[
 R\stackrel{f}{\ra} S\stackrel{g}{\ra} T
\]
is \emph{strict and exact} at $S$ if it exact at $S$ and if $f$ is strict i.e. the quotient norm and the norm induced by $S$ on $R/\ker(f)\cong\ker(g)$ are equivalent.

\begin{lemma}\label{chinese22}
For any  map $\sigma\colon T_\sigma\ra\{0,1\}$ defined on a subset $T_\sigma$ of $\{1,\ldots,n\}$ we denote by $I_\sigma$ the ideal of $R\langle\underline{\theta}\rangle=R\langle{\theta}_1\ldots,\theta_n\rangle$ generated by $\theta_i-\sigma(i)$ as $i$ varies in $T_\sigma$. 
For any finite  set $\Sigma$ of such maps and any complete normed $K$-algebra $R$ the following sequence of Banach $K$-algebras is strict and exact 
 \[
  0\ra R \langle \underline{\theta}\rangle/\bigcap_{\sigma\in\Sigma} I_\sigma\ra\prod_{\sigma\in\Sigma} R \langle \underline{\theta}\rangle/I_{\sigma}\ra\prod_{\sigma,\sigma'\in\Sigma} R\langle \underline{\theta}\rangle/(I_\sigma+I_{\sigma'})
 \]
and the ideal $\bigcap_{\sigma\in\Sigma} I_\sigma$ is generated by a finite set of polynomials in the $\theta_i$ with coefficients in $\Z$.
\end{lemma}

\begin{proof}
We follow the notation and the proof of \cite{kleinert}.   
For a collection of ideals $\mcI=\{I_\sigma\}$ we let $A(\mcI)$ be the kernel of the map $\prod_\sigma R \langle \underline{\theta}\rangle/I_\sigma\ra\prod_{\sigma,\sigma'} R\langle \underline{\theta}\rangle/(I_\sigma+I_{\sigma'})$ and $O(\mcI)$ be the cokernel of $R\langle \underline{\theta}\rangle/\bigcap_\sigma I_\sigma\ra A(\mcI)$. We make induction on the cardinality $m$ of $\mcI$. The case $m=1$ is obvious. 

Let $\mcI'$ be $\mcI\cup\{I_{\eta}\}$. From the diagram
$$\xymatrix{
&0\ar[d]\ar[r]&R \langle \underline{\theta}\rangle\ar[d]\ar[r]^{id}&R \langle \underline{\theta}\rangle\ar[d]\ar[r]&0\\
0\ar[r]&W\ar[r]&A(\mcI')\ar[r]&A(\mcI)
}$$
we obtain by the snake lemma the exact sequence
\[
0\ra I_{\eta}\cap\bigcap I_{\sigma}\ra \bigcap I_{\sigma}\ra W\ra O(\mcI')\ra O(\mcI).
\]
By direct computation, it holds $W=\bigcap(I_\sigma+I_{\eta})/I_{\eta}$. By the induction hypothesis, we obtain $O(\mcI)=0$. Moreover, since $\bigcap I_\sigma + I_{\eta}=\bigcap(I_\sigma+I_{\eta})$ by Lemma \ref{grobner2}, we conclude that the map $\bigcap I_{\sigma}\ra W$ is surjective and hence $O(\mcI')=0$ proving the main claim. 

We remark that  the sequence of the statement  is obtained from the sequence with $R=K$, by tensoring with $R$ over $K$ and completing. The latter is a (strict) exact sequence of affinoid algebras in the sense of Tate \cite[Definition 6.1.1/1]{BGR} so we deduce that the sequence is strict for any $R$ by means of  \cite[Proposition 2.1.8/6]{BGR}. We also remark that the sequence can be deduced from the analogous sequence for $\Z[\underline{\theta}]$ in place of $R\langle \underline{\theta}\rangle$, by tensoring with $R$ and completing. In particular,  the ideal $\bigcap_{\sigma\in\Sigma} I_\sigma$ is already defined over $\Z[\underline{\theta}]$ as claimed. %
\end{proof}

Let $\sigma$ and $\sigma'$ be maps defined from two subsets $T_\sigma$ resp. $T_{\sigma'}$ of $\{1,\ldots,n\}$ to $\{0,1\}$. We say that they are \emph{compatible} if $\sigma(i)=\sigma'(i)$ for all $i\in T_\sigma\cap T_{\sigma'}$ and in this case we denote by $(\sigma,\sigma')$ the map from $T_\sigma\cup T_{\sigma'}$ extending them.

\begin{lemma}\label{chinesecor2b}
Let $X=\varprojlim_hX_h$ be an object in $\wRigSm$ and $\Sigma$ a set as in Lemma \ref{chinese22}. We denote $\mcO(X)$ by $R$ and $\mcO(X_h)$ by $R_h$. 
 For any $\sigma\in\Sigma$ let  $\bar{f}_\sigma$ be an element of $R\langle\underline{\theta}\rangle/I_\sigma$ such that $\bar{f}_\sigma|_{(\sigma,\sigma')}=\bar{f}_{\sigma'}|_{(\sigma,\sigma')}$ for any couple $\sigma,\sigma'\in\Sigma$ of compatible maps. 
\begin{enumerate}
	\item There exists an element $f\in R\langle\underline{\theta}\rangle$ such that $f|_\sigma=\bar{f}_\sigma$. 
	\item There exists a constant $C=C(\Sigma)$ such that if for some $g\in R\langle\underline{\theta}\rangle$ one has $|\bar{f}_\sigma-{g}|_\sigma|<\varepsilon$ 
	for all $\sigma$ then the element $f$ can be chosen so that $|f-g|<C\varepsilon$. Moreover, if $\bar{f}_\sigma\in R_0\langle\underline{\theta}\rangle/I_\sigma$ for all $\sigma$ then the element $f$ can be chosen inside $R_h\langle\underline{\theta}\rangle$ for some integer $h$.
\end{enumerate}
\end{lemma}

\begin{proof}
The first claim and the first part of the second are simply a restatement of Lemma \ref{chinese22}, where $C=C(\Sigma)$ is  the constant defining the compatibility $||\cdot||_1\leq C||\cdot||_2$ between the  norm $||\cdot||_1$ on $R\langle\underline{\theta}\rangle/\bigcap I_\sigma$ induced by the quotient and the norm $||\cdot||_2$ induced by the embedding in $\prod R\langle\underline{\theta}\rangle/I_\sigma$. We now turn to the last sentence of the second claim.

We apply Lemma \ref{chinese22} to each $R_h$ and to $R$. We then obtain  exact sequences of Banach spaces:
\[
 0\ra R_h\langle \underline{\theta}\rangle/\bigcap_{\sigma\in\Sigma} I_\alpha\ra\prod_{\sigma\in\Sigma} R_h\langle \underline{\theta}\rangle/I_{\sigma}\ra\prod_{\sigma,\sigma'\in\Sigma} R_h\langle \underline{\theta}\rangle/(I_\sigma+I_{\sigma'})
\]
\[
 0\ra R\langle \underline{\theta}\rangle/\bigcap_{\sigma\in\Sigma} I_\alpha\ra\prod_{\sigma\in\Sigma} R\langle \underline{\theta}\rangle/I_{\sigma}\ra\prod_{\sigma,\sigma'\in\Sigma} R\langle \underline{\theta}\rangle/(I_\sigma+I_{\sigma'})
\]
where all ideals that appear are finitely generated by polynomials with $\Z$-coefficients, depending only on $\Sigma$.

In particular, there exist two lifts of $\{\bar{f}_\sigma\}$: an element $f_1$ of $R_0\langle \underline{\theta}\rangle$ and an element $f_2$ of $R\langle\underline{\theta}\rangle$ such that $|f_2-g|<C\varepsilon$ and their difference lies in $\bigcap I_\sigma$. Hence, we can find coefficients $\gamma_i\in R\langle \underline{\theta}\rangle$ such that $f_1=f_2+\sum_i\gamma_i p_i$ where $\{p_1,\ldots,p_M\}$ are generators of $\bigcap I_\sigma$ which have coefficients in $K$. Let now $\tilde{\gamma}_i$ be elements of $R_h\langle \underline{\theta}\rangle$ with $|\tilde{\gamma}_i-\gamma_i|<C\varepsilon/M|p_i|$. The element $f_3\colonequals f_1-\sum_i\tilde{\gamma}_i p_i$ lies in $\varinjlim_h(R_h\langle \underline{\theta}\rangle)$ is another lift of $\{\bar{f}_\sigma\}$ and satisfies  $|f_3-g|\leq\max \{|f_2-g|,|f_2-f_3|\}< C\varepsilon$ proving the claim.
\end{proof}

We can now finally prove the approximation result that played a crucial role 
 in Section \ref{motapprox}.

\begin{proof}[Proof of Proposition \ref{solleva2text}]
For any $h\in\Z$ we will denote 
$\mcO(X_h)\langle\theta_1,\ldots,\theta_n\rangle$ by $R_h$. We also denote the $\pi$-adic completion of $\varinjlim_hR^\circ_h$ by $R^+$ and $R^+[\pi^{-1}]$ by $R$.

By  Proposition \ref{solleva} we conclude that there exist integers $m$ and $n$ and a $m$-tuple of polynomials $P=(P_1,\ldots, P_m)$ in $K[\sigma,\tau]$ where $\sigma=(\sigma_1,\ldots,\sigma_n)$ and $\tau=(\tau_1\ldots,\tau_m)$ are systems of variables  such that $K\langle \sigma,\tau\rangle/(P)\cong\mcO(Y)$ and each $f_k$ is induced by maps $(\sigma,\tau)\mapsto(s_k,t_k)$ from $K\langle \sigma,\tau\rangle/(P)$ to $R$ for some $m$-tuples $s_k$ and $n$-tuples $t_k$ in $R$. Moreover, there exists a sequence of power series $F_k=(F_{k1},\ldots,F_{km})$ associated to each $f_k$ such that
\[
(\sigma,\tau)\mapsto(s_k+(\tilde{s}_k-s_k)\chi,F_k(s_k+(\tilde{s}_k-s_k)\chi)\in R\langle\chi\rangle\cong\mcO(X\times\B^n\times\B^1)
\]
defines a map $H_k$ satisfying the first claim, for any choice of $\tilde{s}_k\in \varinjlim_hR^\circ_h$ such that $\tilde{s}_k$ is in the convergence radius of $F_k$ and $F_k(\tilde{s}_k)$ is in $R^+$.

Let now $\varepsilon$ be a positive real number, smaller than all radii of convergence of the series $F_{kj}$ and such that $F(a)\in R^+$ for all $|a-s|<\varepsilon$. Denote by $\tilde{s}_{ki}$ the elements associated to $s_{ki}$ by applying Proposition \ref{2dH=02} with respect to the chosen $\varepsilon$. In particular, they induce a well defined map $H_k$  and the elements $\tilde{s}_{ki}$ lie in $R_{\bar{h}}^\circ\langle\theta_1\ldots,\theta_n\rangle$ for some integer $\bar{h}$. We show that the maps $H_k$ induced by this choice also satisfy the second and third claims of the proposition.

Suppose that $f_k\circ d_{r,\epsilon}=f_{k'}\circ d_{r,\epsilon}$ for some $r\in\{1,\ldots,n\}$ and $\epsilon\in\{0,1\}$. This means that $\bar{s}\colonequals s_{k}|_{\theta_r=\epsilon}=s_{k'}|_{\theta_r=\epsilon}$ and $\bar{t}\colonequals t_{k}|_{\theta_r=\epsilon}=t_{k'}|_{\theta_r=\epsilon}$. 
This implies that both $F_k|_{\theta_r=\epsilon}$ and $F_{k'}|_{\theta_r=\epsilon}$ are two $m$-tuples of formal power series $\bar{F}$ with coefficients in $\mcO(X\times\B^{n-1})$ converging around $\bar{s}$ and such that $P(\sigma,\bar{F}(\sigma))=0$, $\bar{F}(\bar{s})=\bar{t}$. By the uniqueness of such power series stated in Corollary \ref{implicitC}, we conclude that they coincide.
 
Moreover, by our choice of the elements $\tilde{s}_k$ it follows that $\bar{\tilde{s}}\colonequals\tilde{s}_{k}|_{\theta_r=\epsilon}=\tilde{s}_{k'}|_{\theta_r=\epsilon}$. In particular  one has  
\[
F_k((\tilde{s}_{k}-s_k)\chi)|_{\theta_r=\epsilon}=\bar{F}((\bar{\tilde{s}}-\bar{s})\chi)=F_{k'}((\tilde{s}_{k'}-s_{k'})\chi)|_{\theta_r=\epsilon}
\]
and therefore $H_k\circ d_{r,\epsilon}=H_{k'}\circ d_{r,\epsilon}$ proving the second claim.

The third claim follows immediately since the elements $\tilde{s}_{ki}$ satisfy the condition (\ref{cond112}) of Proposition \ref{2dH=02}.
\end{proof}

\end{appendix}

\bibliographystyle{alpha}

 \end{document}